\documentclass[a4paper, 11pt, leqno]{amsart} 
\usepackage[utf8]{inputenc}
\usepackage[T1]{fontenc}	

\usepackage{lmodern}
\usepackage[normalem]{ulem}

\usepackage{newtxtext}

\usepackage[plainpages=false,colorlinks,linkcolor=bleuFonce,citecolor=rougeFonce,urlcolor=vertFonce,breaklinks]{hyperref}

\usepackage{graphicx} 

\usepackage{color}
\definecolor{vertFonce}	{rgb}{0,0.5,0}
\definecolor{numLignes}	{rgb}{0.17,0.57,0.7}	
\definecolor{gris}		{rgb}{0.5,0.5,0.5}
\definecolor{grisFonce}	{rgb}{0.2,0.2,0.2}
\definecolor{orange}	{rgb}{1,0.65,0.31}		
\definecolor{orangeFonce}{rgb}{1,0.4,0}
\definecolor{bleuFonce}	{rgb}{0,0,0.4}
\definecolor{rougeFonce}{rgb}{0.3,0,0}
\definecolor{rougeWord}	{rgb}{0.5,0,0}
\definecolor{vertClair}	{rgb}{0.8,1,0.8}
\definecolor{rougeClair}{rgb}{1,0.5,0.5}
\definecolor{violet}	{rgb}{0.5,0,0.5}

\usepackage{amsmath, amssymb, amscd, amsthm, amsfonts}
\usepackage{mathtools}
\usepackage{bbm} 
\usepackage{amsaddr}
\usepackage{braket}
\usepackage{mathrsfs}

\newtheorem{theorem}{Theorem}[section]
\newtheorem{lemma}[theorem]{Lemma}
\newtheorem{proposition}[theorem]{Proposition}
\newtheorem{remark}[theorem]{Remark}
\newtheorem{definition}[theorem]{Definition}
\newtheorem{corollary}[theorem]{Corollary}

\newcounter{Hcounter}

\newenvironment{HypothesisList}{%
	\setcounter{Hcounter}{0} 
	\begin{enumerate}
	}{%
	\end{enumerate}
}

\newcommand{\Hitem}{\stepcounter{Hcounter}\item}

\newcommand{\N}{\mathbb{N}}
\newcommand{\rr}{\mathbb{R}}

\newcommand		{\lt}			{\left}				%
\newcommand		{\rt}			{\right}			%
\renewcommand	{\(}			{\lt(}
\renewcommand	{\)}			{\rt)}
\newcommand		{\lal}			{\langle}			%
\newcommand		{\ral}			{\rangle}			%
\newcommand		{\bangle}[1]	{\lt\lal #1\rt\ral}
\newcommand		{\weight}[1]	{\lt\lal #1\rt\ral}	

\newcommand		{\inprod}[2]	{\bangle{#1, #2}}

\newcommand		{\cB}		{\mathcal B}
\newcommand		{\cC}		{\mathcal C}

\newcommand		{\cH}		{\mathcal H}

\newcommand		{\cS}		{\mathcal S}

\newcommand		{\cL}		{\mathcal L}		
\newcommand		{\cW}		{\mathcal W}

\newcommand		\sfB		{\mathsf B}

\newcommand		\sfD		{\mathsf D}

\newcommand		\sfK		{\mathsf K}
\newcommand		\sfM		{\mathsf M}

\newcommand		\sfT		{\mathsf T}			

\newcommand		\sfX		{\mathsf X}			

\newcommand		\sfm 		{\mathsf m}

\newcommand		{\intd}			{\int_{\rr^d}}
\newcommand		{\iintd}		{\iint_{\rr^d\times\rr^d}}

\newcommand{\norm}[1]{\left\lVert #1 \right\rVert}
\newcommand		{\nrm}[1]		{\lt\lVert #1\rt\rVert}

\newcommand		{\Nrm}[2]		{\nrm{#1}_{#2}}
\newcommand		{\n}[1]			{\lt\lvert #1 \rt\rvert}
\newcommand{\stnorm}[2]{\left\lVert #2 \right\rVert_{\mathcal{L}^{#1}}}
\newcommand{\cstnorm}[2]{\left\lVert #2 \right\rVert_{#1}}
\newcommand{\znorm}[1]{\left\lVert #1 \right\rVert_{\mathcal{Z}^{A}_1}}
\newcommand{\ynorm}[1]{\left\lVert #1 \right\rVert_{\mathcal{Y}^{A}_1}}
\newcommand{\lnorm}[2]{\left\lVert #2 \right\rVert_{L^{#1}}}
\newcommand{\snorm}[1]{\left| #1 \right|}

\renewcommand		{\d}		{\mathrm{d}}		
\newcommand			{\dd}		{\,\d}		
\newcommand         {\grad}     {\nabla}
\newcommand			{\bd}		{\partial}	
\newcommand			{\Id}		{\mathrm{Id}}		

\newcommand		{\opp}		{\boldsymbol{p}}

\DeclareMathOperator{\tr}		{Tr}				

\newcommand		{\Tr}[1]		{\tr\!\( #1 \)}	

\newcommand		{\Dh}		{\boldsymbol{\nabla}}	

\newcommand     {\lb}       {\left(}
\newcommand     {\rb}       {\right)}

\newcommand{\OPAh}{\operatorname{Op}^{A}_h}
\newcommand{\OPA}{\operatorname{Op}^{A}}
\newcommand{\bB}{B}
\newcommand{\bA}{A}
\newcommand{\bp}{\boldsymbol{p}}
\newcommand{\bXr}{\mathsf{X}_{\boldsymbol{\rho}}}
\newcommand{\bX}{\mathsf{X}}
\newcommand{\bbxx}{\boldsymbol{\mathcal{X}}_{\brho}}
\newcommand{\brho}{\boldsymbol{\rho}}
\newcommand{\brhoAf}{\boldsymbol{\rho}^{A}_f}
\newcommand{\bxhf}{\mathsf{X}_{\brho}}

\newcommand{\mU}{\mathcal{U}(t,s)}

\newcommand{\mG}{K}
\newcommand{\bT}{\mathsf{T}}

\newcommand{\ma}{\sfM_{A}}
\newcommand{\za}{\mathcal{Z}^{A}_1}
\newcommand{\ya}{\mathcal{Y}^{A}_1}
\newcommand{\ee}{\operatorname{e}}

\newcommand{\wmu}{\widetilde{\mu}}
\newcommand{\wnu}{\widetilde{\nu}}

\title[Semiclassical limit of the Magnetic Hartree--Fock Equation]{A Gauge-Covariant Semiclassical Limit from the magnetic Hartree--Fock equation to the magnetic Vlasov--Poisson equation}

\author[\textsc{J. Chong}]{\vspace{-10pt}\textsc{Jacky J. Chong}}
\address[J. Chong]{\small\vspace{-10pt}Department of Mathematics and Statistics, Beijing Institute of Technology, Beijing, China}
\email{jwchong@bit.edu.cn}
\author[\textsc{Z. You}]{\vspace{-10pt}\textsc{Zhuohui You}}
\address[Z. You]{\small\vspace{-10pt}Laboratory of Mathematics and Complex Systems, Ministry of Education, School of Mathematical Sciences, Beijing Normal University, Beijing, China}
\email{zhuohui@mail.bnu.edu.cn}



\begin{document}
	
	\begin{abstract}
    We study the semiclassical limit from the magnetic Hartree--Fock equation to the magnetic Vlasov--Poisson equation in three spatial dimensions in the presence of a nonuniform, purely space-dependent magnetic field. We first establish propagation of regularity for solutions of the magnetic Vlasov--Poisson equation under exponentially weighted assumptions. Building on this result and a gauge-covariant magnetic Weyl quantization framework, we then derive explicit convergence bounds for the semiclassical limit in semiclassical Schatten norms. In particular, the estimates are formulated in terms of the magnetic field rather than a particular choice of vector potential and are therefore gauge covariant. The resulting semiclassical convergence holds for a suitable class of initial data on any fixed finite time interval.
	\end{abstract}
	
	\maketitle
	
	\renewcommand{\contentsname}{\centerline{Table of Contents}}
	\setcounter{tocdepth}{2}

	\section{Introduction}
	\subsection{Background and Context}
	\subsubsection*{The magnetic Vlasov--Poisson system}
	We consider properties of solutions to the Vlasov–Poisson system with an external space‑time dependent magnetic field $B=B(t,x)$, that is, we consider a solution $f(t, x, v)$ (with $x\in \rr^d$, $v \in \rr^d$, $t\ge 0$) to the Cauchy problem
	\begin{equation}\label{equation: MagneticVlasovSys}
		\begin{cases}
			\partial_t f + v\cdot\nabla_x f + (E+B(t,x)v)\cdot \nabla_v f = 0 \\[.5em]
			f(0,x,v) = f^{\mathrm{in}}(x,v) \ge 0
		\end{cases}
	\end{equation}
	with 
	\begin{equation*}
		E(t, x)=-\grad_x V_f,
	\end{equation*}
	where
	\begin{equation*}
		-\Delta V_{f} = \rho_f \quad \text{ and } \quad \rho_f(t, x)= \intd f(t, x, v)\,\mathrm{d} v\,.
	\end{equation*}
	The matrix $B$ is anti-symmetric, and the initial datum $f^{\mathrm{in}}$ is assumed to be nonnegative, integrable, and bounded on the kinetic phase space $\Xi:=\rr^d_x\times\rr^d_v$. Physically, $f(t,x,v)$ describes the particle density distribution on $\Xi$. For particles of unit mass, the velocity $v$ coincides with the kinetic momentum. The self-consistent potential is given by $V_f=K*\rho_f$. In three dimensions, $K$ is the Coulomb potential $K(x)=1/(4\pi|x|)$.

The existence theory for kinetic equations in the presence of magnetic
fields is classical. Global weak solutions for the Vlasov--Maxwell
system were established by DiPerna and Lions
\cite{diperna1989vlasovmaxwell}, and the corresponding renormalized
solution framework also applies to closely related magnetic transport
problems. The Vlasov--Poisson system under a strong external magnetic
field has been studied in various asymptotic regimes; see, for example,
\cite{frenod1998homogenization,golse1999vlasov,golse2003vlasov}. The
effect of an external magnetic field on Landau damping has also been
investigated recently; see \cite{charles2021magnetized}.

As in the non-magnetic case, classical well-posedness and the
propagation of regularity and moments have also been studied in the
presence of an external magnetic field. Knopf
\cite{knopf2018optimal} proved global classical well-posedness for
smooth compactly supported data and sufficiently regular bounded
magnetic fields. On every finite interval $[0,T]$, the solution remains
supported in a phase-space ball whose radius may depend on $T$; hence
all moments remain finite on finite time intervals. In two dimensions,
Bostan \cite{bostan2019asymptotic} gave a concise global-existence
argument for a spatially dependent magnetic field.

For noncompactly supported data, less is known about moment and
regularity propagation. Rege first treated the case of a uniform
magnetic field and also established propagation of regularity
\cite{rege2021vlasov}. This analysis was later extended to spatially
uniform, time-dependent magnetic fields
\cite{rege2023propagation}. Both works build on the classical methods
of Lions--Perthame \cite{lions1991propagation} and Pfaffelmoser
\cite{pfaffelmoser1992global}. Stability estimates for magnetized
Vlasov equations with nonuniform magnetic fields were subsequently
obtained in \cite{rege2025stability}. More recently,
\cite{porat2025propagation} studied moment and regularity propagation
for space-time-dependent magnetic fields, treating the Coulomb
interaction in two dimensions and a screened Coulomb interaction in
three dimensions, under suitable assumptions on the magnetic field.

\subsubsection*{The Hartree--Fock equation and its magnetic counterpart}

The Hartree equation
\begin{equation}
	i\hbar\,\partial_t\brho
	=
	\left[H_{\brho}^{\rm red},\brho\right]
	\quad\text{with}\quad
	\brho(0)=\brho^{\rm in}
\end{equation}
describes the evolution of density operators $\brho=\brho(t)$. These
operators are positive, self-adjoint, and trace class on $L^2(\rr^d)$,
and
\begin{equation*}
	[A,B]:=AB-BA
\end{equation*}
denotes the usual commutator. Here $h$ is Planck's constant and
$\hbar=h/(2\pi)$ is the reduced Planck constant. We impose the
normalization constraints
\begin{equation}\label{def:normalization}
	h^d\Tr{\brho}=1
	\quad\text{and}\quad
	\Nrm{\brho}{\infty}=\cC_\infty,
\end{equation}
where $\Nrm{\cdot}{\infty}$ denotes the operator norm and
$\cC_\infty$ is independent of $\hbar$. The Hartree, or reduced
Hartree--Fock, Hamiltonian is given by
\begin{equation*}
	H_{\brho}^{\rm red}
	=
	-\tfrac12\hbar^2\Delta+V_{\brho},
\end{equation*}
where $V_{\brho}$ is the multiplication operator by the mean-field
potential
\begin{equation*}
	V_{\brho}(x)=(K*\rho)(x),
\end{equation*}
and the quantum spatial density is defined as the scaled restriction
of the integral kernel of $\brho$ to the diagonal,
\begin{equation*}
	\rho(x):=h^d\brho(x,x).
\end{equation*}
Here and throughout the paper, we use the same notation for the
operator $\brho$ and its integral kernel $\brho(x,y)$. The Hartree
equation may be regarded as the quantum counterpart of the Vlasov
equation and provides a natural bridge between quantum and classical
mean-field dynamics.

If the Hartree equation is used to model a system of particles obeying
Fermi statistics, a more accurate description, taking into account
the anti-symmetry of the underlying many-body wave function, is given
by the Hartree--Fock equation
\begin{equation}\label{eq:HF}
	i\hbar\,\partial_t\brho
	=
	[H_{\brho},\brho],
\end{equation}
where
\begin{equation*}
	H_{\brho}
	=
	-\tfrac12\hbar^2\Delta+K*\rho_{\rm HF}-h^d\bxhf.
\end{equation*}
The exchange operator $\bxhf$ has integral kernel
\begin{equation*}
	\bxhf(x,y)
	=
	K(x-y)\brho(x,y).
\end{equation*}

To take into account the magnetic field, we consider the magnetic
Hartree--Fock equation given by \eqref{eq:HF} with effective
Hamiltonian
\begin{equation}\label{eq:magnetic_Hamiltonian}
	H^A_{\brho}
	=
	H^A_0+V_{\brho}-h^d\sfX_{\brho},
	\qquad
	H^A_0
	=
	\tfrac12(\opp-A)^2,
\end{equation}
where $A$ is a vector potential corresponding to the external magnetic
field and
\begin{equation*}
	\opp=-i\hbar\grad.
\end{equation*}
Compared with the non-magnetic case, the mathematical theory is less
developed. Well-posedness and related questions for magnetic Hartree
equations have been studied, for example, in
\cite{michelangeli2015global,dong2021hartree}. Mean-field derivations
in the presence of magnetic fields were obtained for bosonic Hartree
dynamics in \cite{luhrmann2012mean}. More recently, effective
Hartree--Fock dynamics for fermionic systems in magnetic fields has
been investigated in
\cite{ferrero2024effective,benedikter2025magnetic}.

The mathematical literature on Hartree and Hartree--Fock dynamics is
extensive; for an introduction and further references, see
\cite{benedikter2016effective}. Results for bosonic Hartree-type
dynamics can be found, for instance, in
\cite{bardos2002derivation,erdos2001derivation,rodnianski2009quantum,
grillakis2010second, grillakis2011second, grillakis2013pair}, while fermionic Hartree and Hartree--Fock
dynamics have been studied in
\cite{bach2016kinetic,benedikter2014mean,benedikter2016mixed,
petrat2016new,saffirio2017mean}.

\subsubsection*{Semiclassical limits}
The derivation of the Vlasov equation from quantum many-body dynamics
in a combined mean-field and semiclassical scaling goes back to
Narnhofer and Sewell \cite{narnhofer1981vlasov} and Spohn
\cite{spohn1981vlasov}; see, e.g.,
\cite{benedikter2014mean,benedikter2016mixed,petrat2016new,
chen2021combined,chen2022convergence,chong2024many}
for more recent results on effective fermionic dynamics in related
mean-field and semiclassical regimes. In particular, these works considered a scaling
in which the large-particle-number limit is coupled to the classical
limit. Early
results on the semiclassical limit established convergence in weak
topologies under various assumptions; see, for example, Lions and Paul
\cite{lions1993mesures} and Markowich and Mauser
\cite{markowich1993classical}. Stronger and quantitative forms of
semiclassical convergence were subsequently obtained in
\cite{athanassoulis2011strong,amour2013semiclassical,
benedikter2016hartree}.

More recently, considerable attention has been devoted to singular
interactions, and in particular to the Coulomb potential. Strong
semiclassical convergence for singular inverse-power interactions was
established in \cite{saffirio2020inverse}, while the Coulomb case was
treated in trace norm for suitable mixed states in
\cite{saffirio2020hartree}. Further results on propagation of
semiclassical regularity and quantitative convergence for the
Hartree and Hartree--Fock equations can be found in
\cite{lafleche2019propagation,lafleche2021global,chong2022global,
chong2023l2,lafleche2023strong,chong2024many}. In particular,
\cite{lafleche2023strong} establishes strong semiclassical convergence
from both Hartree and Hartree--Fock dynamics to the Vlasov--Poisson
equation in Schatten norms for singular interactions including the
Coulomb potential.

Semiclassical limits in the presence of external magnetic fields have
also been studied. In \cite{benporat2024magnetic}, the magnetic
Liouville equation with a nonconstant external magnetic field was
derived from quantum dynamics by means of a semiclassical transport
distance. The semiclassical limit from the Pauli--Poisson equation to
the magnetic Vlasov--Poisson equation was established in
\cite{moller2025pauli}. In the related setting of a dynamical
electromagnetic field, the semiclassical limit from the fermionic
Maxwell--Schr\"odinger system to the Vlasov--Maxwell system was
established in \cite{LeopoldSaffirio2026}. This result is formulated
in the Coulomb gauge and treats extended charges with a regularized
Coulomb interaction. Together with \cite{Leopold2024}, it also
provides an effective description of the corresponding spinless
Pauli--Fierz dynamics by the Vlasov--Maxwell system. Related magnetic
semiclassical limits leading to fluid equations have also recently
been considered in \cite{benporat2026magnetized}.

A natural tool for studying magnetic semiclassical limits is the
magnetic Weyl calculus. A gauge-covariant magnetic Weyl quantization
and the associated symbolic calculus were developed in
\cite{mantoiu2004magnetic,iftimie2007magnetic}; semiclassical
asymptotic expansions for the magnetic Weyl product were further
studied in \cite{lein2010two}. The use of this calculus is particularly
convenient in the present setting because it allows the magnetic field,
rather than a particular choice of vector potential, to enter the
semiclassical expansion in a gauge-covariant manner.

In this work, we study the semiclassical limit from the magnetic
Hartree--Fock equation to the magnetic Vlasov--Poisson equation in
three dimensions using magnetic Weyl quantization. As $h\to0$, we
establish convergence in Schatten norms on arbitrary fixed finite time
intervals. Our result combines a nonuniform external magnetic field,
the three-dimensional Coulomb interaction, and the Hartree--Fock
exchange term within a strong semiclassical convergence result. The
closest non-magnetic analogue is the strong Hartree--Fock to
Vlasov--Poisson limit of \cite{lafleche2023strong}, whereas
\cite{benporat2024magnetic,moller2025pauli} provide closely related
magnetic semiclassical limits in different quantum settings and
topologies. A key ingredient in our argument is the propagation of
moments and regularity for solutions of the Cauchy problem
\eqref{equation: MagneticVlasovSys}.

\bigskip 
\noindent The paper is organized as follows.
\begin{itemize}
	\item We conclude this section by introducing the main ingredients needed for the semiclassical analysis and by stating our principal results.
	\item In Section \ref{Section:ProofofPropagationofRegularity}, we prove Theorem \ref{Theorem:Propagation-of-Moments} and Theorem \ref{Theorem:PropagationOfRegularity}. The estimates obtained there will be used in the proof of the semiclassical limit in Section \ref{Section:ProofofSemiclassicalLimit}.
	\item In Section \ref{Section:ProofOfRegularityOfMagneticWeyl}, we establish regularity properties of the magnetic Weyl quantization. These may be viewed as magnetic analogues of the corresponding non-magnetic results in \cite{lafleche2023strong}.
	\item In Section \ref{Section:ProofofSemiclassicalLimit}, we prove Theorem \ref{Theorem:Semiclassicallimit} using the propagation of moments and regularity established in Section \ref{Section:ProofofPropagationofRegularity}. The arguments for the magnetic Hartree and magnetic Hartree--Fock equations are largely parallel, with the additional task of controlling the exchange term in the Hartree--Fock case.
	\item In Appendix \ref{Section:PropertiesofMagneticWeylTransform}, we collect several calculations and properties of the magnetic Weyl quantization. While related results are available in the literature, we formulate the versions needed here under finite Sobolev regularity assumptions and with explicit dependence on the semiclassical parameter.
	\item In Appendix \ref{Section:ExistenceResults}, we prove global existence for the magnetic Hartree--Fock equation by an energy method. We first establish local existence in a suitable Banach space and then extend the solution globally using a priori bounds in the associated energy space.
    \item Finally, in Appendix~\ref{Section:classical_ExistenceResults}, we give an argument for the global wellposedness of classical solutions to the magnetic Vlasov--Poisson equation with a space-time-dependent magnetic field and noncompactly supported initial data.
\end{itemize}

	\subsection{Function spaces and the magnetic Weyl quantization}
	Let us fix some notation prior to stating our main results. An element of the phase space $\Xi=\rr^d\times\rr^d$ is given by $z = (x, v)$. We denote the classical Lebesgue space associated to $\Xi$ by $L^p(\Xi)$. Moreover, we also define the weighted Sobolev spaces as follows: we adopt the standard notation
	$$\langle x\rangle:=\sqrt{1+\snorm{x}^2}\quad \text{ and } \quad w_\lambda(v) = e^{\lambda  \langle v\rangle},$$
	then, for every $n, k \in \N$, we define the weighted (kinetic) Sobolev space $W^{n,p}_k(\Xi)$ to the subspace of $L^p(\Xi)$ equipped with the norm
	\begin{equation*}
		\norm{f}_{W^{n,p}_k(\Xi)}:=\Big(\sum_{|\alpha|\le n} \norm{{\langle z\rangle }^k \partial_z^\alpha f(z)}_{L^p_z}^p\Big)^{1/p}\,\quad \text{ with }\quad  
		\renewcommand\arraystretch{1.5}
		\begin{array}{l}
			\alpha \in {\N_{0}^{2d}},\, |\alpha|= \sum^{2d}_{i=1} \alpha_{i},
			\\
			\text{and } \partial_z^\alpha = \partial_{x_1}^{\alpha_1}\cdots \partial_{v_d}^{\alpha_{2d}}
		\end{array}.
	\end{equation*}
    We also write $H^n_k(\Xi):=W^{n,2}_k(\Xi)$. In the case of exponential velocity weights, we write  
    \begin{equation*}
        \norm{f}_{\mathcal{X}^{n,p}_{\mu, \lambda}(\Xi)}=\Big(\sum_{|\alpha|\le n} \norm{w_{\mu}(x) w_{\lambda}(v) \partial_z^\alpha f(z)}_{L^p_z}^p\Big)^{1/p}\,.
    \end{equation*}
    Here, we adopt the usual modification when $p=\infty$. 
    We also use 
    \begin{equation*}
        D^{\beta, \gamma} f := \partial_x^\beta \partial_v^\gamma f \quad \text{ for }\quad \beta, \gamma \in \N_0^d\,.
    \end{equation*}

	We define the space of positive trace-class operators by
	\begin{align}
		\mathcal{L}^1_{+}:=\left\{\brho\in \cB(L^2(\rr^d)),\ \brho=\brho^{*}\ge 0,\ \operatorname{Tr}(\brho)<\infty\right\}\,,
	\end{align}
	where $\cB(L^2(\rr^d))$ denotes the space of bounded linear operators acting on $L^2(\rr^d)$. We denote the semiclassical Schatten-$p$ norm by
	\begin{align}
		\stnorm{p}{\brho}=h^{\frac{d}{p}}\norm{\brho}_p=h^{\frac{d}{p}}\left(\operatorname{Tr}(\snorm{\brho}^p)\right)^{\frac{1}{p}}\,,
	\end{align}
	where $\norm{\cdot}_p$ denotes the usual Schatten $p$-norm.

We regard the magnetic field $B$ as a closed anti-symmetric two-form on
$\mathbb{R}^d$. Since every closed two-form on $\mathbb{R}^d$ is exact,
there exists a one-form vector potential $A$ such that $\mathrm{d}A=B$;
in components,
$B_{ij}=\partial_{x_i}A_j-\partial_{x_j}A_i$.
The choice of $A$ is highly non-unique: for any sufficiently regular
function $\varphi$, the one-form $\widetilde A=A+\mathrm{d}\varphi$
generates the same magnetic field. Moreover, a vector potential may
have worse behavior at spatial infinity than the corresponding
magnetic field. We therefore impose regularity assumptions primarily
on $B$ and only local regularity on $A$.

\begin{HypothesisList}
	\Hitem \label{Assumption: bounded B}
	For some sufficiently large finite $m\in\N$,
	\begin{equation*}
		B\in W^{m,\infty}(\rr^d).
	\end{equation*}

	\Hitem \label{Assumption: vector potential A}
	The vector potential satisfies
	\begin{equation*}
		A\in W_{\mathrm{loc}}^{1,\infty}(\rr^d;\rr^d),
		\qquad
		\d A=B.
	\end{equation*}
\end{HypothesisList}

	\bigskip
	
	Let us now introduce the magnetic Weyl quantization (MWQ). 
	\begin{definition}
		Assume that the magnetic field $B$ and vector potential $A$ satisfy \ref{Assumption: bounded B} and \ref{Assumption: vector potential A}, respectively. For $f\in \mathcal{S}(\Xi)$ and $h>0$, we define the associated semiclassical magnetic Weyl quantized operator acting on $u \in L^2(\rr^d)$ by
		\begin{subequations}\label{def:magnetic_Weyl_quantization}
			\begin{equation}
				\OPAh(f)u:=\iintd e^{-2\pi i\,v\cdot(y-x)}e^{-\frac{2\pi i}{h}\,\Gamma^{A}([x,y])}f(\tfrac{x+y}{2},hv)u(y)\,\mathrm{d} y\mathrm{d}v\,.
			\end{equation}
			Here, $\Gamma^{A}([x,y]):=\int_{[x,y]}A$
			is the integral of the one-form $A$ along the line segment $[x,y]$. When $h=1$, we simply denote $\OPA_1(f)=\OPA(f)$. Moreover, we also denote the integral kernel of the magnetic Weyl quantized operator of $f$ by
			\begin{equation}
				\brho^{A}_f(x,y):=\int_{\rr^d}e^{-2\pi i\,v\cdot(y-x)}e^{-\frac{2\pi i}{h}\,\Gamma^{A}([x,y])}f(\tfrac{x+y}{2},hv)\,\mathrm{d} v\,,
			\end{equation}
		\end{subequations}
		and adopt the notation $\brho^{A}_f:=\OPAh(f)$.
	\end{definition}
	
	It is also convenient to consider the quantum analog of the above weighted (kinetic) Sobolev spaces (at least for $n = 1$). More precisely, we define the quantum gradients as follows:
	\begin{align*}
		\boldsymbol{\nabla}_x\brho:=[\nabla,\brho^{A}_f]
			-
			\frac{2\pi i}{h}
			\left(\brho^A_{Af}-\brhoAf A\right)\quad \text{ and } \quad \boldsymbol{\nabla}_v\brho:=\left[\frac{x}{i\hbar},\brho\right]\,.
	\end{align*}
	Moreover, let us denote $\opp :=-i\hbar \grad$ and $\sfm_n := \weight{\opp}^n$, defined as a multiplication operator with the symbol $\weight{v}^n$ in Fourier space with $\n{\opp}^2=-\hbar^2\Delta$. Then the semiclassical analogues of the kinetic homogeneous Sobolev norms are defined by
	\begin{equation}
		\begin{aligned}
			\Nrm{\brho}{\dot{\cW}^{1,p}_n}^p &:= \Nrm{\Dh_{\!x}\brho\, \sfm_n}{\cL^p}^p+\Nrm{\Dh_{\!v}\brho\, \sfm_n}{\cL^p}^p
			\\
			\Nrm{\brho}{\dot{\cW}^{1,\infty}_n} &:= \Nrm{\Dh_{\!x}\brho\, \sfm_n}{\cL^\infty}+\Nrm{\Dh_{\!v}\brho\, \sfm_n}{\cL^\infty},
		\end{aligned}
	\end{equation}
	and the inhomogeneous version by
	\begin{align}
		\Nrm{\brho}{\cW^{1,p}_n}^p &:= \Nrm{\brho\,\sfm_n}{\cL^p}^p + \Nrm{\brho}{\dot{\cW}^{1,p}_n}^p,
	\end{align}
	with the usual modification when $p=\infty$. Notice that for $p=2$, we have $\Nrm{\brho}{\cW^{1,2}} = \Nrm{f_{\brho}}{H^1}$ which is suggestive of the strong connection between $\cW^{1, p}_n$ and the classical kinetic Sobolev spaces.

The motivation for this definition is partly provided by the following identities. For any $f\in\mathcal{S}(\Xi)$ and $u\in\mathcal{S}(\mathbb{R}^d)$, by direct calculation, we have 
\[
\left[\frac{x}{i\hbar},\brhoAf\right]u=\brho^{A}_{\nabla_v f}u\,,
\]
while, more covariantly, we have
\[
	\left[
		\nabla-\frac{i}{\hbar}\bA ,
		\brho_f^A
	\right]
	=
	\brho_{\mathcal{D}_x^B f}^A
	+
	O(h^2)\,,
\]
where 
\[
	\mathcal{D}_x^B f
	:=
	\nabla_x f-\bB(x)\nabla_v f
\]
is the magnetic phase-space derivative. 
Thus, differentiation of the symbol with respect to the position and velocity variables corresponds, under the magnetic Weyl quantization, to commutators with $\nabla$ and $x/(i\hbar)$, respectively.
	
	\bigskip 
	
	Next, we introduce the magnetic version of the Poisson bracket. 
	
	\begin{definition}[Magnetic Poisson bracket]
		Let $B$ satisfy \ref{Assumption: bounded B}. For functions $f, g\in \cS(\Xi)$, the magnetic Poisson bracket of $f$ and $g$ associated with $B$ is defined as follows:
		\begin{equation}
			\left\{f,g\right\}_{B}=\sum^d_{j=1}\left(\partial_{x_j}f\partial_{v_j}g-\partial_{x_j}g\partial_{v_j}f\right)+\sum^d_{k,j=1}B_{kj}\partial_{v_k}f\partial_{v_j}g\,.
		\end{equation}
	\end{definition}
	By applying this, we can rewrite the equation \eqref{equation: MagneticVlasovSys} as
	\begin{equation}\label{equation:ModifiedMagneticVlasov}
		\partial_tf=\left\{\tfrac{1}{2}\snorm{v}^2+K*\rho_f,f\right\}_{B}\,.
	\end{equation}
	Moreover, when we consider equation \eqref{equation: MagneticVlasovSys} in the case of three dimensions, if we suppose that 
	\begin{equation*}
		B=
		\begin{pmatrix}
			0 & B_3 & -B_2 \\
			-B_3 & 0 & B_1 \\
			B_2 & -B_1 & 0
		\end{pmatrix}\,,
	\end{equation*}
	then the term $Bv$ has the following vector-field representation:
	\begin{equation*}
		\begin{aligned}
			B v = v \wedge\big( B_1,\; B_2,\; B_3\big)\,.
		\end{aligned}
	\end{equation*}
	Hence, if we regard the vector $\big( B_{1}(t,x),\; B_{2}(t,x),\; B_{3}(t,x) \big)$ above as the same as $B$, then $B(t,x)v\cdot \nabla_v f$ in \eqref{equation: MagneticVlasovSys} can be expressed as $v\wedge B\cdot\nabla_v f$. Since these two representations are equivalent, one can discuss the propagation of moments and regularity for \eqref{equation: MagneticVlasovSys} in three dimensions in the latter form.
	
	\subsection{Main results}
	We first present the results on the propagation of moments for the solution of \eqref{equation: MagneticVlasovSys}, which is crucial for proving the results of propagation of regularity.

    \begin{theorem}[Propagation of moments]
    \label{Theorem:Propagation-of-Moments}
    Let $d=3$, $\Xi=\rr^3\times\rr^3$, and $n\geq 3$, and fix a finite time interval $[0,T]$. Assume that the magnetic field $B=B(t,x)$ satisfies Assumption~\ref{Assumption: bounded B} and
    \[
    B\in C\bigl([0,T];L^\infty(\rr^3)\bigr).
    \]
    Let $f\geq 0$ be a classical solution to \eqref{equation: MagneticVlasovSys} on $[0,T]$ with initial datum
    \[
    f^{\rm in}\in L^1(\Xi)\cap L^\infty(\Xi),
    \]
    and assume that
    \[
    M_n(0)<\infty,
    \]
    where
    \[
    M_n(t)
    :=
    \iint_{\Xi}\snorm{v}^n f(t,x,v)\,\dd x\d v.
    \]
    Then
    \begin{equation}
            \rho_f\in L^2\bigl([0,T];L^2(\rr^3)\bigr)
    \qquad\text{and}\qquad
    E\in L^2\bigl([0,T];L^6(\rr^3)\bigr).
    \end{equation}
    Moreover, there exists a constant
    \[
    C=C(f^{\rm in},n,B)>0,
    \]
    depending only on the initial datum, $n$, and the magnetic field $B$, such that
    \begin{equation}
        \sup_{0\leq t\leq T} M_n(t)
        \leq
        M_n(0)
        \exp\!\left\{C(f^{\rm in},n,B)(1+T)\right\}.
    \end{equation}
    \end{theorem}
	
\begin{remark}
	In fact, to establish the propagation of higher-order regularity, it is sufficient to propagate a velocity moment of any order strictly greater than $6$. Such a moment yields the control of $\norm{E}_{L^\infty}$ needed to close the higher-order derivative estimates without derivative loss.

	One possible approach to the propagation of velocity moments is to adapt the method used in \cite{rege2023propagation}, which is related to the characteristic method introduced by Pfaffelmoser. For a general space--time dependent magnetic field, however, difficulties arise when estimating quantities analogous to $\snorm{v-V^*(t)}$, since the magnetic field alters the geometry of the characteristics. Similarly, a direct adaptation of the Lions--Perthame method \cite{lions1991propagation} appears difficult in this setting, in part because the presence of a nonuniform space--time dependent magnetic field prevents the use of the explicit representations available in the nonmagnetic case.

	Instead, our proof of propagation of moments follows the recent approach of \cite{wang2026critical}, combining standard kinetic interpolation estimates with the compensated-integrability method.
\end{remark}

	Next we present results on the propagation of regularity for the solution of \eqref{equation: MagneticVlasovSys}. Unlike the usual Vlasov--Poisson equations, when the propagation of regularity of solution of \eqref{equation: MagneticVlasovSys} the weight is in the form of exponential term rather than polynomial due to the difficulties caused by magnetic field.
    \begin{theorem}[Propagation of regularity]
    \label{Theorem:PropagationOfRegularity}
    Let $d=3$, $\Xi=\rr^3\times\rr^3$, $m\geq2$.
    Fix $T>0$ and assume that
    \begin{equation}
    B(t, x)\in C\bigl([0,T];W^{m,\infty}(\rr^3)\bigr).
    \end{equation}
    Let $f$ be a nonnegative classical solution of
    \eqref{equation: MagneticVlasovSys} on $[0,T]$, with initial datum
    $f^{\mathrm{in}}$ satisfying
    \begin{equation}\label{controlcondition}
    f^{\mathrm{in}}\in\mathcal{X}^{m,1}_{\mu_0,\lambda_0}\cap\mathcal{X}^{m,\infty}_{\mu_0,\lambda_0},
    \qquad
    \mu_0>0,
    \qquad
    \lambda_0>0,
    \end{equation}
    where
    \begin{equation}
    \norm{h}_{\mathcal{X}^{m,p}_{\mu,\lambda}}
    =
    \left(
    \sum_{|\alpha|+|\beta|\leq m}
    \norm{
    e^{\mu\langle x\rangle+\lambda\langle v\rangle}
    D^{\alpha,\beta}h
    }_{L^p(\Xi)}^p
    \right)^{1/p},
    \end{equation}
    with the usual modification when $p=\infty$.
    
    Let $C_{m}^{\rm mag}>0$ be the constant that appears in
 Proposition~\ref{prop:velocity_commutator}, and define
    \begin{align}\label{eq:choice-lambda}
    \mu(t):=&\, \mu_0 e^{-\kappa t} \quad \text{ and }\\
        \lambda(t) :=&\, \lambda_0 -\frac{\mu_0}{\kappa}(1-e^{-\kappa t})- C^{\rm mag}_{m}\,\|\grad_x B\|_{L^\infty([0, T]; W^{m-1, \infty}_x)}\, t\,,
    \end{align}
    for some $\kappa>0$.
    Assume that
    \begin{equation}
    \lambda_*:=\inf_{t\in[0,T]}\lambda(t)>0\,.
    \end{equation}
    Then, for all $1\le p \le \infty$, we have that
    \begin{equation}
    f\in L^\infty\bigl([0,T];\mathcal{X}^{m,p}_{\mu(t),\lambda(t)}\bigr)
    \end{equation}
    and
    \begin{equation}\label{eq:propagation-regularity}
    \sup_{0\leq t\leq T}
    \norm{f(t)}_{\mathcal{X}^{m,p}_{\mu(t),\lambda(t)}}
    \leq
    C_T,
    \end{equation}
    where $C_T$ depends only on
    \begin{equation}
    m,\quad p,\quad T,\quad \lambda_0,\quad \lambda_*,
    \quad
    \norm{B}_{L^\infty([0,T];W^{m,\infty})},
    \quad
    \norm{f^{\mathrm{in}}}_{L^1},
    \quad\text{and}\quad
    \norm{f^{\mathrm{in}}}_{\mathcal{X}^{m,1}_{\mu_0,\lambda_0}\cap\mathcal{X}^{m,\infty}_{\mu_0,\lambda_0}}.
    \end{equation}
    \end{theorem}

	\begin{remark}
		For the case of a constant magnetic field, combining the ideas from \cite{lafleche2023strong} and \cite{rege2023propagation}, one can obtain the results of propagation of moments and regularity under the condition that the given initial data has polynomial decay. However, for a general magnetic field, polynomial decay is not sufficient to obtain the results of propagation of regularity. Moreover, even if the initial data are Schwartz functions, the propagation of regularity is not guaranteed. To see this, we consider the corresponding DiPerna--Lions flow of \eqref{equation: MagneticVlasovSys} in $3$-dimensional space under the condition that the electric field is zero:
		\begin{equation}\label{eq:DiPerna-Lions flow}
			\begin{cases}
				\dot{X}(t,x,v)=V(t,x,v)\,,\\[.5em]
				\dot{V}(t,x,v)=V(t,x,v)\wedge B(t,X(t,x,v))\,.
			\end{cases}
		\end{equation}
		We set $B(t,x)=(0,0,\epsilon\tanh(x_2))$, where $\epsilon>0$ is a positive real number and $x=(x_1,x_2,x_3)\in\mathbb{R}^3$. We set $X(0)=(0,\delta,0)$ and $V(0)=(R,0,0)$, where $\delta$ is a sufficiently small positive real number (this initial data can be regarded as an operation of perturbation around the initial data $X(0)=0$ and $V(0)=(R,0,0)$). We also set $R>0$. Then the coordinate $V_2$ satisfies
		\begin{equation*}
			\dot V_2=R\epsilon\tanh(X_2).
		\end{equation*}
		For convenience we do the linear approximation of this ODE. That is to say
		\begin{equation*}
			\ddot X_2=\dot V_2=R\epsilon \tanh(X_2)\approx R\epsilon X_2,
			\qquad X_2(0)=\delta,\quad \dot X_2(0)=0.
		\end{equation*}
		Then we obtain $X_2(t)=\delta\cosh(\sqrt{R\epsilon}\,t)$, and $X_1(t)=Rt$, $X_3(t)=0$, and consequently
		\[
		\frac{\partial X_2(t)}{\partial x_2}
		=\frac{\partial(\delta\cosh(\sqrt{R\epsilon}\,t))}{\partial\delta}
		=\cosh(\sqrt{R\epsilon}\,t),
		\]
		and
		\[
		\frac{\partial X_1(t)}{\partial x_2}=\frac{\partial X_3(t)}{\partial x_2}=0.
		\]
		
		Now we take a nonnegative smooth bump function $\chi$ in $\mathbb{R}^3$ supported near the origin. Set
		\[
		\psi(z)=\exp(-|z|^a)(1-\chi(z)),\qquad 0<a<\frac12,
		\]
		and initial data $f^{\mathrm{in}}(x,v)=\psi(x)\psi(v)$. Therefore
		\begin{equation*}
			\partial_{x_2} f(t,x,v)
			=\cosh(\sqrt{R\epsilon}\,t)\left(\partial_{x_2} f^{\mathrm{in}}\right)(X(-t),V(-t))
			=\cosh(\sqrt{v_1\epsilon}\,t)\left(\partial_{x_2} f^{\mathrm{in}}\right)(X(-t),V(-t)).
		\end{equation*}
		Taking the $L^1$ integral one obtains
		\begin{equation*}
			\begin{aligned}
				\iint_{\Xi}
				\snorm{\partial_{x_2} f(t,x,v)}\dd x\,\d v
				&=
				\iint_{\Xi}
				\snorm{\cosh(\sqrt{v_1\epsilon}\,t)}
				\snorm{\left(\partial_{x_2} f^{\mathrm{in}}\right)(X^{-1}(t),V^{-1}(t))}\dd x\,\d v\\
				&=
				\iint_{\Xi}
				\snorm{\cosh(\sqrt{v_1\epsilon}\,t)}
				\snorm{\partial_{x_2} f^{\mathrm{in}}(x,v)}\dd x\d v
				=+\infty.
			\end{aligned}
		\end{equation*}
		This gives a heuristic argument of the reason that we choose exponential weight.
	\end{remark}
	
With the aid of Theorem \ref{Theorem:PropagationOfRegularity}, we obtain the following semiclassical limit from the magnetic Hartree--Fock equation to the Vlasov--Poisson equation \eqref{equation: MagneticVlasovSys} in three dimensions, under the assumption that the vector potential $\bA$ is purely space-dependent. This assumption is needed to establish global wellposedness for solutions of the magnetic Hartree--Fock equation.

We also remark that, in the Coulomb case, the proof strategy is largely parallel to that of \cite{lafleche2023strong}. However, some regularity indices must be recalculated because the magnetic Calder\'on--Vaillancourt theorem requires different regularity assumptions from its non-magnetic counterpart. In addition, implementing this strategy requires new regularity results for the magnetic Weyl quantization; see Section \ref{Section:ProofOfRegularityOfMagneticWeyl}.

	The energy of solution $\brho$ of the magnetic Hartree--Fock equation is given by
	\begin{equation*}
		\begin{aligned}
			\mathcal{E}_{\hbar}(\brho)=&h^3\operatorname{Tr}(H^{A}_{0}\brho)+\frac{h^6}{2}\iint_{\rr^3\times\rr^3}K(x-y)\(\brho(x,x)\brho(y,y)-\snorm{\brho(x,y)}^2\)\dd x\d y\,.
		\end{aligned}
	\end{equation*}
	The kinetic part of this energy is $\mathcal{E}_{\mathrm{kin},\hbar}(\brho):=h^3\operatorname{Tr}\left(H^A_{0}\brho\right)$, which we assume throughout our work to be uniformly bounded for $h \in (0, 1]$. It turns out that the energy of the solution can be controlled by the kinetic energy of the initial data, whether for the magnetic Hartree–Fock equation (see Appendix \ref{Section:ExistenceResults}). Consequently, we have the following semiclassical limit results:

\begin{theorem}[Semiclassical limit]
\label{Theorem:Semiclassicallimit}
Let $d=3$, $\Xi=\rr^3\times\rr^3$, and
$K(x)=(4\pi|x|)^{-1}$. Fix $T>0$. Assume that the vector potential $A=A(x)$ and the magnetic field $B = \d A$ are time-independent and satisfy Assumptions~\ref{Assumption: vector potential A}--\ref{Assumption: bounded B}, respectively. 

Let $f$ be a nonnegative classical solution of
\eqref{equation: MagneticVlasovSys} whose initial datum satisfies, for
some sufficiently large $m\in\mathbb{N}$, $\mu_0>0$ and $\lambda_0>0$,
\[
    f^{\rm in}
    \in
    \mathcal{X}^{m,1}_{\mu_0,\lambda_0}
    \cap
    \mathcal{X}^{m,\infty}_{\mu_0,\lambda_0}\,.
\]
Assume that the hypotheses of
Theorem~\ref{Theorem:PropagationOfRegularity} hold on $[0,T]$, including
\[
    \lambda_*:=\inf_{t\in[0,T]}\lambda(t)>0,
\]
and that
\[
    f\in
    L^\infty([0, T]; 
		H^{\sigma}_{\sigma}(\Xi)\cap W^{\delta,\infty}(\Xi))\,.
\]

Where $\sigma$ and $\delta$ are sufficiently large integers. Let $\brho^{\rm in}\in\mathcal{L}^1_+$ satisfy
$\mathcal{E}_{\mathrm{kin},\hbar}
    (\brho^{\rm in})<\infty$,
and let $\brho$ be the corresponding positive solution of the magnetic
Hartree--Fock equation on $[0,T]$. Then, for every $t\in[0,T]$,
\begin{equation}
    \stnorm{1}{\brho(t)-\brho_f^A(t)}
    \leq
    \left(
        \stnorm{1}{\brho^{\rm in}-\brho_{f^{\rm in}}^A}
        +C_1(t)\hbar
    \right)
    \exp\bigl(E_1(t)\bigr).
\end{equation}
For every $1\leq p<\frac32$,
\begin{equation}
    \stnorm{p}{\brho(t)-\brho_f^A(t)}
    \leq
    \left(
        \stnorm{p}{\brho^{\rm in}-\brho_{f^{\rm in}}^A}
        +C_2(t)\hbar
    \right)
    E_2(t).
\end{equation}
Moreover, for every $r\in[p,\infty)$,
\begin{multline}
    \stnorm{r}{\brho(t)-\brho_f^A(t)}
    \leq
    \left(
        \stnorm{\infty}{\brho^{\rm in}}
        +
        \norm{f}_{L^\infty([0, T]; 
		H^{17}_{16}(\Xi)\cap W^{27,\infty}(\Xi))}
    \right)^{1-\frac{p}{r}}
    \\
    {}\times
    \left[
        \left(
            \stnorm{p}{\brho^{\rm in}-\brho_{f^{\rm in}}^A}
            +C_2(t)\hbar
        \right)
        E_2(t)
    \right]^{\frac{p}{r}}.
\end{multline}
Here, $C_1$, $E_1$, $C_2$, and $E_2$ are nonnegative nondecreasing
functions on $[0,T]$. They depend only on $T$, on the regularity of $\bB$, on the relevant initial spatial moments, and on
\[
    \mathcal{E}_{\mathrm{kin},\hbar}(\brho^{\rm in}),
    \qquad
    \stnorm{\infty}{\brho^{\rm in}},
    \qquad
    \norm{f}_{L^\infty([0, T]; 
		H^{\sigma}_{\sigma}(\Xi)\cap W^{\delta,\infty}(\Xi))}.
\]
\end{theorem}
\begin{remark}[Gauge covariance]
	Although the proof is carried out in the transversal gauge (see Section~\ref{Section:ProofOfRegularityOfMagneticWeyl_2}), the
	result is gauge covariant. Indeed, by Proposition~\ref{prop:A1-gauge-covariance}, if
	$\widetilde{\bA}=\bA+\d \varphi$ and
	$U_\varphi=e^{i\varphi/\hbar}$, then
	\[
		\brho_f^{\widetilde A}
		=
		U_\varphi\brho_f^A U_\varphi^*.
	\]
	Moreover, if $\brho$ solves the magnetic Hartree--Fock equation
	associated with $\bA$, then
	$U_\varphi\brho U_\varphi^*$ solves the corresponding equation associated
	with $\widetilde{\bA}$. Since the semiclassical Schatten norms are
	unitarily invariant, the estimates of Theorem
	\ref{Theorem:Semiclassicallimit} are independent of the chosen
	vector potential and depend only on $\bB$ and its relevant
	seminorms.
\end{remark}
\begin{remark}
    From the proof in Section~\ref{Section:ProofofSemiclassicalLimit} we know that we can choose $\sigma=17$ and $\delta=27$. The regularity indices are primarily influenced by the application of the magnetic Calder\'{o}n--Vaillancourt inequality.
\end{remark}

\subsection*{Acknowledgements}

The first author would like to thank Chiara Saffirio and Laurent
Lafl\`eche for many helpful discussions on the topics considered in
this work. The work of J.~Chong was partially supported by the National
Key R\&D Program of China (Project No.~2024YFA1015500).

The authors also acknowledge the use of ChatGPT (GPT-5.6 Sol, OpenAI)
during the preparation of this manuscript. It was used primarily for
language editing, improving the presentation of the exposition, and
assisting with literature searches. In particular, ChatGPT brought to
the authors' attention a paper by Zhaopeng Wang that proved important
in the development of the propagation-of-moments argument used in this
work. All mathematical statements, proofs, calculations, and citations
appearing in the final manuscript were independently checked and
verified by the authors, who take full responsibility for the content
of the paper.

	\section{Propagation of Moments and Regularity}\label{Section:ProofofPropagationofRegularity}
	The aim of this section is to prove Theorem \ref{Theorem:Propagation-of-Moments} and Theorem \ref{Theorem:PropagationOfRegularity}. Throughout this section, we set $d=3$ and $\Xi=\rr^3\times\rr^3$. Before starting the proof, we present several useful preliminary results. The following lemmas and theorems are stated with constants of explicit dependence, since these constants are needed to determine the solution of a certain ordinary differential equation.
	
	\bigskip
	
	\subsection{Proof of propagation of moments}
The proof in this part draws on Serre's compensated-integrability method, developed in \cite{serre2018divergence}, and its recent application by Wang \cite{wang2026critical} to obtain a critical spacetime estimate for the Vlasov--Poisson system. Compensated integrability may be viewed as closely related in spirit to the interaction Morawetz method used in the study of nonlinear Schr\"odinger equations; see, for instance, \cite{colliander2004global,colliander2009tensor}. Before beginning the proof, we collect some necessary ingredients below.
    	
	\bigskip 
    
	We first recall the following well-known inequality.
\begin{lemma}[Kinetic interpolation inequality
\cite{lions1991propagation,golse2013mean}]
\label{lem:kinetic_interpolation}
For each $m\ge \ell\ge0$, there exists a constant
$C(d,m,\ell)>0$ such that, for every nonnegative measurable
function $f=f(x,v)$,
\begin{equation}\label{est:kinetic_interpolation}
	\norm{\int_{\rr^d}\snorm{v}^{\ell}f(\cdot,v)\dd v}
		_{L^{\frac{m+d}{\ell+d}}(\rr^d)}
	\le
	C(d,m,\ell)
	\norm{f}_{L^\infty(\Xi)}^{\frac{m-\ell}{m+d}}
	\norm{\snorm{v}^{m}f}_{L^1(\Xi)}^{\frac{\ell+d}{m+d}}.
\end{equation}
More generally, for $n>k\ge0$ and $1<q\le\infty$,
\begin{equation}
	\norm{\int_{\rr^d}\snorm{v}^{k}f(\cdot,v)\dd v}
		_{L^{p_{n,k}}(\rr^d)}
	\le
	C(d,n,k,q)
	\norm{\snorm{v}^{n}f}_{L^1(\Xi)}^{1-\theta}
	\norm{f}_{L^q(\Xi)}^\theta,
\end{equation}
where $q'$ is the Hölder conjugate of $q$,
$p_{n,k}=\frac{nq'+d}{n(q'-1)+k+d}$, and
$\theta=\frac{q'(n-k)}{nq'+d}$.
\end{lemma}
	
	Next, we state the compensated-integrability result.
	
    \begin{proposition}[Euclidean compensated integrability \cite{serre2018divergence}]
    \label{prop:serre-euclidean}
    Let $A \in L^1(\mathbb{R}^N; \mathrm{Sym}_N^+)$, where $\mathrm{Sym}_N^+$ denotes the cone of all real symmetric positive semidefinite $N\times N$ matrices. Assume that $\operatorname{div} A \in \mathcal{M}(\mathbb{R}^N; \mathbb{R}^N)$, the space of $\rr^N$-valued Radon measures.
    Then $(\det A)^{1/N} \in L^{N/(N-1)}(\mathbb{R}^N)$ and
    \begin{equation}
            \int_{\mathbb{R}^N} (\det A)^{1/(N-1)}\,\mathrm{d}z
    \leq C_N \| \operatorname{div} A \|_{\mathcal{M}(\mathbb{R}^N)}^{N/(N-1)}.
    \end{equation}
    \end{proposition}
	
	We will also need the following lemma, whose proof can be found, e.g. in \cite[Lemma 3.3]{wang2026critical}.
	\begin{lemma}[Kinetic determinant lower bound \cite{wang2026critical}]\label{Lemma:kinetic-determinant-lower-bound}
		Let $0 \le g \in L^1(\mathbb{R}^3) \cap L^\infty(\mathbb{R}^3)$ and $|v|^2 g \in L^1(\mathbb{R}^3)$.
		Set
		\[
		m := \int_{\mathbb{R}^3} g(v)\,\mathrm{d}v,
		\qquad
		A_g := \int_{\mathbb{R}^3} \begin{pmatrix} 1 \\ v \end{pmatrix} \otimes \begin{pmatrix} 1 \\ v \end{pmatrix} g(v)\,\mathrm{d}v.
		\]
		If $m=0$, then $g=0$ almost everywhere and $\det A_g=0$.
		If $m>0$, then
		\[
		\det A_g \ge c\, \|g\|_\infty^{-2} \, m^6,
		\]
		where $c>0$ is a universal constant.
	\end{lemma}

    Applying both Proposition~\ref{prop:serre-euclidean} and Lemma~\ref{Lemma:kinetic-determinant-lower-bound}, we can prove the following key result. 

    \begin{proposition}\label{prop:Control-of-density-L2L2}
    Under the assumptions of Theorem~\ref{Theorem:Propagation-of-Moments}, the density $\rho_f$ satisfies
    \begin{equation}
    	\norm{\rho_f}_{L^{2}([0,T];L^2(\rr^3))}
    	\le
    	C(f^{\rm in},n,B)(1+T)^{\frac{1}{2}}\,.
    \end{equation}
    Moreover, by the Hardy--Littlewood--Sobolev inequality, we have
    \begin{equation}
    	\norm{E}_{L^2([0,T];L^6(\rr^3))}
    	\le
    	C\norm{\rho_f}_{L^2([0,T];L^2(\rr^3))}
    	\le
    	C(f^{\rm in},n,B)(1+T)^{\frac{1}{2}}\,.
    \end{equation}
    \end{proposition}

\begin{proof}
	In the following proof, we use $C(f^{\mathrm{in}},n,B)$ to denote a constant depending only on $B$, $f^{\mathrm{in}}$, and $n$, whose value may change from line to line. We follow the approach of \cite{wang2026critical}. Define
	\begin{equation*}
		\rho_f(t,x)=\int_{\rr^3}f(t,x,v)\dd v\,,
		\qquad
		j_f(t,x)=\int_{\rr^3}v f(t,x,v)\dd v\,.
	\end{equation*}
	Integrating equation~\eqref{equation: MagneticVlasovSys} with respect to $v$ yields the conservation laws
	\begin{equation}\label{eq:Equations-of-Conservation-Law}
		\begin{cases}
			\partial_t\rho_f+{\rm div}_x j_f=0\,,\\[.5em]
			\partial_t j_f+{\rm div}_x\Pi=\rho_fE+j_f\times B\,,
		\end{cases}
	\end{equation}
	where
	\begin{equation*}
		\Pi(t,x)=\int_{\rr^3}v\otimes v\,f(t,x,v)\dd v\,.
	\end{equation*}
	We further define the space--time kinetic tensor $A_f(t,x)$ by
	\begin{equation*}
		A_f(t,x)
		=
		\int_{\rr^3}
		\begin{pmatrix}
			1\\
			v
		\end{pmatrix}
		\otimes
		\begin{pmatrix}
			1\\
			v
		\end{pmatrix}
		f(t,x,v)\dd v\,.
	\end{equation*}
	As in \cite{wang2026critical}, we obtain
	\begin{equation}\label{eq:Divergence-of-kinetic-tensor}
		{\rm div}_{t,x}A_f(t,x)
		=
		\mathbf{1}_{[0,T]}
		\begin{pmatrix}
			0\\
			\rho_fE+j_f\times B
		\end{pmatrix}
		+\delta_{t=0}
		\begin{pmatrix}
			\rho_f(0)\\
			j_f(0)
		\end{pmatrix}
		-\delta_{t=T}
		\begin{pmatrix}
			\rho_f(T)\\
			j_f(T)
		\end{pmatrix},
	\end{equation}
	where $\delta_{t=a}$ denotes the Dirac measure concentrated at $t=a$.

	Applying Lemma~\ref{Lemma:kinetic-determinant-lower-bound} with
	$g(v)=f(t,x,v)$, we obtain
	\begin{equation}\label{Equation:estimate-of-rhof}
		\snorm{\rho_f(t,x)}^6
		\le
		C\,(\det A_f(t,x))
		\norm{f^{\rm in}}_{L^\infty}^2\,.
	\end{equation}
	We first observe that the terms appearing in the divergence of $A_f$
	are uniformly bounded in time. By the kinetic interpolation inequality
	and the Hardy--Littlewood--Sobolev inequality,
	\begin{equation*}
		\begin{aligned}
			\lnorm{1}{\rho_fE}
			&\leq
			\lnorm{\frac{15}{11}}{\rho_f}
			\lnorm{\frac{15}{4}}{E}\\
			&\leq
			C\lnorm{1}{\rho_f}^{\frac13}
			\lnorm{\frac53}{\rho_f}^{\frac53}\\
			&\leq
			C\lnorm{1}{f^{\rm in}}^{\frac13}
			\lnorm{\infty}{f^{\rm in}}^{\frac23}
			M_2(0)
			\leq
			C(f^{\rm in},n,B)\,.
		\end{aligned}
	\end{equation*}
	Similarly,
	\begin{equation*}
		\lnorm{1}{j_f\times B}
		\leq
		\lnorm{\infty}{B}\,
		\lnorm{1}{f^{\rm in}}^{\frac12}
		M_2(0)^{\frac12}
		\leq
		C(f^{\rm in},n,B)\,.
	\end{equation*}
	Moreover, by conservation of mass and energy,
	\begin{equation*}
		\norm{\rho_f(t)}_{L^1}
		+
		\norm{j_f(t)}_{L^1}
		\leq
		C(f^{\rm in},n,B)
	\end{equation*}
	uniformly for $t\in[0,T]$.

	We now apply the compensated integrability estimate locally in time.
	Let
	\[
		I=[s,s+\ell]\subset[0,T],
		\qquad
		0<\ell\leq1,
	\]
	and consider the zero extension
	\[
		\widetilde A_f:=\mathbf{1}_I A_f.
	\]
	By \eqref{eq:Equations-of-Conservation-Law}, we have
	\begin{equation*}
		{\rm div}_{t,x}\widetilde A_f
		=
		\mathbf{1}_{I}
		\begin{pmatrix}
			0\\
			\rho_fE+j_f\times B
		\end{pmatrix}
		+\delta_{t=s}
		\begin{pmatrix}
			\rho_f(s)\\
			j_f(s)
		\end{pmatrix}
		-\delta_{t=s+\ell}
		\begin{pmatrix}
			\rho_f(s+\ell)\\
			j_f(s+\ell)
		\end{pmatrix}.
	\end{equation*}
	Therefore, since $\ell\leq1$, the preceding estimates imply
	\begin{equation*}
		\|{\rm div}_{t,x}\widetilde A_f\|_{\mathcal{M}}
		\leq
		C(f^{\rm in},n,B),
	\end{equation*}
	where the constant is independent of $s$ and $\ell$.

	By Proposition~\ref{prop:serre-euclidean}, we obtain
	\begin{equation*}
		\int_I\int_{\rr^3}
		(\det A_f)^{1/3}\dd x\d t
		\leq
		C
		\|{\rm div}_{t,x}\widetilde A_f\|_{\mathcal{M}}^{4/3}
		\leq
		C(f^{\rm in},n,B)\,.
	\end{equation*}
	Combining this estimate with
	\eqref{Equation:estimate-of-rhof}, we conclude that
	\begin{equation}\label{eq:local-rho-L2}
		\int_I\norm{\rho_f(t)}_{L^2}^2\d t
		\leq
		C(f^{\rm in},n,B)
	\end{equation}
	for every interval $I\subset[0,T]$ of length at most one.

	Finally, partition $[0,T]$ into at most $1+T$ intervals of length at
	most one. Summing \eqref{eq:local-rho-L2} over these intervals completes the proof.
\end{proof}
    As an immediate consequence of Proposition~\ref{prop:Control-of-density-L2L2}, we can now prove Theorem~\ref{Theorem:Propagation-of-Moments}.

	\begin{proof}[Proof of Theorem~\ref{Theorem:Propagation-of-Moments}]

Let $F=\|f^{\mathrm{in}}\|_{L^\infty}$, and set
\[
s_n=\frac{n+3}{3},\qquad
a_n=\frac{3(n+3)}{n+6},\qquad
p_n=\frac{n+3}{n+2}.
\]
By kinetic interpolation,
\[
\|\rho_f\|_{L^{s_n}}
\le C_nF^{n/(n+3)}M_n^{3/(n+3)}.
\]
Since
\[
\frac1{a_n}=\frac23\frac12+\frac13\frac1{s_n},
\qquad
\frac1{n+3}=\frac1{a_n}-\frac13,
\]
interpolation and the Hardy--Littlewood--Sobolev inequality give
\[
\|E\|_{L^{n+3}}
\le
C_n\|\rho_f\|_{L^2}^{2/3}
F^{n/[3(n+3)]}M_n^{1/(n+3)}.
\]
The same estimate remains valid for $n=3$, when $a_n=s_n=2$.

Kinetic interpolation also gives
\[
\left\|\int_{\mathbb R^3}|v|^{n-1}f\,dv\right\|_{L^{p_n}}
\le
C_nF^{1/(n+3)}M_n^{(n+2)/(n+3)}.
\]
Since $p_n$ and $n+3$ are conjugate and
$v\cdot(v\wedge B)=0$, we obtain
\[
\frac{d}{dt}M_n(t)
\le
C_nF^{1/3}\|\rho_f(t)\|_{L^2}^{2/3}M_n(t).
\]
Therefore,
\[
M_n(t)
\le
M_n(0)\exp\left(
C_nF^{1/3}
\int_0^t\|\rho_f(s)\|_{L^2}^{2/3}\dd s
\right).
\]
By Hölder's inequality and Proposition~2.4, we have
\[
\int_0^T\|\rho_f(s)\|_{L^2}^{2/3}\dd s
\le
T^{2/3}
\left(\int_0^T\|\rho_f(s)\|_{L^2}^2\dd s\right)^{1/3}
\le C(f^{\mathrm{in}},n,B)(1+T).
\]
This completes the proof. 
\end{proof}

\subsection{Estimates for spatial density \texorpdfstring{$\rho_f$}{rhof} and electric field \texorpdfstring{$E$}{E}}
	Let us recall some standard estimates.

    \begin{lemma}
\label{Lemma:EstimateofSingularPotential1}
Let $f\geq0$ and suppose that
\[
	f\in L^1(\Xi)\cap L^\infty(\Xi),
	\qquad
	M_m:=\iint_{\Xi}\snorm{v}^m f(x,v)\dd x\dd v<\infty
\]
for some $m>6$. 
Then
\[
	\rho_f\in L^{\frac{m+3}{3}}(\rr^3)
\]
and
\begin{equation}
	\norm{\rho_f}_{L^{\frac{m+3}{3}}}
	\leq
	C_m
	\norm{f}_{L^\infty}^{\frac{m}{m+3}}
	M_m^{\frac{3}{m+3}}.
\end{equation}
In particular, since $(m+3)/3>3$, one has $E\in L^\infty(\rr^3)$ and
\begin{equation}
	\norm{E}_{L^\infty}
	\leq
	C_m
	\norm{f}_{L^1}^{\frac{m-6}{3m}}
	\norm{f}_{L^\infty}^{\frac23}
	M_m^{\frac2m} \le C_m\exp\(C(f^{\rm in}, m, B) (1+T)\)\,.
\end{equation}
Consequently, if $M_m(t)$ remains bounded on every finite time interval,
then so does $\norm{E(t)}_{L^\infty}$.
\end{lemma}

\begin{lemma}[Propagation of pointwise velocity tails]
\label{lem:propagation-velocity-tail}
Let $f\geq0$ be a classical solution to
\eqref{equation: MagneticVlasovSys} on $[0,T]$ and assume that
\[
	E\in L^1\bigl([0,T];L^\infty(\rr^3)\bigr).
\]
Suppose that, for some $k>0$,
\[
	\norm{\langle v\rangle^k f^{\rm in}}_{L^\infty(\Xi)}<\infty.
\]
Then
\begin{equation}
	\sup_{0\leq t\leq T}
	\norm{\langle v\rangle^k f(t)}_{L^\infty(\Xi)}
	\leq
	C_k
	\left(
		1+\int_0^T\norm{E(s)}_{L^\infty}\d s
	\right)^k
	\norm{\langle v\rangle^k f^{\rm in}}_{L^\infty(\Xi)}.
\end{equation}
In particular, polynomial velocity decay is propagated on every finite
time interval.
\end{lemma}

\begin{proof}
	Let $(X(s),V(s))$ be a characteristic associated with
	\eqref{equation: MagneticVlasovSys}
    and given by the DiPerna--Lions flow of
\eqref{equation: MagneticVlasovSys} as follows:
\begin{equation}\label{eq:DiPerna-Lions-flow2}
    \begin{cases}
        \dot{X}(s,x,v)=V(s,x,v),\\[.5em]
        \dot{V}(s,x,v)
        =
        E(s,X(s,x,v))
        +
        V(s,x,v)\wedge B(s,X(s,x,v)),
    \end{cases}
\end{equation}
    Since $V(s)\cdot\bigl(V(s)\times B(s,X(s))\bigr)=0$, 
	we have
	\[
		\frac{\d}{\d s}|V(s)|
		\leq
		\norm{E(s)}_{L^\infty}.
	\]
	Hence, 
	\begin{equation}\label{eq:velocity-magnitude-comparison}
		\bigl||V(t)|-|V(0)|\bigr|
		\leq
		\int_0^t\norm{E(s)}_{L^\infty}\d s =: A_t.
	\end{equation}
	Then this implies
	\[
		|V(t)|\leq |V(0)|+A_t \quad \implies \quad \langle V(t)\rangle
		\leq
		C(1+A_t)\langle V(0)\rangle\,,
	\]
	and consequently
	\begin{equation}\label{eq:weighted-velocity-comparison}
		\langle V(t)\rangle^k
		\leq
		C_k(1+A_t)^k\langle V(0)\rangle^k.
	\end{equation}

	Since $f$ is constant along characteristics, then combining this identity with
	\eqref{eq:weighted-velocity-comparison}, we obtain
	\[
		\begin{aligned}
			\langle V(t)\rangle^k
			f(t,X(t),V(t))
			&=
			\langle V(t)\rangle^k
			f^{\rm in}(X(0),V(0))\\
			&\leq
			C_k(1+A_t)^k
			\langle V(0)\rangle^k
			f^{\rm in}(X(0),V(0))\\
			&\leq
			C_k(1+A_t)^k
			\norm{\langle v\rangle^k f^{\rm in}}_{L^\infty}.
		\end{aligned}
	\]
	Taking the supremum over all characteristics yields
	\[
		\norm{\langle v\rangle^k f(t)}_{L^\infty}
		\leq
		C_k
		\left(
			1+\int_0^t\norm{E(s)}_{L^\infty}\d s
		\right)^k
		\norm{\langle v\rangle^k f^{\rm in}}_{L^\infty}.
	\]
	Taking the supremum over $t\in[0,T]$ completes the proof.
\end{proof}

As a consequence, we have the following result. 
\begin{lemma}[Pointwise velocity tails imply bounded density]
\label{lem:velocity-tail-bounded-density}
Let $k>3$ and suppose that
\[
	\norm{\langle v\rangle^k f(t)}_{L^\infty(\Xi)}<\infty.
\]
Then
\[
	\rho_f(t)\in L^\infty(\rr^3)
\]
and
\begin{equation}
	\norm{\rho_f(t)}_{L^\infty}
	\leq
	C_k
	\norm{\langle v\rangle^k f(t)}_{L^\infty(\Xi)}.
\end{equation}
Consequently, if
\[
	\sup_{0\leq t\leq T}
	\norm{\langle v\rangle^k f(t)}_{L^\infty(\Xi)}<\infty,
\]
then
\[
	\rho_f\in L^\infty\bigl([0,T]\times\rr^3\bigr).
\]
\end{lemma}

	\begin{lemma}\label{Lemma:EstimateofSingularPotential2}
		Let $k\in\mathbb{N}^*$. There exists a universal constant $C>0$ such that
		\begin{equation}
			\lnorm{\infty}{\nabla^k_xE}\le 
			C\left(1+\lnorm{1}{\rho_{\nabla^{k-1}_{x}f}}+\lnorm{\infty}{\rho_{\nabla^{k-1}_{x}f}}\ln\left(1+\lnorm{\infty}{\rho_{\nabla^k_{x}f}}\right)\right)\,.\label{est:L^infty_derivatives_E}
		\end{equation}
        Here, we define
        \[
        \rho_g(x):=\int_{\mathbb{R}^3}g(x,v)\,\mathrm{d}v,
        \]
        and the norms of the tensor-valued quantities
        $\rho_{\nabla_x^\ell f}$ are understood componentwise.
	\end{lemma}
	\begin{proof}
			Since $\nabla^2_x K$ is a Calder\'{o}n--Zygmund kernel whose mean over $\mathbb{S}^2$ vanishes, for simplicity we denote it by $\nabla^2_x K=\mathrm{p.v.}\,\frac{\Omega\left(x/\snorm{x}\right)}{\snorm{x}^3}$. Following the method of the proof of \cite[Lemma 4.5.2]{golse2013mean}, we obtain
			\begin{equation*}
					\begin{aligned}
							\int_{\snorm{y}>\epsilon}\frac{\Omega\left(\frac{y}{\snorm{y}}\right)}{\snorm{y}^3}\rho_{\nabla^{k-1}_xf}(x-y)\,\mathrm{d} y&=\int_{\snorm{y}>1}+\int_{r\le\snorm{y}\le 1}+\int_{\epsilon\le\snorm{y}\le r}\frac{\Omega\left(\frac{y}{\snorm{y}}\right)}{\snorm{y}^3}\rho_{\nabla^{k-1}_xf}(x-y)\,\mathrm{d} y\\&=:I_1+I_2+I_3\,.
						\end{aligned}
				\end{equation*}
			First, it is easy to see 
			\begin{equation*}
					\snorm{I_1}\le \norm{\Omega}_{\mathbb{S}^2}\norm{\rho_{\nabla^{k-1}_xf}}_{L^1}\,.
				\end{equation*}
			For the $I_2$ we have
			\begin{equation*}
					\begin{aligned}
							\snorm{I_2}&\le \norm{\Omega}_{L^{\infty}(\mathbb{S}^2)}\norm{\rho_{\nabla^{k-1}_xf}}_{L^{\infty}(\rr^3)}\int_{r\le\snorm{y}\le 1}\frac{dy}{\snorm{y}^d}\\
							&=\norm{\Omega}_{L^{\infty}(\mathbb{S}^2)}\norm{\rho_{\nabla^{k-1}_xf}}_{L^{\infty}(\rr^3)}\snorm{\mathbb{S}^2}\ln\frac{1}{r}\,.
						\end{aligned}
				\end{equation*}
			We apply the vanishing property of integral kernel to obtain
			\begin{equation*}
					\begin{aligned}
							\snorm{I_3}&=\snorm{\int_{\epsilon\le \snorm{y}\le r}\frac{\Omega(\frac{y}{\snorm{y}})}{\snorm{y}^d}\left(\rho_{\nabla^{k-1}_xf}(x-y)-\rho_{\nabla^{k-1}_xf}(x)\right)\,\mathrm{d} y}\\
							&\le\int_{\epsilon\le \snorm{y}\le r}\frac{\snorm{\Omega(\frac{y}{\snorm{y}})}}{\snorm{y}^d}\snorm{\rho_{\nabla^{k}_xf}}\snorm{y}\,\mathrm{d} y\\
							&\le \norm{\Omega}_{L^{\infty}(\mathbb{S}^2)}\norm{\rho_{\nabla^{k}_xf}}_{L^{\infty}(\rr^3)}\snorm{\mathbb{S}^2}r\,.
						\end{aligned}
				\end{equation*}
			Therefore
			\begin{equation*}
					\begin{aligned}
							\snorm{I_1}+\snorm{I_2}+\snorm{I_3}\le  \norm{\Omega}_{L^{\infty}}
							\bigl(
							\lnorm{1}{\rho_{\nabla^{k-1}_xf}}
							+ \lnorm{\infty}{\rho_{\nabla^{k-1}_xf}}\,\snorm{\mathbb{S}^2}\,\ln\frac{1}{r}
							+ \lnorm{\infty}{\rho_{\nabla^k_xf}}\,\snorm{\mathbb{S}^2}\,r
							\bigr)\,.
						\end{aligned}
				\end{equation*}
			Take $r=(1+\lnorm{\infty}{\rho_{\nabla^k_xf}})^{-1}$, then we have
			\begin{equation}
					\lnorm{\infty}{\nabla^k_xE}\le C\left(1+\lnorm{1}{\rho_{\nabla^{k-1}_{x}f}}+\lnorm{\infty}{\rho_{\nabla^{k-1}_{x}f}} \ln\left(1+\lnorm{\infty}{\rho_{\nabla^k_{x}f}}\right)\right)\,.
				\end{equation}
			This completes the proof. 
		\end{proof}

	\subsection{Proof of propagation of regularity}
    In this subsection, we shall prove Theorem~\ref{Theorem:PropagationOfRegularity}.

    \bigskip 

    Let us begin by proving some lemmas. 

\begin{lemma}
\label{lem:weighted-velocity-integration}
Let $1\leq p\leq\infty$, $r\in\mathbb{N}_0$, and $\lambda\geq\lambda_*>0$. Suppose that
\begin{equation*}
w_\lambda\partial_x^\alpha u
\in L^p(\mathbb{R}^6),
\qquad
|\alpha|\leq r.
\end{equation*}
Then
$\rho_u\in W^{r,p}(\mathbb{R}^3)$ and satisfies
\begin{equation}
\label{eq:weighted-velocity-integration}
\|\partial_x^\alpha\rho_u\|_{L^p_x}
\leq
C_{\lambda,p}
\|w_\lambda\partial_x^\alpha u\|_{L^p_{x,v}},
\qquad
|\alpha|\leq r,
\end{equation}
where
\begin{equation}
\label{eq:velocity-weight-endpoints}
C_{\lambda,p}:=\|e^{-\lambda\langle v\rangle}\|_{L^{p'}_v}
\leq
\begin{cases}
e^{-\lambda},
& p=1,\\[2mm]
C_p\left(1+\lambda^{-3/p'}\right),
& 1<p<\infty,\\[2mm]
C\left(1+\lambda^{-3}\right),
& p=\infty.
\end{cases}
\end{equation}
Consequently,  we have
\begin{equation}
\label{eq:density-Sobolev-bound}
\|\rho_u\|_{W^{r,p}_x}
\leq
C_{\lambda_*,r,p}\|u\|_{\mathcal{X}^{r, p}_{0, \lambda}}.
\end{equation}
\end{lemma}

\begin{proof}
Hölder's inequality in the velocity variable gives
\begin{equation*}
\begin{aligned}
|\partial_x^\alpha\rho_u(x)|
&=
\left|
\int_{\mathbb{R}^3}
\partial_x^\alpha u(x,v)\,\mathrm{d}v
\right|
\\
&\leq
\left(
\int_{\mathbb{R}^3}
w_\lambda^p
|\partial_x^\alpha u(x,v)|^p\,\mathrm{d}v
\right)^{1/p}
\left(
\int_{\mathbb{R}^3}
e^{-p'\lambda\langle v\rangle}\,\mathrm{d}v
\right)^{1/p'}.
\end{aligned}
\end{equation*}
Taking the $L^p_x$ norm completes the proof. 
\end{proof}
        \begin{lemma}[Commutator structure]\label{lem:exact-commutator}  
        Let
        \[
        \mathcal \sfT:=\partial_t+v\cdot\nabla_x+F\cdot\nabla_v \quad \text{ with } \quad F=E+v\wedge B\,,
        \]
        and suppose $\sfT f=0$.  Suppose
        $g_{\alpha\beta}=D^{\alpha,\beta}f$, then we have
        \begin{equation} \label{eq:T_acting_on_derivatives}
        \begin{aligned}
        \mathsf{T}g_{\alpha\beta}
        ={}&-\sum_{i=1}^3\beta_i
        D^{\alpha+e_i,\beta-e_i}f\\
        &-\sum_{\substack{\alpha_1\leq\alpha,\ \beta_1\leq\beta\\
                 |\alpha_1|+|\beta_1|\geq1}}
        C_{\alpha,\beta}^{\alpha_1,\beta_1}
        (D^{\alpha_1,\beta_1}F)\cdot
        \nabla_vD^{\alpha-\alpha_1,\beta-\beta_1}f .
        \end{aligned}
        \end{equation}
        Furthermore, we have
        \begin{equation}\label{eq:derivative_of_F}
        D^{\alpha_1,\beta_1}F=
        \begin{cases}
        \partial_x^{\alpha_1}E+ v\wedge\partial_x^{\alpha_1}B,
           &\beta_1=0,\\
        e_\ell\wedge\partial_x^{\alpha_1}B,
           &\beta_1=e_\ell,\\
        0,&|\beta_1|\geq2.
        \end{cases}                                       
        \end{equation}
        Note that in \eqref{eq:T_acting_on_derivatives}, the undifferentiated force is excluded.  No term contains more
        than $|\alpha|+|\beta|$ derivatives of $f$.  The only commutator terms that retain a factor $v$ are
\begin{equation}\label{eq:linear_in_v}
    \bigl(v\wedge\partial_x^{\alpha_1}B\bigr)
\cdot
\nabla_vD^{\alpha-\alpha_1,\beta}f,
\qquad
0<\alpha_1\leq\alpha\,.
\end{equation}
        \end{lemma}
\begin{proof}[Proof]

Commute $D^{\alpha,\beta}$ through each part of $\sfT$, i.e., compute $[D^{\alpha, \beta}, \sfT]f$. 
Every velocity derivative that lands on $v\cdot\nabla_x$ produces
one $x$-derivative and gives the first line of \eqref{eq:T_acting_on_derivatives}; its total order is
unchanged. Leibniz's formula for the force term gives the second line.
Since $E$ is independent of $v$ and $v\mapsto v\wedge B$ is
linear, \eqref{eq:derivative_of_F} is exact. If a velocity derivative lands on that linear
map, the factor $v$ disappears. Thus it remains only when at least one
spatial derivative, and no velocity derivative, lands on $B$, which is
exactly \eqref{eq:linear_in_v}. 
\end{proof}
\begin{lemma}[Phase-space weighted transport identity]
\label{lem:weighted-transport-identity}
Let $1\leq p<\infty$, and let $g$ be a smooth, sufficiently decaying
real-valued solution of
\[
    \mathsf{T}g=H,
    \qquad
    \mathsf{T}
    :=
    \partial_t
    +v\cdot\nabla_x
    +F\cdot\nabla_v\quad \text{ with } \quad F=E+v\times B\,.
\]
Let $\lambda,\mu\in C^1([0,T])$, and define
\[
    \Phi(t,x,v)
    :=
    \lambda(t)\langle v\rangle
    +
    \mu(t)\langle x\rangle,
    \qquad
    W(t,x,v):=e^{p\Phi(t,x,v)}.
\]
Then
\begin{multline}
\label{eq:weighted_transport_id}
    \frac{1}{p}\frac{\d}{\d t}
    \iint_{\Xi}W|g|^p\,\dd x\d v
    =
    \iint_{\Xi}
    WH|g|^{p-2}g\,\dd x\d v
    \\
    +
    \iint_{\Xi}
    \left[
        \lambda'(t)\langle v\rangle
        +
        \mu'(t)\langle x\rangle
        +
        \mu(t)\frac{x\cdot v}{\langle x\rangle}
        +
        \lambda(t)\frac{v\cdot E(t,x)}{\langle v\rangle}
    \right]
    W|g|^p\,\dd x\d v.
\end{multline}
For $p=1$, the factor $|g|^{p-2}g$ is understood as
$\operatorname{sgn}(g)$, and the identity follows by a standard
regularization argument.
\end{lemma}

\begin{proof}
The phase-space vector field $\bigl(v, F\bigr)$
is divergence free. Multiplying $\mathsf{T}g=H$ by $W|g|^{p-2}g$, integrating over
$\Xi$, and integrating by parts therefore gives
\begin{multline}
    \frac{1}{p}\frac{\d}{\d t}
    \iint_{\Xi}W|g|^p\,\dd x\d v
    =
    \iint_{\Xi}
    WH|g|^{p-2}g\,\dd x\d v
    \\
    +
    \frac{1}{p}
    \iint_{\Xi}
    \left[
        \partial_tW
        +
        v\cdot\nabla_xW
        +
        (E+v\times B)\cdot\nabla_vW
    \right]
    |g|^p\,\dd x\d v.
\end{multline}
The derivatives of the weight are
\[
    \partial_tW
    =
    p\left[
        \lambda'(t)\langle v\rangle
        +
        \mu'(t)\langle x\rangle
    \right]W,
\]
\[
    \nabla_xW
    =
    p\mu(t)\frac{x}{\langle x\rangle}W,
    \qquad
    \nabla_vW
    =
    p\lambda(t)\frac{v}{\langle v\rangle}W.
\]
Consequently, we have
\[
    v\cdot\nabla_xW
    =
    p\mu(t)\frac{x\cdot v}{\langle x\rangle}W.
\]
Moreover, since $v\cdot(v\times B)=0$,
\begin{align*}
    (E+v\times B)\cdot\nabla_vW
    &=
    p\lambda(t)
    \frac{v\cdot(E+v\times B)}{\langle v\rangle}W
    \\
    &=
    p\lambda(t)
    \frac{v\cdot E}{\langle v\rangle}W.
\end{align*}
Substitution of these identities gives
\eqref{eq:weighted_transport_id}.
\end{proof}

\begin{proposition}\label{prop:velocity_commutator}
    Let $1\le p<\infty$, and let $n$ be the order of the derivatives. There exists a constant $C_{n}^{\rm mag}>0$, dependent on $n$, such that the sum of all contributions to \eqref{eq:weighted_transport_id} by terms \eqref{eq:linear_in_v}, denoted by $\mathcal{C}^{v}_{p, n}$, has the bound 
    \begin{equation}
        |\mathcal{C}^{v}_{p, n}|\le C_{n}^{\rm mag}\|\grad_x B\|_{W^{n-1, \infty}} \sum_{|\alpha|+|\beta|\le n}\iint_{\Xi} \langle v\rangle W |D^{\alpha, \beta}f|^p\dd x\d v\,.
    \end{equation}
\end{proposition}
\begin{proof}
    Observe that for any pairs $(g, h)$, by Young's inequality, we have
    \begin{equation*}
        \iint_{\Xi} \langle v\rangle W |h||g|^{p-1}\dd x\d v \le \frac{1}{p}\iint_{\Xi} \langle v\rangle W |h|^p\dd x\d v +\frac{p-1}{p}\iint_{\Xi} \langle v\rangle W |g|^p\dd x\d v\,.
    \end{equation*}
    In this application, we set $g=D^{\alpha, \beta} f$ and $h = D^{\alpha-\alpha_1, \beta+e_\ell} f$ where  we already have $\|B\|_{W^{|\alpha_1|, \infty}}$. Then the desired estimate follows.
\end{proof}

\begin{proposition}\label{prop:remaining_commutator}
    For all the commutator terms in \eqref{eq:T_acting_on_derivatives}--\eqref{eq:derivative_of_F} not covered by the previous proposition, denoted by $\mathcal{C}_{p,n}$, we have the estimate  
    \begin{equation}
        |\mathcal{C}_{p, n}| \le C_{n}\(1+\| B\|_{L^\infty}+\|\grad_xB\|_{W^{n-1, \infty}}+\|E\|_{W^{n, \infty}}\)\|f\|_{\mathcal{X}^{n, p}_{\mu(t), \lambda(t)}}^p\,.
    \end{equation}
\end{proposition}

\begin{proof}
    The argument is similar to the proof of the previous proposition. 
\end{proof}

\begin{lemma}\label{lem:diff_ineq_for_regularity}
    Fix $[0, T]$. Choose 
    \begin{align}
        \mu(t):=&\, \mu_0 e^{-\kappa t} \quad \text{ and }\\
        \lambda(t) :=&\, \lambda_0 -\frac{\mu_0}{\kappa}(1-e^{-\kappa t})- C^{\rm mag}_{ n}\,\|\grad_x B\|_{L^\infty([0, T]; W^{n-1, \infty}_x)}\, t\,,
    \end{align}
    such that $\lambda(t)>0$ on $[0, T]$, where $C^{\rm mag}_{n}>0$ is from Proposition~\ref{prop:velocity_commutator} and $\mu_0, \kappa>0$.
    Then we have the bound 
    \begin{equation}
        \frac{1}{p}\frac{\d}{\d t}\|f\|_{\mathcal{X}^{n, p}_{\mu(t), \lambda(t)}}^p \le C_{n}\(1+\|\grad_x B\|_{W^{n-1, \infty}}+\|E\|_{W^{n, \infty}}+\lambda_0 \|E\|_{L^\infty}\)\|f\|_{\mathcal{X}^{n, p}_{\mu(t), \lambda(t)}}^p\,.
    \end{equation}
\end{lemma}
\begin{proof}
    By Lemma~\ref{lem:weighted-transport-identity}, we see that the term containing $\lambda'(t)$ yields 
    \begin{align}
       -\(\mu(t) +C^{\rm mag}_{n}\,\|\grad_x B\|_{L^\infty([0, T]; W^{n-1, \infty}_x)}\)\sum_{|\alpha|+|\beta|\le n}\iint_{\Xi} \langle v\rangle W |D^{\alpha, \beta}f|^p\dd x\d v\,,
    \end{align}
    which exactly cancels out sum of the upper bound of the terms with linear growth in $v$ given in Proposition~\ref{prop:velocity_commutator} and the term with $\mu(t) x\cdot v/\langle x\rangle$. The term with $\mu'(t)$ can be dropped since $\mu'(t)\le 0$. The remaining terms can be controlled by Proposition~\ref{prop:remaining_commutator} and the fact
    \[
    \n{\lambda(t)\frac{v}{\langle v\rangle}\cdot E(t,x)}
    \leq
    \lambda(t)\|E(t)\|_{L^\infty}\le \lambda_0\|E(t)\|_{L^\infty} \,.
    \]
    This completes the proof.
\end{proof}

\begin{proof}[Proof of Theorem~\ref{Theorem:PropagationOfRegularity}]
    Let $n=1$. By Lemma~\ref{lem:diff_ineq_for_regularity} and Lemma~\ref{Lemma:EstimateofSingularPotential2},  we obtain the bound 
    \begin{align*}
               \frac{1}{p}\frac{\d}{\d t}\|f\|_{\mathcal{X}^{1, p}_{\mu(t), \lambda(t)}}^p
                &\le\,  (b_1 +d_1 \ln\left(1+\lnorm{\infty}{\rho_{\nabla_{x}f}}\right)) \|f\|_{\mathcal{X}^{1, p}_{\mu(t),\lambda(t)}}^p\,,
    \end{align*}
where $b_1=b_1(\|B\|_{W^{1, \infty}}, \lambda_0, \|f^{\rm in}\|_{W^{1,1}})$ and $d_1 = d_1(\lambda_0, \|f^{\rm in}\|_{W^{1,\infty}})$ are constants, $k\in \mathbb{N}$. Hence, by a Gr\"onwall argument, for all $t$ in $[0,T]$, we have that 
\begin{equation}
\|f(t)\|_{\mathcal{X}_{\mu(t),\lambda(t)}^{1,p}}
\leq
\|f^{\mathrm{in}}\|_{\mathcal{X}_{\mu_0,\lambda_0}^{1,p}}
\exp\left[
b_1t
+
d_1
\int_0^t
\ln\left(
1+\|\rho_{\nabla_x f}(s)\|_{L^\infty}
\right)\dd s
\right]\,.
\end{equation}

		Take $p\to \infty$, we have
        \begin{equation}
        \|f(t)\|_{\mathcal{X}_{\mu(t),\lambda(t)}^{1,\infty}}
        \leq
        \|f^{\mathrm{in}}\|_{\mathcal{X}_{\mu_0,\lambda_0}^{1,\infty}}
        \exp\left[
        b_1t
        +
        d_1
        \int_0^t
        \ln\left(
        1+\|\rho_{\nabla_x f}(s)\|_{L^\infty}
        \right)\dd s
        \right]\,.
        \end{equation}
        Then, by Lemma~\ref{lem:weighted-velocity-integration}, we have
        \begin{equation}
            \|\rho_{\grad_x f}\|_{L^\infty} \le C\|f^{\mathrm{in}}\|_{\mathcal{X}_{\mu_0,\lambda_0}^{1,\infty}}
            \exp\left[
            b_1t
            +
            d_1
            \int_0^t
            \ln\left(
            1+\|\rho_{\nabla_x f}(s)\|_{L^\infty}
            \right)\dd s
            \right]
        \end{equation}
		which yields
        \begin{equation}
        \begin{aligned}
        1+\ln\left(
        1+\left\|\rho_{\grad_x f}(t)\right\|_{L^\infty}
        \right)
        &\le C_1
        +D_1\int_0^t
        \left[
        1+\ln\left(
        1+\left\|\rho_{\grad_x f}(s)\right\|_{L^\infty}
        \right)
        \right]\dd s\,,
        \end{aligned}
        \end{equation}
        for some $C_1, D_1$ positive constants. 
		Again, we apply Gr\"{o}nwall's inequality to obtain
		\begin{equation*}
        \lnorm{\infty}{\rho_{\nabla_x f}}\le \exp\left(c_1\exp\left(c_1 t\right)\right)\,.
		\end{equation*}
         Here and below, the constants $c_k$ have the following dependence
        $$c_k=c_k\left(\norm{B}_{W^{k,\infty}},\mu_0,\lambda_0,\norm{f^{\mathrm{in}}}_{W^{k,1}\cap W^{k,\infty}}\right),\qquad k\in \mathbb{N}\,.$$
        Consequently, set $p=1$ we directly obtain for $\forall t\in [0,T]$,
        \begin{equation}
            \begin{aligned}
                \norm{f}_{\mathcal{X}^{1, 1}_{\mu(t),\lambda(t)}}\le \norm{f^{\mathrm{in}}}_{\mathcal{X}^{1,1}_{\mu_0,\lambda_0}}\exp\left[c_1 t+\exp(c_1t)\right]
            \end{aligned}
        \end{equation}

        Now we suppose that for $1\le j\le k-1$, $k\le m$ and for all $t\in [0,T]$, we have the following estimates:
        \begin{equation}
            \begin{aligned}
            \norm{f}_{\mathcal{X}^{j,1}_{\mu(t),\lambda(t)}\cap\mathcal{X}^{j,\infty}_{\mu(t),\lambda(t)}}\le \norm{f^{\mathrm{in}}}_{\mathcal{X}^{j,1}_{\mu_0,\lambda_0}\cap\mathcal{X}^{j,\infty}_{\mu_0,\lambda_0}}\exp\left[c_j t+\exp(c_j t)\right]
            \end{aligned}
        \end{equation}
        and 
        \begin{equation}
            \begin{aligned}
                \norm{\rho_{\nabla^j_x f}}_{L^\infty}\le \exp[c_j\exp(c_j t)]\,.
            \end{aligned}
        \end{equation}
        Then according to Lemma \ref{Lemma:EstimateofSingularPotential2} and Lemma \ref{lem:diff_ineq_for_regularity} we repeat the argument above to obtain
        \begin{equation*}
            \begin{aligned}
                \norm{f}_{\mathcal{X}^{k,1}_{\mu(t),\lambda(t)}\cap\mathcal{X}^{k,\infty}_{\mu(t),\lambda(t)}}\le \norm{f^{\mathrm{in}}}_{\mathcal{X}^{k,1}_{\mu_0,\lambda_0}\cap\mathcal{X}^{k,\infty}_{\mu_0,\lambda_0}}\exp\left[c_{k-1}t+c_{k-1}\int^t_0\left(1+\ln\left(1+\norm{\rho_{\nabla^k_x f}(s)}_{L^{\infty}}\right)\right)\dd s\right]\,
            \end{aligned}
        \end{equation*}
        and
           \begin{equation}
        \begin{aligned}
        1+\ln\left(
        1+\left\|\rho_{\grad_x^k f}(t)\right\|_{L^\infty}
        \right)
        &\le C_k
        +D_k\int_0^t
        \left[
        1+\ln\left(
        1+\left\|\rho_{\grad_x^k f}(s)\right\|_{L^\infty}
        \right)
        \right]\dd s\,,
        \end{aligned}
        \end{equation}
        Therefore,
        \begin{equation}
            \begin{aligned}
                \norm{f}_{\mathcal{X}^{k,1}_{\mu(t),\lambda(t)}\cap\mathcal{X}^{k,\infty}_{\mu(t),\lambda(t)}}\le \norm{f^{\mathrm{in}}}_{\mathcal{X}^{k,1}_{\mu_0,\lambda_0}\cap\mathcal{X}^{k,\infty}_{\mu_0,\lambda_0}}\exp\left[c_k t+\exp(c_k t)\right]\,.
            \end{aligned}
        \end{equation}
        Hence the proof of Theorem \ref{Theorem:PropagationOfRegularity} is completed via mathematical induction.
\end{proof}

        \begin{corollary}[Propagation of spatial moments]
\label{cor:propagation-spatial-moments}
Let the assumptions of
Theorem~\ref{Theorem:PropagationOfRegularity} hold, and let
$\lambda_*=\inf_{t\in[0,T]}\lambda(t)>0$. For $k\geq 1$, assume in
addition that
\begin{equation}
    M_k^x(0)
    :=
    \iint_{\Xi}\langle x\rangle^k
    f^{\mathrm{in}}(x,v)\,\dd x\d v
    <\infty.
\end{equation}
Then
\begin{equation}
    M_k^x(t)
    :=
    \iint_{\Xi}\langle x\rangle^k
    f(t,x,v)\,\dd x\d v
\end{equation}
remains finite on $[0,T]$. More precisely,
\begin{equation}
\label{eq:propagation-spatial-moments}
    \sup_{0\leq t\leq T}M_k^x(t)
    \leq
    e^{C_kT}
    \left(
        M_k^x(0)
        +
        C_kT\,C_{k,\lambda_*}
        \sup_{0\leq t\leq T}
        \norm{f(t)}_{\mathcal{X}^{m,1}_{\mu(t),\lambda(t)}}
    \right),
\end{equation}
where
\begin{equation}
    C_{k,\lambda_*}
    :=
    \sup_{v\in\mathbb{R}^3}
    \langle v\rangle^k
    e^{-\lambda_*\langle v\rangle}
    <\infty.
\end{equation}
In particular, every spatial moment that is finite initially remains
finite on any interval on which the exponential regularity estimate
holds.
\end{corollary}
\begin{proof}
Taking the scalar product of the second equation in
\eqref{eq:DiPerna-Lions-flow2} with $V$, we obtain
\[
    \frac{1}{2}\frac{\d}{\d t}|V|^2
    =
    V\cdot E(t,X)\,.
\]
Consequently, we have
\[
    |V(t,x,v)|
    \leq
    |v|+\int_0^t\norm{E(s)}_{L^\infty}\,\dd s.
\]
The field estimates used in the proof of
Theorem~\ref{Theorem:PropagationOfRegularity} imply that
\[
    C_E(T):=
    \sup_{0\leq s\leq T}\norm{E(s)}_{L^\infty}<\infty.
\]
Therefore,
\begin{equation*}
    |V(t,x,v)|
    \leq
    |v|+C_E(T)t.
\end{equation*}
Since $\dot X=V$, it follows that
\begin{equation*}
\begin{aligned}
    |X(t,x,v)|
    &\leq
    |x|+\int_0^t|V(s,x,v)|\,\dd s\\
    &\leq
    |x|+|v|t+\tfrac{1}{2}C_E(T)t^2.
\end{aligned}
\end{equation*}

The phase-space vector field
\[
    (v,E+v\wedge B)
\]
is divergence free. Hence the associated DiPerna--Lions flow preserves
Lebesgue measure, and
\[
    f(t,X(t,x,v),V(t,x,v))
    =
    f^{\mathrm{in}}(x,v)
\]
for almost every $(x,v)\in\Xi$. Thus,
\begin{align}
    \int_{\Xi}|x|^k f(t,x,v)\,\dd x\d v
    &=
    \int_{\Xi}
    |X(t,x,v)|^k
    f^{\mathrm{in}}(x,v)\,\dd x\d v \notag\\
    &\leq
    C_k
    \int_{\Xi}
    \left(
        |x|^k
        +|v|^k t^k
        +C_E(T)^k t^{2k}
    \right)
    f^{\mathrm{in}}(x,v)\,\dd x\d v \notag\\
    &\leq
    C_{k,T}
    \left(
        \int_{\Xi}|x|^k f^{\mathrm{in}}\,\dd x\d v
        +
        M_k(0)
        +
        \norm{f^{\mathrm{in}}}_{L^1}
    \right)
    <\infty.
\end{align}
Here, the initial spatial moment is finite by assumption, while
$M_k(0)<\infty$ follows from the exponential velocity decay of
$f^{\mathrm{in}}$. This proves the propagation of the spatial moment.
\end{proof}

	\section{Regularity of the Magnetic Weyl Quantization}\label{Section:ProofOfRegularityOfMagneticWeyl}
    The aim of this section is to establish the regularity estimates for the magnetic Weyl quantization needed in the proof of Theorem~\ref{Theorem:Semiclassicallimit}. Throughout this section, the dimension $d$ is arbitrary. We divide the section into two parts. The first subsection contains gauge-covariant estimates whose constants depend on the magnetic field $B$ but not on the choice of vector potential, while the second contains auxiliary estimates proved in a fixed reference gauge. We then use the gauge-covariance property of the magnetic Weyl quantization to transfer the fixed-gauge estimates to arbitrary vector potentials generating the same magnetic field; see Proposition~\ref{prop:A1-gauge-covariance} for a proof of this property. 

\subsection{Gauge--covariant estimates} We begin with the magnetic analogue of the semiclassical Calder\'on--Vaillancourt theorem, which provides the $\mathcal{L}^\infty$ estimate required to control the error term; see \cite{boulkhemair1999l2,calderon1972class} for the corresponding nonmagnetic results.
	
\begin{proposition}[Magnetic Calder\'on--Vaillancourt theorem \cite{iftimie2007magnetic}]
\label{prop:MagneticCalderonVaillancourt}
	Assume $\bB$ is a bounded magnetic field and $\bA$ is a polynomially bounded vector potential associated with $\bB$ (i.e., all components of $\bA$ are bounded by smooth polynomials). Set $h=1$; then for all $f\in\mathcal{S}^{0}_{\rho,\delta}$ with $0\le\rho=\delta<1$ or $0\le\delta<\rho\le 1$, we can assert that $\OPA(f)\in \mathcal{B}(L^{2}(\rr^d))$ and the operator norm (i.e., the Schatten-$\infty$ norm) can be bounded by the following relation:
	\begin{equation}
		\norm{\OPA(f)}_{\mathcal{B}(L^{2})}
		\le
		C(d)
		\sup_{\snorm{a},\snorm{\alpha}\le p(d)}
		\sup_{(x,v)\in\rr^d\times\rr^d}
		\langle v\rangle^{\rho\snorm{a}-\delta\snorm{\alpha}}
		\snorm{D^{a,\alpha}f(x,v)}.
	\end{equation}
	The constants $C(d)$ and $p(d)$ depend only on $d$ and can be determined explicitly.
\end{proposition}

From the work of Athmouni and Purice in \cite{athmouni2018schatten}, we also have the following $\cL^1$ estimate result.

\begin{proposition}\label{Proposition:L1estimateofsymbol}
	Suppose $B$ is a magnetic field with corresponding vector potential $A$ which satisfies Assumption \ref{Assumption: bounded B}. Let $f\in \mathcal{S}'(\Xi)$, and set
	\[
		s(d):=2\left\lfloor \frac{d}{2}\right\rfloor+2,
		\qquad
		t(d):=d+\left\lfloor\frac{d}{2}\right\rfloor+1.
	\]
	If $\partial^\alpha_x\partial^\beta_vf\in L^1(\Xi)$ for $\snorm{\alpha}\le s(d)$ and $\snorm{\beta}\le t(d)$, then $\OPA_1(f)\in \mathfrak{S}^1$, and there exists a positive constant $C$, depending on the dimension $d$ and finitely many seminorms of $B$, such that
	\begin{equation*}
		\cstnorm{1}{\OPA_1(f)}
		\le
		C
		\sum_{\snorm{\alpha}\le s(d)}
		\sum_{\snorm{\beta}\le t(d)}
		\lnorm{1}{\partial^\alpha_x\partial^\beta_vf}.
	\end{equation*}
\end{proposition}

We first record the scaling relation for the magnetic Weyl
quantization. For $h>0$, define the unitary dilation
$\sfD_{h}:L^2(\mathbb{R}^d)\to L^2(\mathbb{R}^d)$ by
\begin{equation}
    (\sfD_{h} u)(x):=h^{d/2}u(hx).
\end{equation}
We also introduce the rescaled symbol and vector potential
\begin{equation}
    f_h(x,v):=f(hx,v),
    \qquad
    A_h(x):=A(hx).
\end{equation}

\begin{lemma}[Scaling of the magnetic Weyl quantization]
\label{Lemma:ScalingOfMagneticWeylTransform}
Let $f\in\mathcal{S}(\Xi)$ and
$u\in\mathcal{S}(\mathbb{R}^d)$. Then
\begin{equation}
\label{eq:scaling-magnetic-Weyl}
    \sfD_{h}\OPAh(f)\sfD_{h}^{-1}
    =
    \operatorname{Op}^{A_h}_1(f_h).
\end{equation}
Equivalently, we have
\begin{equation}
\label{eq:scaling-magnetic-Weyl-pointwise}
    \bigl(\OPAh(f)u\bigr)(x)
    =
    \left[
        \operatorname{Op}^{A(h\cdot)}_1
        \bigl(f(h\cdot,\cdot)\bigr)
        \bigl(u(h\cdot)\bigr)
    \right]\left(\frac{x}{h}\right).
\end{equation}
\end{lemma}

\begin{proof}
For $X,Y\in\mathbb{R}^d$, the magnetic phases satisfy
\begin{equation}
\label{eq:scaling-magnetic-phase}
    \frac{1}{h}\Gamma^A([hX,hY])
    =
    \Gamma^{A_h}([X,Y]).
\end{equation}
Indeed, we have
\begin{align*}
    \Gamma^A([hX,hY])
    &=
    \int_0^1
    A\bigl(hX+sh(Y-X)\bigr)
    \cdot h(Y-X)\,\mathrm ds
    \\
    &=
    h\int_0^1
    A\bigl(h(X+s(Y-X))\bigr)
    \cdot(Y-X)\,\mathrm ds
    \\
    &=
    h\,\Gamma^{A_h}([X,Y]).
\end{align*}

Using the definition of $\OPAh(f)$, we obtain
\begin{align*}
    &\bigl(
        \sfD_{h}
        \OPAh(f)
        \sfD_{h}^{-1}u
    \bigr)(X)
    \\
    &\quad =
    h^{d/2}
    \iint_{\mathbb{R}^d\times\mathbb{R}^d}
    e^{-2\pi i\,v\cdot(y-hX)}
    e^{-\frac{2\pi i}{h}\Gamma^A([hX,y])}
    f\left(\frac{hX+y}{2},hv\right)
    h^{-d/2}u\left(\frac{y}{h}\right)
    \,\mathrm dy\,\mathrm dv.
\end{align*}
Making the changes of variables $y=hY,
    \eta=hv$,
and using \eqref{eq:scaling-magnetic-phase}, we find
\begin{align*}
    &\bigl(
        \sfD_{h}
        \OPAh(f)
        \sfD_{h}^{-1}u
    \bigr)(X)
    \\
    &\quad =
    \iint_{\mathbb{R}^d\times\mathbb{R}^d}
    e^{-2\pi i\,\eta\cdot(Y-X)}
    e^{-2\pi i\,\Gamma^{A_h}([X,Y])}
    f\left(h\tfrac{X+Y}{2},\eta\right)
    u(Y)\,\mathrm dY\,\mathrm d\eta
    \\
    &\quad =
    \bigl(\operatorname{Op}^{A_h}_1(f_h)u\bigr)(X).
\end{align*}
This proves \eqref{eq:scaling-magnetic-Weyl}. Formula
\eqref{eq:scaling-magnetic-Weyl-pointwise} follows by applying
$\sfD_{h}^{-1}$ and using the linearity of
$\operatorname{Op}^{A_h}_1(f_h)$.
\end{proof}

We now apply the preceding $h=1$ estimates. Notice that the
magnetic field associated with $A_h$ is
\begin{equation}
\label{eq:rescaled-magnetic-field}
    B_h:=\mathrm dA_h=hB(h\cdot).
\end{equation}
Consequently, for every multi-index $\gamma$,
\begin{equation}
\label{eq:rescaled-B-seminorms}
    \|\partial^\gamma B_h\|_{L^\infty}
    =
    h^{|\gamma|+1}
    \|\partial^\gamma B\|_{L^\infty}.
\end{equation}
In particular, when $0<h\leq1$, all the finitely many seminorms of
$B_h$ appearing in the scale-one estimates are bounded by the
corresponding seminorms of $B$. Thus the constants in those estimates
can be chosen independently of $h$.

We first apply the magnetic Calder\'on--Vaillancourt estimate. Since
$\sfD_{h}$ is unitary, Lemma
\ref{Lemma:ScalingOfMagneticWeylTransform} gives
\begin{equation}
\label{eq:scaled-operator-norm-identity}
    \norm{\OPAh(f)}_{\cL^\infty}
    =
    \norm{\operatorname{Op}^{A_h}_1(f_h)}_{\cL^\infty}.
\end{equation}
Applying Proposition \ref{prop:MagneticCalderonVaillancourt} to
$f_h$ and $A_h$, we obtain
\begin{align}
\label{eq:semiclassical-magnetic-CV}
    \norm{\OPAh(f)}_{\cL^\infty}
    &\leq
    C
    \sup_{\snorm{a},\snorm{\alpha}\leq p(d)}
    \sup_{(x,v)\in\Xi}
    \langle v\rangle^{\rho\snorm{a}-\delta\snorm{\alpha}}
    \snorm{D^{a,\alpha}f_h(x,v)}
    \notag\\
    &=
    C
    \sup_{\snorm{a},\snorm{\alpha}\leq p(d)}
    h^{\snorm{a}}
    \sup_{(x,v)\in\Xi}
    \langle v\rangle^{\rho\snorm{a}-\delta\snorm{\alpha}}
    \snorm{D^{a,\alpha}f(x,v)}
    \notag\\
    &\leq
    C
    \sup_{\snorm{a},\snorm{\alpha}\leq p(d)}
    \sup_{(x,v)\in\Xi}
    \langle v\rangle^{\rho\snorm{a}-\delta\snorm{\alpha}}
    \snorm{D^{a,\alpha}f(x,v)},
    \qquad 0<h\leq1.
\end{align}
Here we used
\[
    D^{a,\alpha}f_h(x,v)
    =
    h^{\snorm{a}}
    \bigl(D^{a,\alpha}f\bigr)(hx,v),
\]
where $a$ denotes the multi-index corresponding to differentiation
in the $x$-variable. Since the $\cL^\infty$ norm
coincides with the usual operator norm, this also gives
\begin{equation}
\label{eq:semiclassical-Schatten-infinity-estimate}
    \stnorm{\infty}{\OPAh(f)}
    \leq
    C
    \sup_{\snorm{a},\snorm{\alpha}\leq p(d)}
    \sup_{(x,v)\in\Xi}
    \langle v\rangle^{\rho\snorm{a}-\delta\snorm{\alpha}}
    \snorm{D^{a,\alpha}f(x,v)},
    \qquad 0<h\leq1.
\end{equation}

We next apply the preceding trace-class estimate at $h=1$. Since
$\sfD_{h}$ is unitary, Lemma
\ref{Lemma:ScalingOfMagneticWeylTransform} gives
\begin{equation}
    \cstnorm{1}{\OPAh(f)}
    =
    \cstnorm{1}{\operatorname{Op}^{A_h}_1(f_h)}.
\end{equation}
The $h=1$ estimate therefore yields
\begin{equation}
\label{eq:classical-trace-scaled}
    \cstnorm{1}{\OPAh(f)}
    \leq
    C
    \sum_{|\alpha|\leq s(d)}
    \sum_{|\beta|\leq t(d)}
    \left\|
       D^{\alpha, \beta} f_h
    \right\|_{L^1(\Xi)}.
\end{equation}
The derivatives of the rescaled symbol satisfy
\begin{equation}
   D^{\alpha, \beta} f_h(x,v)
    =
    h^{|\alpha|}
    \bigl(
       D^{\alpha, \beta} f
    \bigr)(hx,v),
\end{equation}
and hence
\begin{equation}
\label{eq:L1-rescaled-symbol}
    \left\|
       D^{\alpha, \beta} f_h
    \right\|_{L^1(\Xi)}
    =
    h^{|\alpha|-d}
    \left\|
       D^{\alpha, \beta} f
    \right\|_{L^1(\Xi)}.
\end{equation}
Recalling that
\[
    \stnorm{1}{T}=h^d\cstnorm{1}{T},
\]
we conclude from
\eqref{eq:classical-trace-scaled}--\eqref{eq:L1-rescaled-symbol} that
\begin{equation}\label{eq:semiclassical-trace-estimate}
    \begin{aligned}
    \stnorm{1}{\OPAh(f)}
    &\leq
    C
    \sum_{|\alpha|\leq s(d)}
    \sum_{|\beta|\leq t(d)}
    h^{|\alpha|}
    \left\|
       D^{\alpha, \beta} f
    \right\|_{L^1(\Xi)}\\
    &\leq
    C
    \sum_{|\alpha|\leq s(d)}
    \sum_{|\beta|\leq t(d)}
    \left\|
       D^{\alpha, \beta} f
    \right\|_{L^1(\Xi)},
    \qquad 0<h\leq1.
\end{aligned}
\end{equation}

Here $C$ depends only on $d$ and on the regularity of $B$
required by the corresponding $h=1$ estimates, and is independent
of $h$.

	\begin{remark}
		According to the proof of the magnetic Calder\'{o}n--Vaillancourt theorem presented in \cite{iftimie2007magnetic}, $p(d)$ can be chosen as $p(d)=2\left\lfloor\frac{d}{2}\right\rfloor+2d+4$, and in the following discussion we always take $p(d)$ to be this value.
	\end{remark}

	\subsection{Estimates with the transversal gauge}\label{Section:ProofOfRegularityOfMagneticWeyl_2} Next we present a proof of the regularity results for the magnetic Weyl quantization. More specifically, we prove that if the solution $f$ of the magnetic Vlasov--Poisson equation \eqref{equation: MagneticVlasovSys} is sufficiently regular, then we can obtain uniform bounds in $\hbar$ for the magnetic Weyl quantization of $f$, which is denoted by $\brhoAf$. However, the estimates in this section will depend on the choice of gauge $A$. Here, we use the transversal gauge
    \begin{equation}
    A_B(x)
    :=
    -\int_0^1 sB(sx)x\dd s
    \end{equation}
    as our reference gauge. One can readily check that $\d A_B=B$ and that $A_B$ depends only on $B$. For simplicity of notation, we shall write $A$ instead of $A_B$.

Before stating and proving the regularity results, we establish estimates
for the derivatives of the magnetic phase
\[
    e^{-\frac{2\pi i}{h}\Gamma^A([x,y])}.
\]
The line integral along the segment from $x$ to $y$ is given by
\begin{equation}
    \Gamma^A([x,y])
    =
    \int_0^1
    (y-x)\cdot A\bigl(x+s(y-x)\bigr)\,\mathrm{d}s.
\end{equation}
Differentiating with respect to $y$, we obtain
\begin{align}
    \nabla_y\Gamma^A([x,y])
    &=
    \int_0^1
    A\bigl(x+s(y-x)\bigr)\,\mathrm{d}s
    \notag\\
    &\quad+
    \int_0^1
    s\sum_{j=1}^d
    (y_j-x_j)
    \nabla A_j\bigl(x+s(y-x)\bigr)\,\mathrm{d}s.
\end{align}
Since $A$ is the transversal gauge, for every multi-index $\gamma$,
\begin{equation}
    \snorm{\partial^\gamma A(z)}
    \leq
    C_{\gamma,B}\langle z\rangle,
\end{equation}
provided that the corresponding derivatives of $B$ are bounded.
Consequently, whenever $|\alpha|+|\beta|\geq 1$,
\begin{equation}
    \snorm{
        \partial_x^\alpha\partial_y^\beta
        \Gamma^A([x,y])
    }
    \leq
    C_{\alpha,\beta,B}
    \bigl(\langle x\rangle+\langle y\rangle\bigr)
    \langle x-y\rangle.
\end{equation}
Repeated application of the chain and product rules gives the following
estimate.

\begin{lemma}[Estimate of the magnetic phase]\label{lem:magnetic_phase_est}
    Let $k\in\mathbb N$, assume that
    $B\in W^{k,\infty}(\mathbb R^d)$, and let $A$ be the transversal
    gauge defined above. Then there exists a constant $C_{k,B}>0$ such
    that, for every $0<h\leq 1$ and all multi-indices $\alpha,\beta$
    satisfying $|\alpha|+|\beta|=k$,
    \begin{equation}
        \snorm{
            \partial_x^\alpha\partial_y^\beta
            e^{-\frac{2\pi i}{h}\Gamma^A([x,y])}
        }
        \leq
        \frac{C_{k,B}}{h^k}
        \bigl(\langle x\rangle+\langle y\rangle\bigr)^k
        \langle x-y\rangle^k.
    \end{equation}
    The constant $C_{k,B}$ depends only on $k$, $d$, and
    $\norm{B}_{W^{k,\infty}(\mathbb R^d)}$.
\end{lemma}

Next, we state the regularity estimate for the magnetic Weyl
quantization.

\begin{lemma}
\label{Lemma:RegularityofMagneticWeylTransform}
    Let $n,n_1\in 2\mathbb N_0$ and set $\sigma:=2n_1+n$.
    Assume that $B\in W^{n_1,\infty}$ and that $A$ is the
    transversal gauge.
    Then, for every $f\in H_\sigma^\sigma(\Xi)$ and every $0<h\leq1$,
    \begin{equation}
    \label{eq:mixed-magnetic-Weyl-HS}
        \stnorm{2}{
            \brhoAf\snorm{\bp}^{n_1}\snorm{x}^n
        }
        \leq
        C_{d,n,n_1,B}
        \norm{f}_{H_\sigma^\sigma(\Xi)}.
    \end{equation}
    In particular,
    \begin{align}
        \stnorm{2}{\brhoAf\snorm{\bp}^{n}}
        &\leq
        C_{d,n,B}
        \norm{f}_{H_{2n}^{2n}(\Xi)},
        \label{eq:magnetic-Weyl-momentum-HS}
        \\
        \stnorm{2}{\brhoAf\snorm{x}^{n}}
        &\leq
        C_{d,n}
        \norm{f}_{H_\sigma^\sigma(\Xi)}.
        \label{eq:magnetic-Weyl-position-HS}
    \end{align}
    The constants are independent of $h$.
\end{lemma}

\begin{remark}
    It should be noted that \eqref{eq:magnetic-Weyl-position-HS} is actually gauge-covariant. 
\end{remark}

\begin{proof}
    Let $\mathcal F_v$ denote the partial Fourier transform in the
    velocity variable:
    \[
        \widehat f(q,\eta)
        :=
        \int_{\mathbb R^d}
        e^{-2\pi i v\cdot\eta}
        f(q,v)\,\mathrm{d}v.
    \]
    Introduce the center and relative coordinates
    \[
        q:=\tfrac{1}{2}(x+y),
        \qquad
        r:=y-x\,.
    \]
    After making the change of variables $\xi=hv$ in the definition of
    the magnetic Weyl quantization, its integral kernel becomes
    \begin{equation}
    \label{eq:magnetic-Weyl-kernel-Fourier}
        K_f^A(x,y)
        =
        h^{-d}
        e^{-\frac{2\pi i}{h}\Gamma^A([x,y])}
        \widehat f\left(q,\frac{r}{h}\right).
    \end{equation}

    Since $n_1$ is even, right multiplication by
    $\snorm{\bp}^{n_1}$ gives
    \begin{equation}\label{equation:integralkernelK1}
        K_{\brhoAf\snorm{\bp}^{n_1}}(x,y)
        =
        \left(-\hbar^2\Delta_y\right)^{n_1/2}
        K_f^A(x,y).
    \end{equation}
    Consequently,
    \begin{equation}
    \label{eq:mixed-kernel}
        K_{\brhoAf\snorm{\bp}^{n_1}\snorm{x}^n}(x,y)
        =
        \snorm{y}^n
        \left(-\hbar^2\Delta_y\right)^{n_1/2}
        K_f^A(x,y).
    \end{equation}

    In the variables $(q,r)$, differentiation with respect to $y$ is
    given by
    \[
        \nabla_y
        =
        \tfrac{1}{2}\nabla_q+\nabla_r.
    \]
    Moreover,
    \[
        \nabla_r
        \widehat f\left(q,\frac{r}{h}\right)
        =
        \frac{1}{h}
        \left(\nabla_\eta\widehat f\,\right)
        \left(q,\frac{r}{h}\right).
    \]
    Thus, after multiplication by $\hbar=h/(2\pi)$, derivatives in
    $r$ produce derivatives in $\eta$ with constants independent of
    $h$, whereas derivatives in $q$ carry a nonnegative power of $h$.

    By Lemma~\ref{lem:magnetic_phase_est}, for every
    multi-index $\gamma$ with $1\leq|\gamma|\leq n_1$,
    \begin{equation}
    \label{eq:phase-bound-used-in-HS-proof}
        \snorm{
            \partial_y^\gamma
            e^{-\frac{2\pi i}{h}\Gamma^A([x,y])}
        }
        \leq
        \frac{C_{\gamma,B}}{h^{|\gamma|}}
        \bigl(\langle x\rangle+\langle y\rangle\bigr)^{|\gamma|}
        \langle x-y\rangle^{|\gamma|}.
    \end{equation}
    In center-relative coordinates,
    \[
        \langle x\rangle+\langle y\rangle
        \leq
        C\bigl(\langle q\rangle+\langle r\rangle\bigr).
    \]
    We also have
    \[
        \snorm{y}^n
        =
        \snorm{q+\frac{r}{2}}^n
        \leq
        C_n\bigl(\langle q\rangle^n+\langle r\rangle^n\bigr).
    \]

    Applying the product rule to \eqref{eq:mixed-kernel}, using
    \eqref{eq:phase-bound-used-in-HS-proof}, and then setting
    $r=h\eta$, we obtain a finite sum of terms bounded by expressions
    of the form
    \begin{equation}
    \label{eq:generic-HS-term}
        C_{d,n,n_1,B}\,
        h^a
        \snorm{q}^{m_1}
        \snorm{\eta}^{m_2}
        \snorm{
            \partial_q^\alpha
            \partial_\eta^\beta
            \widehat f(q,\eta)
        },
    \end{equation}
    where
    \[
        a\geq0,
        \qquad
        m_1+m_2+|\alpha|+|\beta|
        \leq
        2n_1+n=\sigma.
    \]
    Since $0<h\leq1$, the factor $h^a$ is uniformly bounded.

    Multiplication by $\eta$ corresponds, under the inverse Fourier
    transform, to differentiation in $v$, while differentiation in
    $\eta$ corresponds to multiplication by $v$. Therefore,
    Plancherel's theorem in the $v$ variable gives
    \begin{equation}
    \label{eq:generic-HS-term-bound}
        \left\|
            \snorm{q}^{m_1}
            \snorm{\eta}^{m_2}
            \partial_q^\alpha
            \partial_\eta^\beta
            \widehat f
        \right\|_{L^2_{q,\eta}}
        \leq
        C_{d,\sigma}
        \norm{f}_{H_\sigma^\sigma(\Xi)}.
    \end{equation}

    Finally, the change of variables
    \[
        (x,y)\longmapsto(q,r),
        \qquad
        r=h\eta,
    \]
    has Jacobian
    \[
        \mathrm{d}x\,\mathrm{d}y
        =
        h^d\,\mathrm{d}q\,\mathrm{d}\eta.
    \]
    Combining \eqref{eq:magnetic-Weyl-kernel-Fourier}--%
    \eqref{eq:generic-HS-term-bound} yields
    \[
        h^{d/2}
        \left\|
            K_{\brhoAf\snorm{\bp}^{n_1}\snorm{x}^n}
        \right\|_{L^2_{x,y}}
        \leq
        C_{d,n,n_1,B}
        \norm{f}_{H_\sigma^\sigma(\Xi)}.
    \]
    Since
    \[
        \stnorm{2}{T}
        =
        h^{d/2}\norm{T}_{\mathfrak S^2}
        =
        h^{d/2}\norm{K_T}_{L^2_{x,y}},
    \]
    this proves \eqref{eq:mixed-magnetic-Weyl-HS}.

    Taking $n=0$ gives
    \eqref{eq:magnetic-Weyl-momentum-HS}, while taking $n_1=0$ gives
    \eqref{eq:magnetic-Weyl-position-HS}. This concludes our proof.
\end{proof}
	
	We also obtain the following estimate, analogous to
\cite[Proposition 3.4]{lafleche2023strong}.

\begin{corollary}
\label{Corollary:EstimateofTrace}
    Let $n\in 2\mathbb N$ and $n_1\in 2\mathbb N_0$, with
    $n>d/2$, and set $k:=n+n_1$, $\sigma:=k+n=2n+n_1$.
    Assume that $A$ is the transversal gauge associated with $B$ and
    that $B$ has sufficiently many bounded derivatives. Then
    \begin{equation}    \label{eq:weighted-trace-magnetic-Weyl}
        \begin{aligned}
        h^d
        \operatorname{Tr}
        \left(
            \snorm{\brhoAf}\snorm{\bp}^{n_1}
        \right)
        &\leq
        C_{d,n,n_1,B}
        \Bigl(
            \norm{
                \langle v\rangle^k
                \langle x\rangle^n f
            }_{L^2(\Xi)}
            +
            h^k
            \norm{
                \langle x\rangle^n
                \nabla_{x,v}^k f
            }_{L^2(\Xi)}
       \\
        &\hspace{8em}
            +
            h^n
            \norm{
                \langle v\rangle^k
                \nabla_{x,v}^n f
            }_{L^2(\Xi)}
            +
            h^\sigma
            \norm{
                \nabla_{x,v}^\sigma f
            }_{L^2(\Xi)}
        \Bigr)
       \\
        &\leq
        C_{d,n,n_1,B}
        \norm{f}_{H^\sigma_\sigma(\Xi)}.
    \end{aligned}
    \end{equation}
    Here, the constant $C_{d,n,n_1,B}$ is independent of $h\in(0,1]$.
\end{corollary}

\begin{proof}
    Since $\snorm{\bp}^{n_1}$ is positive, the trace inequality for
    products of operators gives
    \begin{equation}
        \operatorname{Tr}
        \left(
            \snorm{\brhoAf}\snorm{\bp}^{n_1}
        \right)
        \leq
        \cstnorm{1}{
            \brhoAf\snorm{\bp}^{n_1}
        }.
    \end{equation}

    Define
    \[
        \omega_n(z):=(1+\snorm{z}^n)^{-1},
        \qquad
        \boldsymbol m_n
        :=
        \omega_n(x)\omega_n(\bp).
    \]
    Since $n>d/2$, we have $\omega_n\in L^2(\mathbb R^d)$.
    The Kato--Seiler--Simon inequality therefore gives
    \begin{equation}
    \label{eq:mn-HS-bound}
        \cstnorm{2}{\boldsymbol m_n}
        \leq
        C_{d,n}h^{-d/2}.
    \end{equation}

    Since $\boldsymbol m_n^{-1}
        =
        \bigl(1+\snorm{\bp}^n\bigr)
        \bigl(1+\snorm{x}^n\bigr),$
    then  the Schatten Hölder's inequality gives
    \begin{align}
        \cstnorm{1}{
            \brhoAf\snorm{\bp}^{n_1}
        }
        &=
        \cstnorm{1}{
            \brhoAf\snorm{\bp}^{n_1}
            \boldsymbol m_n^{-1}\boldsymbol m_n
        }
        \leq
        \cstnorm{2}{
            \brhoAf\snorm{\bp}^{n_1}
            \boldsymbol m_n^{-1}
        }
        \cstnorm{2}{\boldsymbol m_n}.
    \end{align}
    Moreover, notice we have
    \begin{align*}
        \brhoAf\snorm{\bp}^{n_1}\boldsymbol m_n^{-1}
        &=
        \brhoAf\snorm{\bp}^{n_1}
        \bigl(1+\snorm{\bp}^n\bigr)
        \bigl(1+\snorm{x}^n\bigr)
        \notag\\
        &=
        \brhoAf\snorm{\bp}^{n_1}
        +
        \brhoAf\snorm{\bp}^{n_1}\snorm{x}^n
        \notag\\
        &\quad+
        \brhoAf\snorm{\bp}^{k}
        +
        \brhoAf\snorm{\bp}^{k}\snorm{x}^n.
    \end{align*}
    Using \eqref{eq:mn-HS-bound} and $\cstnorm{2}{T}
        =
        h^{-d/2}\stnorm{2}{T}$, 
    we obtain
\begin{multline*}
	h^d
	\operatorname{Tr}
	\left(
		\snorm{\brhoAf}\snorm{\bp}^{n_1}
	\right)
	\leq
	C_{d,n}
	\Bigl(
		\stnorm{2}{\brhoAf\snorm{\bp}^{n_1}}
		+
		\stnorm{2}{\brhoAf\snorm{\bp}^{n_1}\snorm{x}^n}
		\\
		+
		\stnorm{2}{\brhoAf\snorm{\bp}^{k}}
		+
		\stnorm{2}{\brhoAf\snorm{\bp}^{k}\snorm{x}^n}
	\Bigr)\,.
\end{multline*}
    Applying Lemma
    \ref{Lemma:RegularityofMagneticWeylTransform} to these four terms
    and using $k+n=\sigma$ gives
\begin{multline*}
	h^d
	\operatorname{Tr}
	\left(
		\snorm{\brhoAf}\snorm{\bp}^{n_1}
	\right)
	\leq
	C_{d,n,n_1,B}
	\Bigl(
		\norm{\langle v\rangle^k\langle x\rangle^n f}_{L^2(\Xi)}
		+
		h^k\norm{\langle x\rangle^n\nabla_x^k f}_{L^2(\Xi)}
		\\
		+
		h^n\norm{\langle v\rangle^k\nabla_v^n f}_{L^2(\Xi)}
		+
		h^\sigma\norm{\nabla_x^k\nabla_v^n f}_{L^2(\Xi)}
	\Bigr)\,.
\end{multline*}
    Since $0<h\leq1$, the right-hand side is bounded by
    $C_{d,n,n_1,B}\norm{f}_{H^\sigma_\sigma(\Xi)}$. This proves
    \eqref{eq:weighted-trace-magnetic-Weyl}.
\end{proof}

\begin{corollary}
\label{Corollary:EstimateofDiag}
    Let $n,n_1\in 2\mathbb N$, with $n>d/2$, and set
    \[
        k:=n+n_1,
        \qquad
        \sigma:=k+n=2n+n_1,
        \qquad
        N_0
        :=
        2\left\lfloor\frac d2\right\rfloor+2d+4.
    \]
    Assume that $A$ is the transversal gauge associated with $B$ and
    that $B$ has sufficiently many bounded derivatives. Then, for every
    \[
        1\leq p\leq1+\frac{n_1}{d},
    \]
    we have
    \begin{equation}
    \label{eq:diagonal-gradient-estimate}
        \lnorm{p}{
            \operatorname{diag}
            \left(
                \snorm{\boldsymbol{\nabla}_v\brhoAf}
            \right)
        }
        \leq
        C_{d,n,n_1,B}
        \norm{\nabla_vf}_{
            H^\sigma_\sigma(\Xi)
            \cap W^{2N_0,\infty}(\Xi)
        }.
    \end{equation}
    The constant $C_{d,n,n_1,B}$ is independent of $h\in(0,1]$.
\end{corollary}

\begin{proof}
	Set $T:=\boldsymbol{\nabla}_v\brhoAf$. Differentiation in the
	velocity variable corresponds exactly to commutation with the
	position operator, so that $T=\brho^A_{\nabla_v f}$. The following
	estimates are applied componentwise to $T$.

	Let
	\[
		p_*:=1+\frac{n_1}{d}.
	\]
	The quantum kinetic interpolation inequality
	\cite[Theorem 6]{lafleche2019propagation}, applied to $\snorm{T}$,
	gives
	\begin{equation}
		\label{eq:quantum-kinetic-interpolation-endpoint}
		\lnorm{p_*}{\operatorname{diag}(\snorm{T})}
		\leq
		C_{d,n_1}
		\left[
			h^d
			\operatorname{Tr}
			\left(
				\snorm{T}\snorm{\bp}^{n_1}
			\right)
		\right]^{1/p_*}
		\stnorm{\infty}{T}^{1-1/p_*}.
	\end{equation}

	Applying Corollary \ref{Corollary:EstimateofTrace} to $\nabla_v f$
	gives
	\begin{equation}
		\label{eq:trace-gradient-bound}
		h^d
		\operatorname{Tr}
		\left(
			\snorm{T}\snorm{\bp}^{n_1}
		\right)
		\leq
		C_{d,n,n_1,B}
		\norm{\nabla_v f}_{H^\sigma_\sigma(\Xi)}.
	\end{equation}
	The magnetic Calder\'on--Vaillancourt theorem gives
	\begin{equation}
		\label{eq:operator-gradient-bound}
		\stnorm{\infty}{T}
		\leq
		C_{d,B}
		\norm{\nabla_v f}_{W^{2N_0,\infty}(\Xi)}.
	\end{equation}
	Combining
	\eqref{eq:quantum-kinetic-interpolation-endpoint}--%
	\eqref{eq:operator-gradient-bound}, we obtain
	\[
		\lnorm{p_*}{\operatorname{diag}(\snorm{T})}
		\leq
		C_{d,n,n_1,B}
		\norm{\nabla_v f}{
			H^\sigma_\sigma(\Xi)
			\cap W^{2N_0,\infty}(\Xi)
		}.
	\]

	At the other endpoint,
	\begin{align}
		\lnorm{1}{\operatorname{diag}(\snorm{T})}
		&=
		h^d\operatorname{Tr}(\snorm{T})
		\notag\\
		&\leq
		C_{d,n,B}
		\norm{\nabla_v f}_{H^{2n}_{2n}(\Xi)}
		\notag\\
		&\leq
		C_{d,n,n_1,B}
		\norm{\nabla_v f}_{H^\sigma_\sigma(\Xi)}.
	\end{align}
	Interpolation between $L^1$ and $L^{p_*}$ proves
	\eqref{eq:diagonal-gradient-estimate} for every
	$p\in[1,p_*]$.
\end{proof}

	\section{Semiclassical Limit: Proof of Theorem~\ref{Theorem:Semiclassicallimit}}\label{Section:ProofofSemiclassicalLimit}
	The aim of this section is to prove Theorem \ref{Theorem:Semiclassicallimit}. The main strategy is to compare the difference between the solution $\brho$ of the magnetic Hartree--Fock equation and the magnetic Weyl quantization of the solution $f$ of the Cauchy problem \eqref{equation: MagneticVlasovSys}.
	
The magnetic Weyl quantization of the Hamiltonian
$\frac{1}{2}\snorm{v}^2$ satisfies, in the sense of distributions,
\begin{equation*}
	\OPAh\lb\tfrac{1}{2}\snorm{v}^2\rb
	=
	H^A_0
	:=
	\tfrac{1}{2}(\opp-\bA)^2\,.
\end{equation*}
The proof is given in
Lemma~\ref{Lemma:TransformOfHamiltonian} of
Appendix~\ref{Section:PropertiesofMagneticWeylTransform}.

We introduce the two-parameter unitary propagator
$\mathcal{U}(t,s)$ satisfying, for $t>s$,
\begin{align*}
	i\hbar\,\partial_t\mathcal{U}(t,s)
	&=
	H^A_{\brho}(t)\mathcal{U}(t,s)\,,
	&
	\mathcal{U}(s,s)
	&=
	\Id\,.
\end{align*}
Denote by $\brho$ the
solution of the magnetic Hartree--Fock equation. Then
\begin{multline}\label{equation:ControlOfErrorHartree}
	i\hbar\,\partial_t
	\lb
	\mU^*(\brho-\brhoAf)\mU
	\rb
	=
	\mU^*
	\left[K*(\rho-\rho_f),\brhoAf\right]
	\mU
	\\
	-h^d\mU^*
	\left[\bXr,\brhoAf\right]
	\mU
	+\mU^*{\sf B}_t\mU
	-\hbar^2\mU^*\OPAh(\widetilde{R})\mU\,.
\end{multline}
Here $\widetilde{R}$ satisfies
\begin{equation*}
	\snorm{
	\partial_x^\alpha\partial_v^\beta
	\widetilde{R}(x,v)
	}
	\leq
	C_{\alpha,\beta}
	\lal x\ral^{-d-2-\snorm{\alpha}}
	\lal v\ral^{-d-2-\snorm{\beta}}
\end{equation*}
for all multi-indices $\alpha,\beta\in\mathbb{N}^d$ allowed by the
assumed finite regularity. The proof of
\eqref{equation:ControlOfErrorHartree} is given in
Section~\ref{Subsection:ExplicitCalculation} of the appendix. The error term ${\sf B}_t$ is defined by
\begin{align}\label{equation:DefinitionOfErrorTerm}
	{\sf B}_t(x,y)
	={}&
	\left\{
	(K*\rho_f^A)(x)
	-(K*\rho_f^A)(y)
	-(x-y)\cdot
	(\nabla K*\rho_f^A)
	\lb\tfrac{x+y}{2}\rb
	\right\}
	\brhoAf(x,y)\,.
\end{align}

	Take the classical Schatten norm of both sides we have
	\begin{equation*}
		\begin{aligned}
			\stnorm{p}{\brho-\brhoAf}\le& \stnorm{p}{\brho^{\mathrm{in}}-\brho^{A}_{f^{\mathrm{in}}}}+\frac{1}{\hbar}\int_0^t\stnorm{p}{\left[K*(\rho-\rho_f),\brhoAf\right]}\,\mathrm{d} s\\&+\hbar\int^t_0 \stnorm{p}{\OPAh(\widetilde{R})}\,\mathrm{d} s+\frac{1}{\hbar}\int^t_0\stnorm{p}{\left[h^d\bxhf,\brhoAf\right]}\,\mathrm{d} s+\frac{1}{\hbar}\int^t_0\stnorm{p}{\sfB_s}\,\mathrm{d} s\,.
		\end{aligned}
	\end{equation*}
	Therefore, to study the semiclassical limit of the magnetic Hartree--Fock equation, it suffices to estimate the remainder term $\OPAh(\widetilde{R})$, the error term $\sf{B}_t$, the commutator term $[K(\cdot-z),\brhoAf(s)]$ and the exchange term $\left[h^d\bxhf,\brhoAf\right]$. We first estimate the remainder term, which turns out to be a straightforward corollary of interpolation.

	\bigskip
	\noindent\textbf{Step 1.} \textit{Estimate of the remainder term.}
	According to \eqref{eq:semiclassical-Schatten-infinity-estimate} and \eqref{eq:semiclassical-trace-estimate},  we know that the remainder term $h^2\widetilde{R}$ contributes optimal convergence rate: for $1\le p\le \infty$, by applying interpolation we have
	\begin{equation*}
		\stnorm{p}{h^2\OPAh(\widetilde{R})}\le Ch^2\lb\sum_{\snorm{\alpha}\le s(d)}\sum_{\snorm{\beta}\le t(d)}\lnorm{1}{\partial^\alpha_x\partial^\beta_v\widetilde{R}}\rb^{\frac{1}{p}}\lb \sum_{\snorm{\alpha},\snorm{\beta}\le p(d)}\norm{\partial^\alpha_x\partial^\beta_v\widetilde{R}}_{L^{\infty}}\rb^{1-\frac{1}{p}}\,.
	\end{equation*} 
	For the sake of convenience, we define the function $\sf{R}_p(t)$ as
	\begin{equation*}
		\begin{aligned}
			{\sf{R}_p(t)}:=\lb\sum_{\snorm{\alpha}\le s(d)}\sum_{\snorm{\beta}\le t(d)}\lnorm{1}{\partial^\alpha_x\partial^\beta_v\widetilde{R}}\rb^{\frac{1}{p}}\lb \sum_{\snorm{\alpha},\snorm{\beta}\le p(d)}\norm{\partial^\alpha_x\partial^\beta_v\widetilde{R}}_{L^{\infty}}\rb^{1-\frac{1}{p}}\,.
		\end{aligned}
	\end{equation*}
	and consequently
	\begin{equation*}
		\begin{aligned}
			\hbar\int^t_0 \stnorm{p}{\OPAh(\widetilde{R})}\,\mathrm{d} s\le C\hbar\int^t_0{\sf{R}}_p(s)\,\mathrm{d} s\,.
		\end{aligned}
	\end{equation*}
	\noindent\textbf{Step 2.} \textit{Estimate of the error term.}
    Estimate of the error term is given by following proposition:
	\begin{proposition}\label{Proposition:ControlOfErrorTerm}
		For any $1\le p\le\infty$, we have the following estimate:
		\begin{align*}
			&\stnorm{p}{\sf{B}_t}\le C_{d}\hbar^2\norm{\rho_f}_{H^{\lfloor\frac{d}{2}\rfloor+1}}\left(\norm{f(t,\cdot,\cdot)}_{W^{4\left\lfloor\frac{d}{2}\right\rfloor+4d+10,\infty}(\Xi)}\right)^{1-\frac{1}{p}}\widetilde{B}^{\frac{1}{p}}_d\,.
		\end{align*}
		The definition of $\widetilde{B}_d$ is
		\begin{equation}
			\begin{aligned}
				\widetilde{B}_d&:=C_{d} 
				\left(
				h^{2\lceil\frac{d+2}{4}\rceil}\norm{\nabla^{2\lceil\frac{d+2}{4}\rceil}_{x,v} f}_{L^2(\Xi)}+\norm{{\langle x\rangle}^{2\lceil\frac{d+2}{4}\rceil}f}_{L^2(\Xi)}+\norm{{\langle v \rangle}^{2\lceil\frac{d+2}{4}\rceil}f }_{L^2(\Xi)} \right)\,.
			\end{aligned}
		\end{equation}
		Here, $C_d$ depends only on the dimension $d$.
	\end{proposition}
	\begin{proof}
        By repeating the calculation presented in the Proposition 4.7 of \cite{lafleche2023strong} we can obtain
		\begin{align*}
			\stnorm{p}{\sfB_t}\le C_{d,n}\hbar^2\norm{\rho_f}_{H^{{n}}}\stnorm{p}{\brho^{A}_{\nabla_v^2f}}\,.
		\end{align*}
		For the sake of simplicity, we shall take ${n}=\lfloor\frac{d}{2}\rfloor+1$. When $p=\infty$, the magnetic Calder\'{o}n--Vaillancourt theorem gives that
		\begin{align*}
			\stnorm{\infty}{\brho^{A}_{\nabla_v^2f}}\le C \sum_{\snorm{\alpha}\le p(d),\snorm{\beta}\le p(d)+2}\lnorm{\infty}{\partial^\alpha_x\partial^\beta_v f(t,x,v)}\le C\norm{f(t,\cdot,\cdot)}_{W^{4\left\lfloor\frac{d}{2}\right\rfloor+4d+10,\infty}(\Xi)}\,.
		\end{align*}
		When $p=1$ we use $\nabla_v^2f$ to replace the $f$ in $\brhoAf$ and set $n=2\left\lceil\frac{d+2}{4}\right\rceil$, $n_1=0$ (hence $n$ is the smallest positive even integer which satisfies $n>\frac{d}{2}$) and apply Lemma \ref{Lemma:RegularityofMagneticWeylTransform}, then we have
		\begin{align*}
			\stnorm{1}{\brho^{A}_{\nabla_v^2f}}\le& C_{d}\sum_{\snorm{\alpha+\beta+\gamma}\le2\left\lceil\frac{d+2}{4}\right\rceil}\sum_{\snorm{\delta}\le\snorm{\beta}}h^{2\left\lceil\frac{d+2}{4}\right\rceil-\snorm{\alpha}-\snorm{\beta}}\lnorm{2}{\partial^\delta_v(v^\alpha\partial^\gamma_xf(t,x,v))}
			\\&+C_{d}\sum_{\snorm{\alpha+\beta}\le2\left\lceil\frac{d+2}{4}\right\rceil}h^{\snorm{\beta}}\lnorm{2}{\snorm{x}^{\snorm{\alpha}}\partial^\beta_v f(t,x,v)}
			\\\le& C_{d} h^{2\lceil\frac{d+2}{4}\rceil}
			\left(
			\norm{\nabla^{2\lceil\frac{d+2}{4}\rceil}_x f}_{L^2(\Xi)}
			+\norm{\nabla^{2\lceil\frac{d+2}{4}\rceil}_v f}_{L^2(\Xi)}
			\right)
			\\&+C_d\left(\norm{{\langle x\rangle}^{2\lceil\frac{d+2}{4}\rceil}f}_{L^2(\Xi)}+\norm{{\langle v \rangle}^{2\lceil\frac{d+2}{4}\rceil}f }_{L^2(\Xi)} \right)\,.
		\end{align*}
		Then the interpolation gives the desired result.
	\end{proof}

	\noindent\textbf{Step 3.} \textit{Estimate of the commutator term.}
	By applying the method introduced in the work \cite{lafleche2023strong} we have the following estimate of commutator term:
	\begin{proposition}\label{Proposition:ControlOfCommutator}
		We have the following estimates:
		\begin{enumerate}
			\item For any $\epsilon \in (0,d-1]$ and any $\widetilde{\epsilon}\in (0,\frac{\epsilon}{2d})$ there exists a constant $C>0$ such that
			\begin{equation*}
				\stnorm{1}{\left[K(\cdot-z),\brhoAf\right]}\le Ch\lnorm{d-\epsilon}{\operatorname{diag}\left(\snorm{\boldsymbol{\nabla}_v \brhoAf}\right)}^{\frac{1}{2}+\widetilde{\epsilon}}\lnorm{d+\epsilon}{\operatorname{diag}\left(\snorm{\boldsymbol{\nabla}_v\brhoAf}\right)}^{\frac{1}{2}-\widetilde{\epsilon}}\,.
			\end{equation*}
			\item For $1\le p<\frac{d}{d-1}$ and $\epsilon\in(0,q-1)$, $n>d-1$ there exists a constant $C>0$ such that
			\begin{equation*}
				\stnorm{p}{\left[K(\cdot-z),\brhoAf\right]}\le Ch\stnorm{q+\epsilon}{\boldsymbol{\nabla}_v\brhoAf \boldsymbol{m}_n}^{\frac{1}{2}+\widetilde{\epsilon}}\stnorm{q-\epsilon}{\boldsymbol{\nabla}_v\brhoAf \boldsymbol{m}_n}^{\frac{1}{2}-\widetilde{\epsilon}}\,.
			\end{equation*}
			Where $\widetilde{\epsilon}=\frac{\epsilon}{q}$, $\boldsymbol{m}_n=1+\snorm{\bp}^n$ and with $\frac{1}{p}=\frac{1}{q}+\frac{d-1}{d}$.
		\end{enumerate}
	\end{proposition}

	\noindent\textbf{Step 4.} \textit{Estimate of the exchange term.}
	In this step, we apply results from quantum function spaces. For more general details on this topic, we refer to \cite{lafleche2024quantum} and \cite{simon2005trace}.
	
	We first present some estimates on exchange terms. Recall the following fractional Hardy–Rellich inequality. For the proof, see Theorem 1.72 and Remark 1.73 of \cite{bahouri2011fourier}.
	\begin{theorem}\label{Theorem:HardyRellichInequality}
		For a function $f$ defined on $\rr^d$ and an index $0\le s<\frac{d}{2}$, we have
		\begin{equation}
			\lnorm{2}{\frac{f}{\snorm{x}^s}}\le C_{d,s}\norm{f}_{\dot{H}^s}\,.
		\end{equation}
	\end{theorem}
	
	For the sake of simplicity, in the following discussion we always set $d=3$ and $\Xi=\rr^3\times\rr^3$. Recall that $\bXr$ is the operator with the time-dependent integral kernel $\bXr(x,y)=K(x-y)\brho(x,y)$. Here we use $\brho(x,y)$ to represent the integral kernel of $\brho$. Then, by applying the Hardy–Rellich inequality, we obtain the following estimate \cite{lafleche2023strong}. Again, for the sake of clarity, we shall give a brief proof of this estimate due to the slight differences in notation between our work and \cite{lafleche2023strong}.
	\begin{proposition}
		Let $\mG(x)=\frac{1}{4\pi \snorm{x}}$. $\brho$ is a positive trace class operator. Then we have
		\begin{equation*}
			\stnorm{1}{h^3\bXr\brho}\le Ch^2\stnorm{2}{\snorm{\bp}^{\frac{1}{2}}\brho}^2\,.
		\end{equation*}
		Suppose $\brho_2$ is another trace class operator. Then when $p\in [1,2]$ for each $q$ such that $\frac{1}{q}=\frac{1}{p}-\frac{1}{2}$ we have
		\begin{equation*}
			\stnorm{p}{h^3\bXr\brho_2}\le Ch^2\stnorm{q}{\brho_2}\stnorm{2}{\snorm{\opp-\bA}\brho}
		\end{equation*}
		When $p\in[2,\infty]$ we have
		\begin{equation*}
			\stnorm{p}{h^3\bXr\brho_2}\le Ch^{2-3\left(\frac{1}{2}-\frac{1}{p}\right)}\stnorm{\infty}{\brho_2}\stnorm{2}{\snorm{\opp-\bA}\brho}\,.
		\end{equation*}
	\end{proposition}
	\begin{proof}
		By changing variables and applying the Hardy--Rellich inequality we can obtain
		\begin{equation*}
			\begin{aligned}
				\operatorname{Tr}\left(\bXr\brho\right)&=\iint_{\Xi}\frac{1}{4\pi \snorm{x-y}}\snorm{\brho(x,y)}^2\,\mathrm{d} x\mathrm{d}y=\iint_{\Xi}\frac{1}{4\pi\snorm{z}}\snorm{\brho(z+y,y)}^2\,\mathrm{d} z\mathrm{d}y\\&\le C\iint_{\Xi}\snorm{\Delta_z^{\frac{1}{2}}}\snorm{\brho(z+y,y)}^2\,\mathrm{d} z\mathrm{d}y\le Ch^{-1}\norm{\snorm{\bp}^{\frac{1}{2}}\brho}_{2}^2\,.
			\end{aligned}
		\end{equation*}
		Thus $\stnorm{1}{h^3\bXr\brho}\le Ch^2\stnorm{2}{\snorm{\bp}^{\frac{1}{2}}\brho}^2$. Now let $\bbxx$ be the exchange operator with kernel $\mG^2$. Then by applying diamagnetic inequality (see Theorem \ref{Theorem:DiamagneticInequality}) we have
		\begin{equation*}
			\begin{aligned}
				\operatorname{Tr}\left(\bbxx\brho\right)=\iint_{\Xi}\snorm{\mG(x-y)}^2\snorm{\brho(x,y)}^2\,\mathrm{d} x \mathrm{d}y\le C\frac{1}{h^2}\iint_{\Xi}\snorm{\snorm{\opp-\bA}\brho}^2=\frac{C}{h^2}\norm{\snorm{\opp-\bA}\brho}^2_{2}\,.
			\end{aligned}
		\end{equation*}
		For $p\in[1,2]$ the H\"{o}lder's inequality of Schatten norm gives that
		\begin{equation*}
			\begin{aligned}
				\norm{\bXr\brho_2}_{p}\le\norm{\brho_2}_{q}\norm{\bXr}_{2}&=h^{-\frac{3}{q}}\stnorm{q}{\brho_2}\left(\operatorname{Tr}\left(\bbxx\brho\right)\right)^\frac{1}{2}\le Ch^{-\frac52-\frac{3}{q}}\stnorm{q}{\brho_2}\stnorm{2}{\snorm{\opp-\bA}\brho}\,.
			\end{aligned}
		\end{equation*}
		When $p\in[2,\infty]$ we have $\norm{\bXr\brho_2}_{p}\le \norm{\bXr\brho_2}_{2}$ and consequently
		\begin{equation*}
			\stnorm{p}{h^3\bXr\brho_2}\le h^{\frac{3}{p}-\frac{3}{2}}\stnorm{2}{h^3\bXr\brho_2}\le Ch^{2-3\left(\frac{1}{2}-\frac{1}{p}\right)}\stnorm{\infty}{\brho_2}\stnorm{2}{\snorm{\opp-\bA}\brho}\,,
		\end{equation*}
		and the proof is completed.
	\end{proof}
	
	\noindent\textbf{Step 5.} \textit{Proof of the $\mathcal{L}^p$ semiclassical limit.}
	Now we gather the aforementioned results to finish the proof of semiclassical limit. Let $d=3$. For $1\le p<\frac{3}{2}$ the self adjoint property gives that 
	\begin{equation}\label{Estimate:EstimateOfExchangeTerm}
		\begin{aligned}
			\stnorm{p}{\left[h^3\bxhf, \brhoAf\right]}\le 2\stnorm{p}{h^3\bxhf\brhoAf}\,.
		\end{aligned}
	\end{equation}
	
	First we take $p=1$. From Corollary \ref{Corollary:EstimateofDiag}, Proposition \ref{Proposition:ControlOfErrorTerm} and Proposition \ref{Proposition:ControlOfCommutator} we have
	\begin{equation*}
		\begin{aligned}
			\stnorm{1}{\left(\brho-\brhoAf\right)(t)}\le&\stnorm{1}{{\brho}^{\mathrm{in}}-\brho^{A}_{f^{\mathrm{in}}}}+C\hbar\int^t_0\norm{\rho_f}_{H^2}\widetilde{B}_3d\tau+C\hbar\int^t_0{\sf{R}}_p(\tau)\,\mathrm{d}\tau\\
			&+C\int^t_0\stnorm{1}{\left(\brho-\brhoAf\right)(\tau)}\norm{\nabla_v f}_{H^{16}_{16}(\Xi)\cap W^{26,\infty}(\Xi)}\,\mathrm{d}\tau\\
			&+C\hbar\int^t_0\stnorm{\infty}{\brhoAf}\stnorm{2}{\snorm{\opp-A}\brho}\dd\tau\\
			&+C\hbar\int^t_0\stnorm{1}{\left(\brho-\brhoAf\right)(\tau)}\stnorm{2}{\snorm{\opp-A}\brho}\dd\tau\\
			&+C\hbar\int^t_0\stnorm{\infty}{\brho^{\mathrm{in}}}\stnorm{2}{\snorm{\opp-A}\brho}\dd\tau\,.
		\end{aligned}
	\end{equation*}
	Then Cauchy-Schwarz inequality and Gr\"{o}nwall's inequality give that
	\begin{equation}\label{Estimate:L1Schattenestimate}
		\begin{aligned}
			\stnorm{1}{\left(\brho-\brhoAf\right)(t)}\le \left(\stnorm{1}{{\brho}^{\mathrm{in}}-\brho^{A}_{f^{\mathrm{in}}}}+C_1(t)\hbar\right)\exp(E_1(t))\,.
		\end{aligned}
	\end{equation}
	The definitions of $C_1(t)$ and $E_1(t)$ are as follows:
	\begin{equation*}
		\begin{aligned}
			C_1(t)=&C\int^t_0\norm{f}_{H^3_2(\Xi)}\widetilde{B}_3d\tau+C\int_0^t \left(\stnorm{\infty}{\brho^{\mathrm{in}}}
			+\norm{f}_{H^{17}_{16}\cap W^{27,\infty}(\Xi)}\right)
			\stnorm{2}{\snorm{\opp-\bA}\brho}\,\mathrm{d}\tau\\
            &+C\int^t_0{\sf{R}}_{p}(\tau)\,\mathrm{d}\tau\,,\\
			E_1(t)=&\int^t_0\left(\norm{f}_{H^{17}_{16}(\Xi)\cap W^{27,\infty}(\Xi)}+\hbar\stnorm{2}{\snorm{\opp-\bA}\brho}\right)\,\mathrm{d}\tau\,.
		\end{aligned}
	\end{equation*}
	Where we apply Corollary \ref{Corollary:EstimateofDiag}, we take $\epsilon=2-\delta$, $n=2$ and $n_1=12$, where $\delta>0$ is sufficiently small. This is why we take $m$ as an integer large enough in the condition of the initial data. We also need to ensure that the finiteness of $\norm{\rho_f}_{H^2}$ and $\widetilde{B}_3$ can be propagated. Indeed,
	\begin{equation*}
		\lnorm{2}{\nabla^2\rho_f}\le C\norm{{\langle v\rangle}^2\nabla^2_xf}_{L^2(\Xi)}\le C\norm{f}_{H^3_2(\Xi)}\,,
	\end{equation*}
	and the definition of $\widetilde{B}_d$ gives
	\begin{equation*}
		\begin{aligned}
			\widetilde{B}_3:=C\left( h^{4}
			\left(
			\norm{\nabla^{4}_x f}_{L^2(\Xi)}
			+\norm{\nabla^{4}_v f}_{L^2(\Xi)}
			\right)
			+\norm{{\langle x\rangle}^{4}f}_{L^2(\Xi)}+\norm{{\langle v \rangle}^{4}f }_{L^2(\Xi)} \right)\,.
		\end{aligned}
	\end{equation*}
	Hence, the propagation of finiteness of $\widetilde{B}_3$ and $\norm{\rho_f}_{H^2}$ can be guaranteed by Theorem \ref{Theorem:PropagationOfRegularity}. Once we obtain the estimate of $\mathcal{L}^1$ we just take $1\le p<\frac{3}{2}$ and $\frac{1}{p}=\frac{1}{q}+\frac{1}{2}$ and we have following estimate by applying interpolation:
	\begin{equation}\label{Estimate:PreLpSchattemestimate}
		\begin{aligned}
			\stnorm{p}{\brho-\brhoAf}\le& \stnorm{p}{\brho^{\mathrm{in}}-\brho^{A}_{f^{\mathrm{in}}}}+C\hbar\int^t_0\norm{\rho_f}_{H^2}\left(\norm{f(\tau,\cdot,\cdot)}_{W^{27,\infty}(\Xi)}\right)^{1-\frac{1}{p}}\widetilde{B}^{\frac{1}{p}}_3\,\mathrm{d}\tau\\
			&+C\hbar\int^t_0{\sf{R}}_{p}(\tau)\,\mathrm{d}\tau\\
			&+C\int^t_0\stnorm{1}{\left(\brho-\brhoAf\right)(\tau)}\norm{f}_{H^{17}_{16}(\Xi)\cap W^{27,\infty}(\Xi)}\,\mathrm{d}\tau\\
			&+C\hbar \int_0^t \left(\stnorm{\infty}{\brho^{\mathrm{in}}}+\stnorm{\infty}{\brho^{A}_{f^{\mathrm{in}}}}\right)^{1-\frac{p}{q}}\stnorm{p}{\brho-\brhoAf}^{\frac{p}{q}}
			\stnorm{2}{\snorm{\opp-\bA}\brho}\,\mathrm{d}\tau\,.
		\end{aligned}
	\end{equation}
	Here we applied the second part of Proposition \ref{Proposition:ControlOfCommutator}. In this case $q>3$. By the isometry relation given by Proposition \ref{Proposition:IsometryRelation} and the magnetic Calder\'{o}n--Vaillancourt theorem, we obtain the following inequality by interpolation:
	\begin{equation}
		\stnorm{q}{\brhoAf}\le C{\norm{f}^{1-\frac{2}{q}}_{W^{24,\infty}(\Xi)}}\lnorm{2}{f}^{\frac{2}{q}}\,.
	\end{equation}
	Recall formula \eqref{equation:integralkernelK1} gives the form of the integral kernel of the operator $\boldsymbol{\nabla}_v\brhoAf\snorm{\bp}^n$. Therefore we utilize Young's inequality and estimates given by Lemma \ref{Lemma:RegularityofMagneticWeylTransform} to obtain
	\begin{equation*}
		\begin{aligned}
			\stnorm{q+\epsilon}{\boldsymbol{\nabla}_v\brhoAf \boldsymbol{m}_n}^{\frac{1}{2}+\widetilde{\epsilon}}\stnorm{q-\epsilon}{\boldsymbol{\nabla}_v\brhoAf \boldsymbol{m}_n}^{\frac{1}{2}-\widetilde{\epsilon}}\le C\norm{\nabla_v f}_{H^{16}_{16}(\Xi)\cap W^{26,\infty}(\Xi)}\,.
		\end{aligned}
	\end{equation*}
	Which explains the way we control diagonal term. Now bring \eqref{Estimate:L1Schattenestimate} to \eqref{Estimate:PreLpSchattemestimate} and apply Cauchy-Schwarz inequality we have
	\begin{equation*}
		\begin{aligned}
			\stnorm{p}{\brho-\brhoAf}\le \left(\stnorm{p}{\brho^{\mathrm{in}}-\brho^{A}_{f^{\mathrm{in}}}}+C_2(t)\hbar\right)E_2(t)
		\end{aligned}
	\end{equation*}
	The definitions of $C_2(t)$ and $E_2(t)$ are given by
	\begin{equation*}
		\begin{aligned}
			C_2(t):=&C\int^t_0\norm{\rho_f}_{H^2}\left(\norm{f(\tau)}_{W^{27,\infty}(\Xi)}\right)^{1-\frac{1}{p}}\widetilde{B}^{\frac{1}{p}}_3\,\mathrm{d}\tau\\&+C\norm{f}_{L^\infty([0, T]; H^{17}_{16}(\Xi)\cap W^{27,\infty}(\Xi))}\int^t_0C_1(\tau)\,\mathrm{d}\tau\\&+C\int_0^t \left(\stnorm{\infty}{\brho^{\mathrm{in}}}+\stnorm{\infty}{\brho^{A}_{f^{\mathrm{in}}}}\right)
			\stnorm{2}{\snorm{\opp-\bA}\brho}\,\mathrm{d}\tau+C\int^t_0{\sf{R}}_p(\tau)\,\mathrm{d}\tau\,,\\
			E_2(t):=&\left(1+\int^t_0\exp(\hbar E_1(\tau))\,\mathrm{d}\tau\right)\exp\left(\hbar\int^t_0\stnorm{2}{\snorm{\opp-\bA}\brho}\,\mathrm{d}\tau\right)\,.
		\end{aligned}
	\end{equation*} 
	Finally let $r\in[p,\infty)$ we have
	\begin{equation*}
		\begin{aligned}
			\stnorm{r}{\brho-\brhoAf}\le& \left(\stnorm{\infty}{\brho^{\mathrm{in}}}+\norm{f}_{L^\infty{[0, T]; (H^{17}_{16}(\Xi)\cap W^{27,\infty}(\Xi))}}\right)^{1-\frac{p}{r}}
			\\
			&\times \left(\left(\stnorm{p}{\brho^{\mathrm{in}}-\brho^{A}_{f^{\mathrm{in}}}}+C_2(t)\hbar\right)E_2(t)\right)^{\frac{p}{r}}\,.
		\end{aligned}
	\end{equation*}
	Hence the proof of Theorem \ref{Theorem:Semiclassicallimit} is completed.

	\appendix
	
	\section{Calculations and properties of the magnetic Weyl quantization}\label{Section:PropertiesofMagneticWeylTransform}
	The aim of this appendix is to revisit the magnetic Weyl quantization in the semiclassical regime and to collect several explicit formulas and basic properties that will be used in the main text. Many of the corresponding calculations in the case $\hbar=1$ are standard and can be found, for instance, in \cite{iftimie2007magnetic}. Our main purpose here is therefore to redo these calculations with only finite Sobolev regularity while keeping careful and explicit track of the dependence on the semiclassical parameter $\hbar$ at every step. 
    
    Throughout this appendix, we work in arbitrary spatial dimension $d$ and denote the corresponding phase space by
    $\Xi=\rr^d\times\rr^d$. 
    We write points in $\Xi$ as $X=(x,v), Y=(y,\eta), Z=(z,\zeta)$.
    The phase space $\Xi$ is naturally equipped with the canonical symplectic $2$-form
    \begin{equation*}
        \sigma(Y,Z)=\eta\cdot z-y\cdot\zeta.
    \end{equation*}
    
    We shall also encounter circulations of the magnetic vector potential $A$ along the boundary of triangles in configuration space. By Stokes' theorem, these circulations can be expressed as the flux of the magnetic $2$-form $B$ through the corresponding triangle. Denoting by $\langle x,y,z\rangle$ the oriented triangle with vertices $x,y,z$, we write
    \begin{equation*}
        \Gamma^B(\langle x,y,z\rangle)
        :=
        \Gamma^B([x,y])
        +
        \Gamma^B([y,z])
        +
        \Gamma^B([z,x])
        =
        \int_{\langle x,y,z\rangle}\bB\dd \sigma\,.
    \end{equation*}
	
	\subsection{Properties of magnetic Weyl quantization and estimates of magnetic terms}
	The magnetic Weyl quantization has the gauge covariance property, which is the following proposition:
	\begin{proposition}\label{prop:A1-gauge-covariance}
    	Let $A_1,A_2\in W^{1,\infty}_{\rm loc}$ be real-valued and suppose that $\d A_1=\d A_2$ distributionally. Then there exists a real-valued function $\varphi\in W^{2,\infty}_{\rm loc}$ such that $A_1-A_2=\nabla\varphi$, and
		\begin{equation}
				\operatorname{Op}^{A_1}_{h}(f)=\operatorname{e}^{\frac{2\pi i}{h}\varphi}\operatorname{Op}^{A_2}_{h}(f)\operatorname{e}^{-\frac{2\pi i}{h}\varphi}\,.
		\end{equation}
        on $\mathcal{S}(\mathbb{R}^d)$, whenever the magnetic Weyl quantizations
        are well defined.
        
        In particular, if $\operatorname{Op}_{h}^{A_2}(f)\in\mathfrak{S}^p$
        for some $1\le p\le\infty$, then
        $\operatorname{Op}_{h}^{A_1}(f)\in\mathfrak{S}^p$ and
        \begin{equation}\label{eq:gauge_covariance_schatten-p}
            \bigl\|\operatorname{Op}_{h}^{A_1}(f)\bigr\|_{\cL^p}
            =
            \bigl\|\operatorname{Op}_{h}^{A_2}(f)\bigr\|_{\cL^p}\,.
        \end{equation}
	\end{proposition}

    \begin{proof}
         Set $C=A_1-A_2$ and define $\varphi(x)=\int_0^1 x\cdot C(tx)\dd t$.
        Regularize $C$, differentiate, use $\d C=0$, and pass to the limit yields  $\nabla\varphi=C$ distributionally.  More precisely, we have 
        \begin{align*}
            \bd_k\varphi(x) =&\, \int^1_0\(C_k(tx)+t \sum^d_{j=1}x_j\,\bd_k C_j(tx)\)\d t
            =\, \int^1_0 \frac{\d}{\d t}(t\, C_k(tx))\dd t = C_k(x)\,
        \end{align*}
        where $\d C=0$ was used in the second equality. 
        Hence, we have
        \[
            \Gamma^{A_1}([x,y])-\Gamma^{A_2}([x,y])=\varphi(y)-\varphi(x).
        \]
        Substituting this back into \eqref{def:magnetic_Weyl_quantization} yields the desired result. Moreover, since multiplication by
        $e^{\frac{2\pi i}{h}\varphi}$ is unitary in $L^2(\mathbb{R}^d)$,
        \eqref{eq:gauge_covariance_schatten-p} follows from the unitary invariance of
        $\mathfrak{S}^p$.
    \end{proof}
        
	By imitating and modifying the arguments in Section 3.2 of the lecture notes \cite{laflechesemiclassical}, we can define the magnetic Wigner transform as follows, which associates each density operator $\brho$ with a function $f_{\brho}$ on the phase space. Like the Fourier transform, when the function $f$ on the phase space and $\brho$ behave well enough, one can regard the magnetic Wigner transform as the inverse transform of the magnetic Weyl quantization.
	
	\begin{definition}
		For a Hilbert--Schmidt kernel $\brho(x,y)$, we define in $L^2$ the magnetic Wigner transform of $\brho$ as follows in the strong sense:
		\begin{equation}
			f^h_{\brho}(x,v):=\int_{\rr^d}\operatorname{e}^{-\frac{2\pi i}{h}y\cdot v}\brho\left(x+\tfrac{y}{2},x-\tfrac{y}{2}\right)\operatorname{e}^{\frac{2\pi i}{h}\Gamma^A\left(\left[x+\frac{1}{2}y,\, x-\frac{1}{2}y\right]\right)}\dd y\,.
		\end{equation}
	\end{definition}
	Similar to the Weyl quantization, the magnetic Weyl quantization is also an $L^2$ isometry, which is a direct corollary of Plancherel's theorem:
	\begin{proposition}\label{Proposition:IsometryRelation}
		For $f\in L^2(\Xi)$, we have the following isometry relation:
		\begin{equation}
			\stnorm{2}{\OPAh(f)}=\lnorm{2}{f}\,.
		\end{equation}
	\end{proposition}
	Next, we introduce the composition law for the MWQ. 
\begin{definition}[Magnetic Moyal product]
Let $B$ be real-valued, continuous, and closed.
For $f,g\in C_c(\Xi)$, define
\begin{subequations}
       \begin{multline}\label{eq:magnetic_moyal_product}
                      (f\#_B g)(X)= \left(\frac{2}{h}\right)^{2d}
        \iint_{\Xi\times\Xi}
        e^{-\frac{4\pi i}{h}\sigma(X-Y,X-Z)}\\
        \Omega^B_h(x-z+y, -x+y+z, x-y+z)f(Y)g(Z)\dd Y\d Z
       \end{multline}
where 
       \begin{equation*}
       \Omega^B_h(x, y, z):=e^{-\frac{2\pi i}{h}\Gamma^B(\langle x,y,z\rangle)}, \qquad                 
      \Gamma^B(\langle x,y,z\rangle)
        =
        \int_{\langle x,\;y,\;z\rangle} B\dd \sigma\,.
       \end{equation*}
       We can also write \begin{equation}\label{eq:magnetic_moyal_product2}
    (f\#_B g)(X)=
        \left(\frac{2}{h}\right)^{2d}
        \iint_{\Xi\times\Xi}
        e^{-\frac{4\pi i}{h}\sigma(Y,Z)}
        \omega^B_{h}(x, y, z)f(X-Y)g(X-Z)\dd Y\d Z\,,
\end{equation}
       \end{subequations}
       where 
       \begin{equation}
           \omega^B_{h}(x, y, z) := \Omega^B_{h}(x+y-z,x-y+z,x-y-z)\,.
       \end{equation}
    Here, $f\#_B g$ is called the magnetic Moyal product.
\end{definition}

    Let us now state the composition law. Its proof only requires a corresponding modification of the case $h=1$, see \cite[Section 1.3]{iftimie2007magnetic}.
    
        \begin{proposition}[composition law]
        \label{prop:magnetic-product}
        Let $A \in W_{\rm loc}^{1, \infty}$ be real-valued and $B=\d A\in W_{\mathrm{loc}}^{1,\infty}$. Let $f,g\in C_c(\Xi)$, then we have
        \begin{equation}
            \operatorname{Op}_h^A(f)\operatorname{Op}_h^A(g)
        =
        \operatorname{Op}_h^A(f\#_B g)\,.
        \end{equation}
        \end{proposition}

\begin{proof}
	For $u\in\mathcal{S}(\rr^d)$ and $f$, $g\in C_{c}(\Xi)$, applying the definition of the magnetic Weyl quantization twice gives
	\begin{equation*}
		\begin{aligned}
			\OPAh(f)\OPAh(g)u(x)=&\iint_{\rr^{d}\times\rr^d} e^{-2\pi i\eta\cdot(y-x)-\frac{2\pi i}{h}\Gamma^A([x,y])}f\left(\tfrac{x+y}{2},h\eta\right)\OPAh(g)u(y)\,\mathrm{d} y\d\eta\\
			=&\iiiint_{\rr^{d}\times\rr^d\times \rr^{d}\times\rr^d}e^{-2\pi i\eta\cdot(y-x)-\frac{2\pi i}{h}\Gamma^A([x,y])}f\left(\tfrac{x+y}{2},h\eta\right)\\
			&\times e^{-2\pi i\zeta\cdot(z-y)-\frac{2\pi i}{h}\Gamma^A([y,z])}g\left(\tfrac{y+z}{2},h\zeta\right)u(z)\,\mathrm{d} y\d\eta\mathrm{d}z\d\zeta\,.
		\end{aligned}
	\end{equation*}
	Hence, the integral kernel of $\OPAh(f)\OPAh(g)$, denoted by $\brho_{f\#_B g}$, is given by
	\begin{multline*}
			\brho_{f\#_B g}=\iiint_{\rr^d\times\rr^d\times\rr^d}\ee^{-2\pi i\eta\cdot(y-x)-2\pi i\zeta\cdot(z-y)-\frac{2\pi i}{h}\Gamma^A([x,y])-\frac{2\pi i}{h}\Gamma^A([y,z])}\\
            \times f\left(\tfrac{x+y}{2},h\eta\right)g\left(\tfrac{y+z}{2},h\zeta\right)\,\mathrm{d} y\d\eta\d\zeta\,.
	\end{multline*}

By Definition~\eqref{def:magnetic_Weyl_quantization}, the integral kernel of $\OPAh(f\#_B g)$ is
\[
\brho_{f\#_B g}(x,z)
=
\int_{\rr^d}
e^{-2\pi i v\cdot(z-x)}
e^{-\frac{2\pi i}{h}\Gamma^A([x,z])}
(f\#_B g)(\tfrac{x+z}{2}, hv)\dd v.
\]
Substituting \eqref{eq:magnetic_moyal_product} into this expression gives
\begin{multline*}
\brho_{f\#_B g}(x,z)
=
\left(\frac{2}{h}\right)^{2d}
\iint_{\Xi\times\Xi}\int_{\rr^d}
e^{-2\pi i v\cdot(z-x)}
e^{-\frac{2\pi i}{h}\Gamma^A([x,z])}
e^{-\frac{4\pi i}{h}\sigma(X-Y,X-Z)}
\\
\times
\Omega^B_h(m-q+p, -m+p+q, m-p+q)
f(Y)g(Z)\dd v\d Y\d Z,
\end{multline*}
where we write $m=\frac{x+z}{2}, X=(\frac{x+z}{2}, hv), Y=(p,\xi)$, and $Z = (q,\theta)$.

Since
\[
\sigma(X-Y,X-Z)
=
\xi\cdot(q-m)-(p-m)\cdot\theta
+h\,v\cdot(p-q),
\]
it follows that
\begin{align*}
\brho_{f\#_B g}(x,z)
=&
\left(\frac{2}{h}\right)^{2d}
\iint_{\Xi\times\Xi}\int_{\rr^d}
e^{-2\pi i v\cdot(z-x+2p-2q)}
e^{-\frac{2\pi i}{h}\Gamma^A([x,z])}
e^{-\frac{4\pi i}{h}(\xi\cdot(q-m)-(p-m)\cdot\theta)}
\\
&\qquad\qquad\times
\Omega^B_h(m-q+p, -m+p+q, m-p+q)
f(Y)g(Z)\dd v\d Y\d Z\\
=&
\left(\frac{2}{h}\right)^{2d}\frac{1}{2^d}
\iint_{\Xi\times\Xi}\delta(\tfrac{z-x}{2}+p-q)
e^{-\frac{2\pi i}{h}\Gamma^A([x,z])}
e^{-\frac{4\pi i}{h}(\xi\cdot(q-m)-(p-m)\cdot\theta)}
\\
&\qquad\qquad\times
\Omega^B_h(m-q+p, -m+p+q, m-p+q)
f(Y)g(Z)\d Y\d Z\,.
\end{align*}
Integrating with respect to $q$ therefore yields $q=p+\frac{z-x}{2}$ and hence
\begin{multline*}
\left(\frac{2}{h}\right)^{2d}\frac{1}{2^d}
\iiint_{\rr^d\times\rr^d\times\rr^d}
e^{-\frac{2\pi i}{h}\Gamma^A([x,z])}
e^{-\frac{4\pi i}{h}(\xi\cdot(p-x)-(p-m)\cdot\theta)}
\\
\times
\Omega^B_h(x, 2p-x, z)
f(p, \xi)g(\tfrac{z-x}{2}+p, \theta)\dd p\d\xi\d \theta
\end{multline*}
Finally, set $y:=2p-x$, so that $p=\frac{x+y}{2}$ and $q=\frac{y+z}{2}$,
and rescale $\xi=h\eta, \theta=h\zeta$. Then
\[
-\tfrac{4\pi i}{h}
\left(\xi\cdot(q-m)-(p-m)\cdot\theta\right)
=
-2\pi i\eta\cdot(y-x)
-2\pi i\zeta\cdot(z-y).
\]
Consequently,
\begin{multline*}
\brho_{f\#_B g}(x,z)=\iiint_{\rr^d\times \rr^d\times\rr^d}
e^{-\frac{2\pi i}{h}\Gamma^A([x,z])}
e^{-2\pi i\eta\cdot(y-x)
-2\pi i\zeta\cdot(z-y)}\\
\Omega^B_h(x, y, z)
f(\tfrac{x+y}{2}, h\eta)g(\tfrac{y+z}{2}, h\zeta)\dd y \d \eta\d \zeta\,.
\end{multline*}
Finally, Stokes' theorem gives
\begin{align*}
\Gamma^B(\langle x,y, z\rangle)
&=
\Gamma^A([x,y])
+\Gamma^A([y,z])
+\Gamma^A([z,x])\\
&=
-\Gamma^A([x,z])
+\Gamma^A([x,y])
+\Gamma^A([y,z]),
\end{align*}
and hence
\[
\Gamma^A([x,z])
+\Gamma^B(\langle z,x,y\rangle)
=
\Gamma^A([x,y])+\Gamma^A([y,z]).
\]
This completes the proof of the proposition.
\end{proof}
	Next we investigate the magnetic term. The extra magnetic term in the definition of magnetic Weyl quantization will contribute polynomially growing terms. For the sake of convenience, we introduce the following notation
	\begin{equation*}
		\begin{aligned}
			\Gamma^B_h:=-\frac{2\pi i}{h}\Gamma^B\lb\lal x+y-z,x-y+z,x-y-z\ral\rb\,.
		\end{aligned}
	\end{equation*}
    Hence it follows $\omega^B_{h}= \exp(\Gamma^B_h)$.
	Then the following estimate is a trivial application of mathematical induction.

\begin{lemma}[Estimate of the flux]
	Let $h>0$, $A\in W^{1, \infty}_{\rm loc}$ be real-valued and suppose that $\d A=B\in W^{k,\infty}$ for some $k \in \mathbb{N}$. For any multi-indices $\alpha, \beta, \gamma \in \mathbb{N}^d_0$ satisfying $\snorm{\alpha+\beta+\gamma}\le k$, there exists a positive constant $C\left(\alpha,\beta,\gamma,B, d\right)$ such that
	\begin{equation*}
		\begin{aligned}
			\snorm{\partial^\alpha_x\partial^\beta_y\partial^\gamma_z \omega^B_h(x,y,z)}\le \frac{C\left(\alpha,\beta,\gamma,B,d\right)}{h^{\snorm{\alpha+\beta+\gamma}}}\lb \lal y\ral+\lal z\ral\rb^{2\snorm{\alpha+\beta+\gamma}}\,.
		\end{aligned}
	\end{equation*}
\end{lemma}
\begin{proof}
    By Stokes' theorem,
	\begin{equation}\label{Equation:ExplicitIntegralOfMagneticField}
			\Gamma^B(\lal x-y+z,x-y-z,x+y-z\ral)
            =4\sum_{m,n=1}^dy_m(z_n-y_n)\int^1_0\int^1_0B_{mn}(R_{s, t})\, s\dd s\d t\,,
	\end{equation}
    where $R_{s, t}:=x-y-z+2sy+2st(z-y)$ is used as a convenient shorthand. This representation is particularly convenient for computing derivatives of the flux. Direct differentiation and the skew-symmetry of $B$ give
	\begin{equation}\label{Equation:CalculationsAboutMagneticFieldIntegral 1}
        \begin{aligned}
            		\partial_{y_j}\Gamma^B_h=&-\frac{8\pi i}{h}\sum_{n=1}^dz_n\int^1_0\int^1_0B_{jn}(R_{s, t})\,s\dd s\d t\\
			&-\frac{8\pi i}{h}\sum_{m,n=1}^dy_m z_n\int^1_0\int^1_0\partial_jB_{mn}(R_{s, t})\,s(2s-2st-1)\dd s\d t\,,
        \end{aligned}
	\end{equation}
\begin{equation}\label{Equation:CalculationsAboutMagneticFieldIntegral 2}
\begin{aligned}
	\partial_{y_j}\partial_{y_k}\Gamma^B_h=&\, -\frac{8\pi i}{h}\sum_{n=1}^d
    z_n
	\int_0^1\int_0^1
	\(\partial_kB_{jn}+\partial_jB_{kn}\)(R_{s,t})\,s(2s-2st-1)\dd s\d t\\
    &\, -\frac{8\pi i}{h}\sum_{m,n=1}^d y_m z_n
	\int_0^1\int_0^1
	(\partial_j\partial_kB_{mn})(R_{s,t})\,s(2s-2st-1)^2\dd s\d t\,.
\end{aligned}
\end{equation}
More generally, we have that 
\begin{multline}\label{eq:general_partials_Gamma^B}
            \partial_x^\alpha\partial_y^\beta\partial_z^\gamma
        \Gamma_h^B
        =
        -\frac{8\pi i}{h}
        \sum_{m,n=1}^d
        \int_0^1\int_0^1 
        \Big[
        y_mz_n
        a_{s,t}^{|\beta|}
        b_{s,t}^{|\gamma|}
        (\partial^{\alpha+\beta+\gamma}B_{mn})(R_{s,t})
        \\
        +\beta_m z_n
        a_{s,t}^{|\beta|-1}
        b_{s,t}^{|\gamma|}
        (\partial^{\alpha+\beta+\gamma-e_m}B_{mn})(R_{s,t})
        \\
        +\gamma_n y_m
        a_{s,t}^{|\beta|}
        b_{s,t}^{|\gamma|-1}
        (\partial^{\alpha+\beta+\gamma-e_n}B_{mn})(R_{s,t})
        \\
        +\beta_m\gamma_n
        a_{s,t}^{|\beta|-1}
        b_{s,t}^{|\gamma|-1}
        (\partial^{\alpha+\beta+\gamma-e_m-e_n}B_{mn})(R_{s,t})
        \Big]s
        \dd s\d t\,,
\end{multline}
where $a_{s, t}:=2s-2st-1$ and $b_{s, t}=2st-1$. Here, $e_n$ is the unit vector in the $n$th direction. 

Set $N:=|\alpha+\beta+\gamma|$. Using \eqref{eq:general_partials_Gamma^B}, we immediately have 
\begin{equation*}
    \snorm{\partial_x^{\alpha'}\partial_y^{\beta'}\partial_z^{\gamma'}
        \Gamma_h^B} \le \frac{C(\alpha, \beta, \gamma, d)}{h}\|B\|_{W^{N, \infty}} (1+|y|)(1+|z|)
\end{equation*}
for all $1\le |\alpha'+\beta'+\gamma'|\le N$. In particular, we have that 
\begin{align*}
        \snorm{\partial_x^{\alpha}\partial_y^{\beta}\partial_z^{\gamma}
        \omega_h^B} \le&\, C(\alpha, \beta, \gamma, d)\sum^N_{r=1}\frac{\|B\|_{W^{N, \infty}}^r}{h^r} (1+|y|)^r(1+|z|)^r\\
        \le&\, \frac{C(\alpha, \beta, \gamma, d)P_N(|B\|_{W^{N, \infty}} )}{h^N}(1+|y|)^N(1+|z|)^N
\end{align*}
where $P_N(x)$ is a polynomial of degree at most $N$. 
\end{proof}

    \begin{remark}
        The calculations presented in the above lemma also give the following useful results:
        \begin{equation}\label{Equation:ResultOfMagneticIntegral}
		\begin{gathered}
			\partial_{y_j}\Gamma^B_h=\partial_{y_j}\partial_{y_k}\Gamma^B_h=\partial_{z_j}\Gamma^B_h=\partial_{z_j}\partial_{z_k}\Gamma^B_h=0\,,\\
			\partial_{y_j}\partial_{z_k}\Gamma^B_h=-\frac{4\pi i}{h}B_{jk}(x)\,
		\end{gathered}
        \end{equation}
        when $y=z=0$.
        These computations will be applied in the proof of the finite expansion of the magnetic Moyal product, see the subsection below.
    \end{remark}
	
	\subsection{Asymptotic expansion of the magnetic Weyl quantization}
	In this section, we consider the asymptotic expansion of the magnetic Weyl quantization. This part is inspired by the work \cite{iftimie2007magnetic}.
	
	Before we start the discussion of asymptotic expansion we first introduce some notations and definitions. We define the symbol class $S^{m_1,m_2}_N$ as
	\begin{equation*}
		\begin{aligned}
			S^{m_1,m_2}_N:=\left\{f\in\mathcal{S}'(\Xi):\,\forall \alpha,\beta\in\mathbb{N}^d,\,\snorm{\alpha}\le N,\,\snorm{\beta}\le N,\,
			\snorm{\partial_x^\alpha\partial_v^\beta f}\le C_{\alpha,\beta}
			\lal x\ral^{m_1-\snorm{\alpha}} \lal v\ral^{m_2-\snorm{\beta}}\right\}\,.
		\end{aligned}
	\end{equation*}
	In general application, the distribution will be chosen as some functions with finite regularity and good decay. 
    
    For any $\varphi = \varphi(x, y, z)\in \mathcal{C}^{\infty}(\rr^{3d})$, we define the following first order linear differential operator $M_B(\phi)$, acting on the class of symbols over $\Xi$ with variables 
	\begin{equation*}
			U=(u, \mu) \quad \text{ and } \quad W=(w, \nu)\,,
	\end{equation*}
	 given by 
	\begin{equation}
		\begin{aligned}
			M_B(\varphi)=\lb\omega^B_h\rb^{-1}\sum_{j=1}^d\lb \partial_{y_j}\lb \omega^B_h\varphi\rb\partial_{\nu_j}-\partial_{z_j}\lb \omega^B_h\varphi\rb\partial_{\mu_j}\rb\,.
		\end{aligned}
	\end{equation}
	In particular, if $B=0$, we have 
    \begin{align}
        M_0(\varphi) = \sum_{j=1}^d\lb \partial_{y_j}\varphi\,\partial_{\nu_j}-\partial_{z_j}\varphi\,\partial_{\mu_j}\rb\,.
    \end{align}
    We also define the symplectic form $\sigma(\partial_U,\partial_W)$ on the class of symbols by
	\begin{equation}
		\begin{aligned}
			\sigma(\partial_U,\partial_W)=\sum_{j=1}^d\lb \partial_{\mu_j}\partial_{w_j}-\partial_{u_j}\partial_{\nu_j}\rb\,.
		\end{aligned}
	\end{equation}
    For $t>0$, we also denote that 
	\begin{equation}
		\begin{aligned}
			L_\varphi(t):=\frac{1}{4\pi i}\left\{2\varphi\,\sigma(\partial_U,\partial_W)+t^{-1}M_0(\varphi)\right\}\,.
		\end{aligned}
	\end{equation}
    
    Notice, we use the notation $\partial_U$ and $\partial_W$ to emphasize that these differential operators act on the symbol variables $U$ and $W$, and hence on expressions such as $f(U)g(W)$. This distinguishes them from $\partial_Y$ and $\partial_Z$, which differentiate with respect to the integration variables $Y$ and $Z$. The two sets of derivatives are related only after the substitutions $U=X-tY$ and $W=X-tZ$, as will be used in the computations below.
    
	Let us start by stating the following useful lemma.
\begin{lemma}[Expansion lemma]
\label{Lemma:ExpansionLemma}
Let $h>0$. Then the following statements hold.

\begin{enumerate}
\item Let $\varphi\in C^1(\rr^{3d})$ be such that, for some $K\geq0$,
$\varphi$, $\nabla_y\varphi$, and $\nabla_z\varphi$ have at most
polynomial growth of order $K$ in $(y,z)$, locally uniformly in $x$.
Then, with the integration understood in the sense of tempered
distributions, one has
\begin{equation}\label{eq:oscillatory_integral_id} 
	\left(\frac{2}{h}\right)^{2d}
	\iint_{\Xi\times\Xi}
	e^{-\frac{4\pi i}{h}\sigma(Y,Z)}
	\varphi(x,y,z)
	\,\mathrm{d}Y\,\mathrm{d}Z
	=
	\varphi(x,0,0).
\end{equation}

		\item Let $K,M\geq0$, and choose an integer $q$ such that
		\begin{equation*}
		q>K+M+4d.
		\end{equation*}
		Assume that $\varphi\in C^{q+1}(\rr^{3d})$ satisfies
		\begin{equation*}
		\left|
		\partial_{y,z}^{\alpha}\varphi(x,y,z)
		\right|
		\leq
		C_{\alpha}(x)\langle(y,z)\rangle^K,
		\qquad |\alpha|\leq q+1,
		\end{equation*}
		where $C_{\alpha}$ is locally bounded in $x$. Assume also that
		$\psi\in C^{q+2}(\Xi\times\Xi)$ satisfies
		\begin{equation*}
		\left|
		\partial_{U,W}^{\beta}\psi(U,W)
		\right|
		\leq
		C_{\beta}\langle(U,W)\rangle^M,
		\qquad |\beta|\leq q+2.
		\end{equation*}
		Then, for every $t>0$, the following identity holds in the sense of
		tempered distributions:
		\begin{multline}\label{eq:distributional-expansion-lemma}
		\left(\frac{2}{h}\right)^{2d}
		\iint_{\Xi\times\Xi}
		e^{-\frac{4\pi i}{h}\sigma(Y,Z)}
		\varphi(x,y,z)
		\psi(X-tY,X-tZ)
		\,\mathrm{d}Y\,\mathrm{d}Z
		\\
		=
		\varphi(x,0,0)\psi(X,X)\\
		+
		h\left(\frac{2}{h}\right)^{2d}t^2
		\int_0^1s
		\iint_{\Xi\times\Xi}
		e^{-\frac{4\pi i}{h}\sigma(Y,Z)}
		\bigl(L_\varphi(st)\psi\bigr)
		(X-stY,X-stZ)
		\,\mathrm{d}Y\,\mathrm{d}Z\,\mathrm{d}s.
		\end{multline}
		At $s=0$, the expression $sL_\varphi(st)\psi$ is understood by its
		continuous extension in the sense of tempered distributions.
		\end{enumerate}
		\end{lemma}
        
	\begin{proof}
        The semiclassical parameter in \eqref{eq:oscillatory_integral_id} plays a trivial role. In fact, the proof follows from the standard Fourier inversion formula (see, e.g., \cite[Lemma 2.1]{iftimie2007magnetic}). 

For the second part, set
\begin{equation*}
	I(s)
	:=
	\left(\frac{2}{h}\right)^{2d}
	\iint_{\Xi\times\Xi}
	e^{-\frac{4\pi i}{h}\sigma(Y,Z)}
	\varphi(x,y,z)
	\psi(X-stY,X-stZ)\dd Y\d Z,
\end{equation*}
where the integration is understood in the sense of tempered
distributions.

To justify the calculation, choose
$\chi\in C_c^\infty(\Xi\times\Xi)$ such that $\chi=1$ in a
neighborhood of the origin, and insert
$\chi(\varepsilon Y,\varepsilon Z)$ into the definition of $I(s)$.
For each $\varepsilon>0$, differentiation and integration by parts are
justified in the usual sense. The assumptions
\begin{equation*}
	q>K+M+4d,
	\qquad
	\varphi\in C^{q+1}(\rr^{3d}),
	\qquad
	\psi\in C^{q+2}(\Xi\times\Xi),
\end{equation*}
together with the polynomial-growth bounds, allow us to integrate by
parts repeatedly against the exponential factor. Consequently, the
regularized expressions converge in the sense of tempered
distributions as $\varepsilon\to0$, and all terms in which a
derivative falls on the cutoff tend to zero. Thus the following
calculations are valid in the sense of tempered distributions.

Differentiating with respect to $s$, we obtain
\begin{multline*}
	I'(s)
	=
	-t\left(\frac{2}{h}\right)^{2d}
	\iint_{\Xi\times\Xi}
	e^{-\frac{4\pi i}{h}\sigma(Y,Z)}
	\varphi(x,y,z)
	\\
	{}\times
	\left(
	Y\cdot\partial_U+Z\cdot\partial_W
	\right)
	\psi(X-stY,X-stZ)\dd Y\d Z.
\end{multline*}
We have
\begin{equation*}
	y_j e^{-\frac{4\pi i}{h}\sigma(Y,Z)}
	=
	\frac{h}{4\pi i}
	\partial_{\zeta_j}
	e^{-\frac{4\pi i}{h}\sigma(Y,Z)},
	\qquad
	z_j e^{-\frac{4\pi i}{h}\sigma(Y,Z)}
	=
	-\frac{h}{4\pi i}
	\partial_{\eta_j}
	e^{-\frac{4\pi i}{h}\sigma(Y,Z)},
\end{equation*}
and
\begin{equation*}
	\eta_j e^{-\frac{4\pi i}{h}\sigma(Y,Z)}
	=
	-\frac{h}{4\pi i}
	\partial_{z_j}
	e^{-\frac{4\pi i}{h}\sigma(Y,Z)},
	\qquad
	\zeta_j e^{-\frac{4\pi i}{h}\sigma(Y,Z)}
	=
	\frac{h}{4\pi i}
	\partial_{y_j}
	e^{-\frac{4\pi i}{h}\sigma(Y,Z)}.
\end{equation*}
Integrating by parts in $y,\eta,z,\zeta$, we obtain
\begin{multline*}
	\iint_{\Xi\times\Xi}
	e^{-\frac{4\pi i}{h}\sigma(Y,Z)}
	\varphi(x,y,z)
	\left(
	Y\cdot\partial_U+Z\cdot\partial_W
	\right)
	\psi(X-stY,X-stZ)\dd Y\d Z
	\\
	=
	-hst
	\iint_{\Xi\times\Xi}
	e^{-\frac{4\pi i}{h}\sigma(Y,Z)}
	\bigl(L_\varphi(st)\psi\bigr)(X-stY,X-stZ)
	\dd Y\d Z.
\end{multline*}
It follows that
\begin{equation*}
	I'(s)
	=
	h\left(\frac{2}{h}\right)^{2d}t^2s
	\iint_{\Xi\times\Xi}
	e^{-\frac{4\pi i}{h}\sigma(Y,Z)}
	\bigl(L_\varphi(st)\psi\bigr)(X-stY,X-stZ)
	\dd Y\d Z.
\end{equation*}

Moreover,
\begin{equation*}
	sL_\varphi(st)
	=
	\frac{1}{4\pi i}
	\left(
	2s\varphi\,\sigma(\partial_U,\partial_W)
	+
	\frac{1}{t}M_0(\varphi)
	\right),
\end{equation*}
so the right-hand side has a continuous extension to $s=0$ in the
sense of tempered distributions. By
\eqref{eq:oscillatory_integral_id},
\begin{equation*}
	I(0)=\varphi(x,0,0)\psi(X,X).
\end{equation*}
Integrating the preceding identity in $s$ therefore yields
\begin{equation*}
\begin{aligned}
	I(1)
	={}&
	\varphi(x,0,0)\psi(X,X)
	\\
	&+
	h\left(\frac{2}{h}\right)^{2d}t^2
	\int_0^1s
	\iint_{\Xi\times\Xi}
	e^{-\frac{4\pi i}{h}\sigma(Y,Z)}
	\bigl(L_\varphi(st)\psi\bigr)(X-stY,X-stZ)
	\dd Y\d Z\d s.
\end{aligned}
\end{equation*}
Since $I(1)$ is precisely the left-hand side of the second statement,
the proof is complete.	\end{proof}

	To any differential operator $P:= \sum c_{\alpha, \beta} \partial_U^\alpha \partial_W^\beta$, of order $m$ with respect to the variables $U$ and $W$, we associate another differential operator $M_B(P):= \sum M_B(c_{\alpha\beta}) \partial_U^\alpha \partial_W^\beta$. This operator will evidently have the same form, but will be of order $m+1$.
    
    Next, we introduce some linear partial differential operators by recurrence. For any positive sequence $\{t_j\}$, we define the following sequence of differential operators 
	\begin{equation}\label{Equation:LinearOperatorsInduction}
		\begin{gathered}
			L_0:=\, 1\,,\\
			L_1(t_1):=\, (\omega^B_h)^{-1}L_{\omega^B_h}(t_1)\,,\\
			L_{j+1}(t_1,...,t_{j+1}):=\, L_1(t_{(j+1)})L_j(t_1,...,t_j)+\frac{M_0(L_j(t_1,...,t_j))}{4\pi i\,t_{(j+1)}}\,,
		\end{gathered}
	\end{equation}
	where $t_{(k)}:=t_1t_2...t_k$. Then we have the following theorem.

    \begin{proposition}[Finite expansion of the magnetic Moyal product]
\label{Theorem:AsymptoticExpansion}
Assume $h>0$ and $B\in W^{k+2,\infty}$ for some $k \in \mathbb{N}$. 
For every choice of symbol orders
$m_1,m_2,\widetilde m_1,\widetilde m_2$, there exists an integer
\begin{equation*}
    N_0
    =
    N_0\bigl(
        d,k,m_1,m_2,\widetilde m_1,\widetilde m_2
    \bigr)
\end{equation*}
such that the following conclusion holds whenever
\begin{equation}
    f\in S_{N_1}^{m_1,m_2},
    \qquad
    g\in S_{N_2}^{\widetilde m_1,\widetilde m_2},
    \qquad
    N_1,N_2\geq N_0.
\end{equation}
In the sense of tempered
distributions on $\Xi$, one has the identity
\begin{equation}\label{eq:finite-magnetic-Moyal-expansion}
    f\#_B g
    =
    \sum_{j=0}^{k-1}h^jr_j(f,g)
    +
    h^kR_k(f,g;h)\,.
\end{equation}
Here, we have
\begin{equation}
    r_0(f,g):=f(X)g(X)\,,
\end{equation}
and, for $1\leq j\leq k-1$,
\begin{equation}
    r_j(f,g)(X)
    :=
    \int_{[0,1]^j}
    t_1^{2j-1}t_2^{2j-3}\cdots t_j
    \left.
    L_j(t_1,\ldots,t_j)
    \bigl(f(U)g(W)\bigr)
    \right|_{\substack{
        U=W=X\\
        y=z=0
    }}
    \,\mathrm{d}t_1\cdots\mathrm{d}t_j.   
\end{equation}
The remainder is
\begin{multline}
    R_k(f,g;h)(X)
    :=
    \left(\frac{2}{h}\right)^{2d}
    \int_{[0,1]^k}
    t_1^{2k-1}t_2^{2k-3}\cdots t_k
    \iint_{\Xi\times\Xi}
    \operatorname{e}^{-\frac{4\pi i}{h}\sigma(Y,Z)}
    \omega_h^B(x,y,z)\\
    \times
    \left.
    L_k(t_1,\ldots,t_k)
    \bigl(f(U)g(W)\bigr)
    \right|_{\substack{
        U=X-t_{(k)}Y\\
        W=X-t_{(k)}Z
    }}
    \dd Y\d Z\,
    \d t_1\cdots\d t_k\,.
\end{multline}
\end{proposition}

    \begin{proof}
We prove the proposition by induction. By the choice of $N_0$, all
applications of Lemma~\ref{Lemma:ExpansionLemma} below are justified in
the sense of tempered distributions.

Let $k=1$, take
$\varphi(x,y,z)=\omega_h^B(x,y,z)$, $\psi=f\otimes g$, $t=1$, and
$s=t_1$. Since $\omega_h^B(x,0,0)=1$,
Lemma~\ref{Lemma:ExpansionLemma} gives
\begin{multline*}
(f\#_B g)(X)
=\,
f(X)g(X)
\\
+
h\left(\frac{2}{h}\right)^{2d}
\int_0^1t_1
\iint_{\Xi\times\Xi}
\mathrm{e}^{-\frac{4\pi i}{h}\sigma(Y,Z)}
\left.
L_{\omega_h^B}(t_1)(f\otimes g)(U,W)
\right|_{\substack{
U=X-t_1Y\\
W=X-t_1Z
}}
\dd Y\d Z\d t_1.
\end{multline*}
Notice, by definition, we have $L_1(t_1)=(\omega^B_h)^{-1}L_{\omega^B_{h}}(t_1)$, then it follows
we
\begin{align*}
(f\#_B g)(X)
={}&
f(X)g(X)
\\
&+
h\left(\frac{2}{h}\right)^{2d}
\int_0^1t_1
\iint_{\Xi\times\Xi}
\mathrm{e}^{-\frac{4\pi i}{h}\sigma(Y,Z)}
\omega_h^B(x,y,z)\\
&\, \times
\left.
L_1(t_1)\bigl(f(U)g(W)\bigr)
\right|_{\substack{
U=X-t_1Y\\
W=X-t_1Z
}} \dd Y\d Z\d t_1
\\
={}&
r_0(f,g)+hR_1(f,g;h).
\end{align*}
Thus, the proposition holds for $k=1$.

Suppose now that the proposition holds for some $k\geq1$. We decompose
the remainder $R_k$. By the definition of $L_k$, it can be written in
the form
\begin{equation*}
L_k(t_1,\ldots,t_k)
=
\sum_{\alpha,\beta}
a_{\alpha,\beta}(x,y,z,t_1,\ldots,t_k)\,
\partial_U^\alpha \partial_W^\beta.
\end{equation*}
Consequently, we have
\begin{equation*}
L_k(t_1,\ldots,t_k)\bigl(f(U)g(W)\bigr)
=
\sum_{\alpha,\beta}
a_{\alpha,\beta}(x,y,z,t_1,\ldots,t_k)\,
(\partial^\alpha f)(U)(\partial^\beta g)(W).
\end{equation*}

For each pair $(\alpha,\beta)$, take $\varphi=\omega_h^Ba_{\alpha,\beta}, \psi=(\partial^\alpha f)\otimes(\partial^\beta g),
t=t_{(k)},
s=t_{k+1}$,
and apply Lemma~\ref{Lemma:ExpansionLemma}. Using the fact
$\omega_h^B(x,0,0)=1$, we obtain
\begin{align*}
R_k(f,g;h)
={}&
\int_{[0,1]^k}
t_1^{2k-1}t_2^{2k-3}\cdots t_k
\\
&\quad\times
\left.
L_k(t_1,\ldots,t_k)\bigl(f(U)g(W)\bigr)
\right|_{\substack{
U=W=X\\
y=z=0
}}
\,\mathrm{d}t_1\cdots\mathrm{d}t_k
\\
&+
h\left(\frac{2}{h}\right)^{2d}
\sum_{\alpha,\beta}
\int_{[0,1]^{k+1}}
t_1^{2k+1}t_2^{2k-1}\cdots t_k^3t_{k+1}
\\
&\quad\times
\iint_{\Xi\times\Xi}
\mathrm{e}^{-\frac{4\pi i}{h}\sigma(Y,Z)}
\\
&\quad\times
\left.
L_{\omega_h^Ba_{\alpha,\beta}}
\bigl(t_{(k+1)}\bigr)
\bigl((D^\alpha f)(U)(D^\beta g)(W)\bigr)
\right|_{\substack{
U=X-t_{(k+1)}Y\\
W=X-t_{(k+1)}Z
}}
\\
&\quad\times
\mathrm{d}Y\,\mathrm{d}Z\,
\mathrm{d}t_1\cdots\mathrm{d}t_{k+1}.
\end{align*}
The first term on the right-hand side is $r_k(f,g)$.

By the definition of $L_\varphi$,
\begin{align*}
L_{\omega_h^Ba_{\alpha,\beta}}
\bigl(t_{(k+1)}\bigr)
={}&
\frac{1}{4\pi i}
\left\{
2\omega_h^Ba_{\alpha,\beta}
\sigma(\partial_U,\partial_W)
+
\frac{
M_0(\omega_h^Ba_{\alpha,\beta})
}{
t_{(k+1)}
}
\right\}
\\
={}&
\frac{1}{4\pi i}
\left\{
2\omega_h^Ba_{\alpha,\beta}
\sigma(\partial_U,\partial_W)
+
\frac{
\omega_h^BM_B(a_{\alpha,\beta})
}{
t_{(k+1)}
}
\right\}.
\end{align*}
Moreover, the product rule gives
\begin{equation*}
M_B(a_{\alpha,\beta})
=
a_{\alpha,\beta}M_B(1)
+
M_0(a_{\alpha,\beta}).
\end{equation*}
It follows that
\begin{align*}
L_{\omega_h^Ba_{\alpha,\beta}}
\bigl(t_{(k+1)}\bigr)
={}&
\frac{\omega_h^B}{4\pi i}
\left\{
2a_{\alpha,\beta}
\sigma(\partial_U,\partial_W)
+
\frac{
a_{\alpha,\beta}M_B(1)
+
M_0(a_{\alpha,\beta})
}{
t_{(k+1)}
}
\right\}
\\
={}&
\omega_h^B
\left\{
a_{\alpha,\beta}L_1\bigl(t_{(k+1)}\bigr)
+
\frac{
M_0(a_{\alpha,\beta})
}{
4\pi i\,t_{(k+1)}
}
\right\}.
\end{align*}

Since $a_{\alpha,\beta}$ does not depend on the variables $U$ and $W$,
summing over $\alpha$ and $\beta$ gives
\begin{align*}
R_k(f,g;h)
={}&
r_k(f,g)
\\
&+
h\left(\frac{2}{h}\right)^{2d}
\int_{[0,1]^{k+1}}
t_1^{2k+1}t_2^{2k-1}\cdots t_k^3t_{k+1}
\\
&\quad\times
\iint_{\Xi\times\Xi}
\mathrm{e}^{-\frac{4\pi i}{h}\sigma(Y,Z)}
\omega_h^B(x,y,z)
\\
&\quad\times
\left.
\left\{
L_1\bigl(t_{(k+1)}\bigr)
L_k(t_1,\ldots,t_k)
+
\frac{
M_0\bigl(L_k(t_1,\ldots,t_k)\bigr)
}{
4\pi i\,t_{(k+1)}
}
\right\}
\bigl(f(U)g(W)\bigr)
\right|_{\substack{
U=X-t_{(k+1)}Y\\
W=X-t_{(k+1)}Z
}}
\\
&\quad\times
\mathrm{d}Y\,\mathrm{d}Z\,
\mathrm{d}t_1\cdots\mathrm{d}t_{k+1}.
\end{align*}
By the recursive definition of $L_{k+1}$, this is
\begin{equation*}
R_k(f,g;h)=r_k(f,g)+hR_{k+1}(f,g;h).
\end{equation*}
Therefore, we have
\begin{align*}
f\#_B g
&=
\sum_{j=0}^{k-1}h^jr_j(f,g)
+
h^kR_k(f,g;h)
=
\sum_{j=0}^{k}h^jr_j(f,g)
+
h^{k+1}R_{k+1}(f,g;h).
\end{align*}
Thus, the proposition holds for $k+1$, and the proof is complete.
\end{proof}
    
	Next we only consider the asymptotic expansion under the condition that $f(X)=\frac{1}{2}\snorm{v}^2$ and $g$ is the solution of Cauchy problem of the magnetic Vlasov--Poisson equation with initial data which satisfy the statements listed in Theorem \ref{Theorem:PropagationOfRegularity}. Then we obtain the following proposition.

    \begin{proposition}\label{Propositon:QuantizationOfCommutator}
        Suppose $f$ is the solution of Cauchy problem of magnetic Vlasov--Poisson
        equation with initial data satisfy the conditions stated in Theorem \ref{Theorem:PropagationOfRegularity}, then we have
        \begin{equation}
            \lb\tfrac{1}{2}\snorm{v}^2\rb\#_B f-f\#_B\lb\tfrac{1}{2}\snorm{v}^2\rb=-\frac{\hbar}{i}\left\{\tfrac{1}{2}\snorm{v}^2,f\right\}_B+\hbar^2\widetilde{R}\,,
        \end{equation}
        where the remainder term satisfies that for any multi-index $\alpha$, $\beta\in\mathbb{N}^d$, $\snorm{\alpha}\,\text{,}\snorm{\beta}\le m$, 
        \begin{equation}
            \begin{aligned}
                \snorm{\partial^\alpha_x\partial^\beta_v\widetilde{R}(x,v)}\le C(d,\alpha,\beta, B)\lal x\ral^{-d-2-\snorm{\alpha}}\lal v\ral^{-d-2-\snorm{\beta}}\,.
            \end{aligned}
        \end{equation} 
    \end{proposition}

    \begin{proof}
    Then it's easy to see that $\frac{1}{2}\snorm{v}^2\in S^{0,2}_{\infty}$ and we can regard $f$ is an element in $S^{-\infty,-\infty}_{m}$ as $d=3$, since $f$ only has finite regularity. In the following discussions, we assume $f\in S^{m_1,m_2}_{N_1}$ and $g\in S^{\widetilde{m_1},\widetilde{m_2}}_{N_2}$ with $N_1$, $N_2$ are sufficiently large positive integers. And we only consider the first three terms given by the asymptotic expansion of $f\#_B g$. 
	
	We first calculate the exact formula of $r_1$ and $r_2$. According to the recursive definition~\eqref{Equation:LinearOperatorsInduction}, we have $L_0=1$ and
	\begin{align*}
			L_1(t_1)=&\, \frac{1}{2\pi i}\sigma(\partial_U,\partial_W)+\frac{1}{4\pi i\,t_1}\sum_{j=1}^d\lb\partial_{y_j}\Gamma^B_h\,\partial_{\nu_j}-\partial_{z_j}\Gamma^B_h\,\partial_{\mu_j}\rb \\
            =:&\, \frac{1}{2\pi i}\sigma(\partial_U,\partial_W)+\frac{1}{4\pi i\,t_1}Q_B\,.
	\end{align*}
	Hence, by \eqref{Equation:ResultOfMagneticIntegral}, we have
	\begin{equation*}
		\begin{aligned}
			L_1(t_1)(f\otimes g)(U,W)\big|_{\substack{U=W=X\\y=z=0}}=-\frac{1}{2\pi i}
			\left\{f,g\right\}\,,
		\end{aligned}
	\end{equation*}
	and consequently $r_1=-\frac{1}{4\pi i}\left\{f,g\right\}$. 

Next, we calculate $L_2(t_1,t_2)$. Notice $M_0(\sigma(\partial_U,\partial_W))=0$; then it follows that
    \begin{equation*}
        \begin{aligned}
        M_0\bigl(L_1(t_1)\bigr)
        =
        \frac{1}{4\pi i\,t_1}
        \sum_{j,k=1}^d
        \Big(
        &\partial_{y_k}\partial_{y_j}\Gamma_h^B\,
        \partial_{\nu_k}\partial_{\nu_j}
        -
        \partial_{z_k}\partial_{y_j}\Gamma_h^B\,
        \partial_{\mu_k}\partial_{\nu_j}
        \\
        &-
        \partial_{y_k}\partial_{z_j}\Gamma_h^B\,
        \partial_{\nu_k}\partial_{\mu_j}
        +
        \partial_{z_k}\partial_{z_j}\Gamma_h^B\,
        \partial_{\mu_k}\partial_{\mu_j}
        \Big) =:\frac{1}{4\pi i\, t_1}H_B.
        \end{aligned}
    \end{equation*}
	 Applying \eqref{Equation:ResultOfMagneticIntegral}, we also get
     \begin{align*}
         M_0(L_1(t_1))(f\otimes g)(U,W)\big|_{\substack{U=W=X\\y=z=0}}=&\, M_0(L_1(t_1))(f\otimes g)\big|_{\substack{y=z=0}}(X, X)\\
         =&\, -\frac{2}{h\,t_1}\sum^d_{j, k=1} B_{jk}(x)(\partial_{\xi_j}f)(X) (\partial_{\xi_k}g)(X)\,.
     \end{align*}
	Hence, direct computation yields
    \begin{equation*}
        L_1(t_{(2)})L_1(t_1) = -\frac{1}{4\pi^2}\sigma(\partial_U, \partial_W)^2-\frac{1}{8\pi^2 t_1}\(1+\frac{1}{t_2}\)Q_B\sigma(\partial_U, \partial_W)-\frac{1}{16\pi^2 t_1^2t_2}Q_B^2
    \end{equation*}
    since $[\sigma(\partial_U, \partial_W), Q_B]=0$. This means that 
    \begin{equation*}
        L_2(t_1, t_2) = -\frac{1}{4\pi^2}\sigma(\partial_U, \partial_W)^2-\frac{1}{8\pi^2 t_1}\(1+\frac{1}{t_2}\)Q_B\sigma(\partial_U, \partial_W)-\frac{1}{16\pi^2 t_1^2t_2}\(Q_B^2+H_B\)\,,
    \end{equation*}
    and 
	\begin{equation*}
		\begin{aligned}
			r_2(f, g)(X)=&\, \frac{-1}{32\pi^2}\sum_{j,k=1}^d\{\partial_{v_j}\partial_{v_k}f(X)\partial_{x_j}\partial_{x_k}g(X)-\partial_{x_k}\partial_{v_j}f(X)\partial_{x_j}\partial_{v_k}g(X)\\
			&-\partial_{x_j}\partial_{v_k}f(X)\partial_{x_k}\partial_{v_j}g(X)+\partial_{x_j}\partial_{x_k}f(X)\partial_{v_j}\partial_{v_k}g(X)\}\\
			&-\frac{1}{4\pi ih}\sum_{j,k=1}^dB_{jk}(x)\partial_{v_j}f(X)\partial_{v_k}g(X)\,.
		\end{aligned}
	\end{equation*}
	
	If we take $f$ and $g$ as we have stated at the beginning then we have
	\begin{equation*}
			f\#_B g=f(X)g(X)-\tfrac{\hbar}{2i}\left\{\tfrac{1}{2}\snorm{v}^2,g\right\}_B+h^3R_3\left(\tfrac{1}{2}\snorm{v}^2,g\right)\,.
	\end{equation*}
	Next we consider the decay estimate of $R_3(f,g)$. 
It turns out that it is sufficient to analyze the ingredients of the
linear differential operator $L_3$ rather than calculate its
coefficients explicitly. We first have
\begin{equation*}
	M_0\bigl(L_2(t_1,t_2)\bigr)
	=
	\sum_{l=1}^d
	\left(
	\partial_{y_l}L_2(t_1,t_2)\,\partial_{\nu_l}
	-
	\partial_{z_l}L_2(t_1,t_2)\,\partial_{\mu_l}
	\right),
\end{equation*}
where $\partial_{y_l}$ and $\partial_{z_l}$ act only on the
coefficients of $L_2$. Since
$\sigma(\partial_U,\partial_W)$ has constant coefficients, we obtain
\begin{equation*}
\begin{aligned}
	\partial_{y_l}L_2(t_1,t_2)
	={}&
	-\frac{1}{8\pi^2t_1}
	\sigma(\partial_U,\partial_W)\partial_{y_l}Q_B
	-
	\frac{1}{8\pi^2t_1t_2}
	\partial_{y_l}Q_B\,\sigma(\partial_U,\partial_W)
	\\
	&+
	\frac{1}{(4\pi i)^2t_1^2t_2}
	\left(
	(\partial_{y_l}Q_B)Q_B
	+
	Q_B\partial_{y_l}Q_B
	+
	\partial_{y_l}H_B
	\right),
\end{aligned}
\end{equation*}
and similarly,
\begin{equation*}
\begin{aligned}
	\partial_{z_l}L_2(t_1,t_2)
	={}&
	-\frac{1}{8\pi^2t_1}
	\sigma(\partial_U,\partial_W)\partial_{z_l}Q_B
	-
	\frac{1}{8\pi^2t_1t_2}
	\partial_{z_l}Q_B\,\sigma(\partial_U,\partial_W)
	\\
	&+
	\frac{1}{(4\pi i)^2t_1^2t_2}
	\left(
	(\partial_{z_l}Q_B)Q_B
	+
	Q_B\partial_{z_l}Q_B
	+
	\partial_{z_l}H_B
	\right).
\end{aligned}
\end{equation*}

For the other term in $L_3$, we have
\begin{equation*}
\begin{aligned}
	L_1(t_1t_2t_3)L_2(t_1,t_2)
	&=
	\left(
	\frac{1}{2\pi i}\sigma(\partial_U,\partial_W)
	+
	\frac{1}{4\pi it_1t_2t_3}Q_B
	\right)
	\\
	&\quad
	\Bigg(
	\frac{1}{(2\pi i)^2}
	\sigma(\partial_U,\partial_W)^2
	-
	\frac{1}{8\pi^2t_1}
	\sigma(\partial_U,\partial_W)Q_B
	\\
	&\qquad
	-
	\frac{1}{8\pi^2t_1t_2}
	Q_B\sigma(\partial_U,\partial_W)
	+
	\frac{1}{(4\pi i)^2t_1^2t_2}
	\left(Q_B^2+H_B\right)
	\Bigg).
\end{aligned}
\end{equation*}
Consequently, the coefficients of
$L_1(t_1t_2t_3)L_2(t_1,t_2)$ contain derivatives of
$\Gamma_h^B$ of order at most two, while the coefficients of
$M_0(L_2(t_1,t_2))$ contain derivatives of $\Gamma_h^B$ of order at
most three.

	Then by gathering \eqref{Equation:CalculationsAboutMagneticFieldIntegral 1}\eqref{Equation:CalculationsAboutMagneticFieldIntegral 2}
    we can assert that $L_3(t_1,t_2,t_3)$ can be represented as a linear combination of the following three types of linear differential operators:
	\begin{itemize}
	\item $\displaystyle
	\frac{
	D_{\alpha,\beta,\wmu,\wnu,1}(x,y,z)
	}{
	t_1^{q_1}t_2^{q_2}t_3^{q_3}
	}
	\partial_u^\alpha\partial_w^\beta
	\partial_\mu^{\wmu}\partial_\nu^{\wnu},
	$
	where
	$D_{\alpha,\beta,\wmu,\wnu,1}
	\in W^{m-3,\infty}(\rr^{3d})$,
	$q_j\leq7-2j$ for $1\leq j\leq3$,
	$|\alpha+\beta|\leq3$, and
	$|\wmu+\wnu|\leq3$.
	\item
	$\displaystyle
	\frac{
	D_{\alpha,\beta,\wmu,\wnu,2}(x,y,z)
	}{
	t_1^{q_1}t_2^{q_2}t_3^{q_3}
	}
	\frac{
	y^{\alpha'}z^{\beta'}
	}{
	h^{|\alpha'+\beta'|}
	}
	\partial_u^\alpha\partial_w^\beta
	\partial_\mu^{\wmu}\partial_\nu^{\wnu},
	$
	where
	$D_{\alpha,\beta,\wmu,\wnu,2}
	\in W^{m-3,\infty}(\rr^{3d})$,
	$q_j\leq7-2j$ for $1\leq j\leq3$,
	$|\alpha+\beta|\leq3$,
	$|\wmu+\wnu|\leq3$, and
	$|\alpha'+\beta'|\leq3$.
	\item
	$\displaystyle
	\frac{
	D_{\alpha,\beta,\wmu,\wnu,3}(x,y,z)
	}{
	t_1^{q_1}t_2^{q_2}t_3^{q_3}
	}
	\frac{
	y^{\alpha''}z^{\beta''}
	}{
	h^{|\alpha''+\beta''|+1}
	}
	\partial_u^\alpha\partial_w^\beta
	\partial_\mu^{\wmu}\partial_\nu^{\wnu},
	$
	where
	$D_{\alpha,\beta,\wmu,\wnu,3}
	\in W^{m-3,\infty}(\rr^{3d})$,
	$q_j\leq7-2j$ for $1\leq j\leq3$,
	$|\alpha+\beta|\leq3$,
	$|\wmu+\wnu|\leq3$, and
	$|\alpha''+\beta''|\leq2$.
\end{itemize}
	Here we emphasize that $D_{\alpha,\beta,\wmu,\wnu,j}$, $1\le j\le 3$ do not  rely on $h$. With these observations in mind we know that $h^3R_3(f,g)$ can be represented as a linear combination of the following integrals:
\begin{equation*}
	\begin{aligned}
		J_1
		={}&\,h^3\lb\frac{2}{h}\rb^{2d}
		\int_0^1\int_0^1\int_0^1
		t_1^5t_2^3t_3
		\iint_{\Xi\times\Xi}
		\ee^{-\frac{4\pi i}{h}\sigma(Y,Z)}
		\omega_h^B(x,y,z)
		\\
		&\times
		\frac{D_{\alpha,\beta,\wmu,\wnu,1}(x,y,z)}
		{t_1^{q_1}t_2^{q_2}t_3^{q_3}}
		(\partial_x^\alpha\partial_v^{\wmu}f)(X-t_{(3)}Y)
		(\partial_x^\beta\partial_v^{\wnu}g)(X-t_{(3)}Z)
		\dd Y\d Z\d t_1\d t_2\d t_3\,,
		\\
		J_2
		={}&\,h^3\lb\frac{2}{h}\rb^{2d}
		\int_0^1\int_0^1\int_0^1
		t_1^5t_2^3t_3
		\iint_{\Xi\times\Xi}
		\ee^{-\frac{4\pi i}{h}\sigma(Y,Z)}
		\lb\frac{y}{h}\rb^{\alpha'}
		\lb\frac{z}{h}\rb^{\beta'}
		\omega_h^B(x,y,z)
		\\
		&\times
		\frac{D_{\alpha,\beta,\wmu,\wnu,2}(x,y,z)}
		{t_1^{q_1}t_2^{q_2}t_3^{q_3}}
		(\partial_x^\alpha\partial_v^{\wmu}f)(X-t_{(3)}Y)
		(\partial_x^\beta\partial_v^{\wnu}g)(X-t_{(3)}Z)
		\dd Y\d Z\d t_1\d t_2\d t_3\,,
		\\
		J_3
		={}&\,h^2\lb\frac{2}{h}\rb^{2d}
		\int_0^1\int_0^1\int_0^1
		t_1^5t_2^3t_3
		\iint_{\Xi\times\Xi}
		\ee^{-\frac{4\pi i}{h}\sigma(Y,Z)} 		
		\lb\frac{y}{h}\rb^{\alpha''}
		\lb\frac{z}{h}\rb^{\beta''}
		\omega_h^B(x,y,z)
		\\
		&\times
		\frac{D_{\alpha,\beta,\wmu,\wnu,3}(x,y,z)}
		{t_1^{q_1}t_2^{q_2}t_3^{q_3}}
		(\partial_x^\alpha\partial_v^{\wmu}f)(X-t_{(3)}Y)
		(\partial_x^\beta\partial_v^{\wnu}g)(X-t_{(3)}Z)
		\dd Y\d Z\d t_1\d t_2\d t_3\,.
	\end{aligned}
\end{equation*}
Since
\begin{equation*}
	\left(\frac{y}{h}\right)^\alpha
	\left(\frac{z}{h}\right)^\beta
	\ee^{-\frac{4\pi i}{h}(\eta\cdot z-y\cdot\zeta)}
	=
	\frac{(-1)^{|\beta|}}
	{(4\pi i)^{|\alpha+\beta|}}
	\partial_\zeta^\alpha\partial_\eta^\beta
	\ee^{-\frac{4\pi i}{h}(\eta\cdot z-y\cdot\zeta)},
\end{equation*}
integration by parts shows that $h^3R_3(f,g)$ is a finite linear
combination of integrals of the following forms:
\begin{equation*}
\begin{aligned}
	\widetilde{J}_1
	={}&h^3\lb\frac{2}{h}\rb^{2d}
	\int_0^1\int_0^1\int_0^1
	\iint_{\Xi\times\Xi}
	t_1^{p_1}t_2^{p_2}t_3^{p_3}
	\ee^{-\frac{4\pi i}{h}\sigma(Y,Z)}
	\omega_h^B(x,y,z)
	D_{\alpha,\beta,\wmu,\wnu,1}(x,y,z)
	\\
	&\times
	(\partial_x^\alpha\partial_v^{\wmu}f)(X-t_{(3)}Y)
	(\partial_x^\beta\partial_v^{\wnu}g)(X-t_{(3)}Z)
	\,\mathrm{d}Y\,\mathrm{d}Z\,
	\mathrm{d}t_1\,\mathrm{d}t_2\,\mathrm{d}t_3\,,
	\\
	\widetilde{J}_2
	={}&h^3\lb\frac{2}{h}\rb^{2d}
	\int_0^1\int_0^1\int_0^1
	\iint_{\Xi\times\Xi}
	t_1^{p_1}t_2^{p_2}t_3^{p_3}
	\ee^{-\frac{4\pi i}{h}\sigma(Y,Z)}
	\omega_h^B(x,y,z)
	D_{\alpha,\beta,\wmu,\wnu,2}(x,y,z)
	\\
	&\times
	(\partial_x^\alpha\partial_v^{\wmu+\beta'}f)
	(X-t_{(3)}Y)
	(\partial_x^\beta\partial_v^{\wnu+\alpha'}g)
	(X-t_{(3)}Z)
	\,\mathrm{d}Y\,\mathrm{d}Z\,
	\mathrm{d}t_1\,\mathrm{d}t_2\,\mathrm{d}t_3\,,
	\\
	\widetilde{J}_3
	={}&h^2\lb\frac{2}{h}\rb^{2d}
	\int_0^1\int_0^1\int_0^1
	\iint_{\Xi\times\Xi}
	t_1^{p_1}t_2^{p_2}t_3^{p_3}
	\ee^{-\frac{4\pi i}{h}\sigma(Y,Z)}
	\omega_h^B(x,y,z)
	D_{\alpha,\beta,\wmu,\wnu,3}(x,y,z)
	\\
	&\times
	(\partial_x^\alpha\partial_v^{\wmu+\beta''}f)
	(X-t_{(3)}Y)
	(\partial_x^\beta\partial_v^{\wnu+\alpha''}g)
	(X-t_{(3)}Z)
	\,\mathrm{d}Y\,\mathrm{d}Z\,
	\mathrm{d}t_1\,\mathrm{d}t_2\,\mathrm{d}t_3\,.
\end{aligned}
\end{equation*}
Here $p_1,p_2,p_3$ are finite nonnegative integers.

We introduce the following operators, which fix
$\ee^{-\frac{4\pi i}{h}(\eta\cdot z-y\cdot\zeta)}$:
\begin{equation*}
\begin{gathered}
	L_{y,h}
	=
	\lal\frac{y}{h}\ral^{-2}
	\left(
	\operatorname{Id}
	-
	\frac{1}{(4\pi)^2}\Delta_\zeta
	\right),
	\qquad
	L_{\eta,h}
	=
	\lal\eta\ral^{-2}
	\left(
	\operatorname{Id}
	-
	\left(\frac{h}{4\pi}\right)^2\Delta_z
	\right),
	\\
	L_{z,h}
	=
	\lal\frac{z}{h}\ral^{-2}
	\left(
	\operatorname{Id}
	-
	\frac{1}{(4\pi)^2}\Delta_\eta
	\right),
	\qquad
	L_{\zeta,h}
	=
	\lal\zeta\ral^{-2}
	\left(
	\operatorname{Id}
	-
	\left(\frac{h}{4\pi}\right)^2\Delta_y
	\right).
\end{gathered}
\end{equation*}
After applying powers of these operators to the exponential factor
and integrating by parts, their formal adjoints act on the amplitude.
Expanding the resulting powers of
$\operatorname{Id}-(h/(4\pi))^2\Delta_y$ and
$\operatorname{Id}-(h/(4\pi))^2\Delta_z$, we obtain integers
$\ell_y$ and $\ell_z$ satisfying
\begin{equation*}
	0\leq\ell_y\leq N_y,
	\qquad
	0\leq\ell_z\leq N_z.
\end{equation*}
Consequently, $h^3R_3(f,g)$ is a finite linear combination of terms
of the form
\begin{equation*}
\begin{aligned}
	K_j
	={}&
	h^{\alpha(j)}\lb\frac{2}{h}\rb^{2d}
	\int_0^1\int_0^1\int_0^1
	\iint_{\Xi\times\Xi}
	t_1^{q_1}t_2^{q_2}t_3^{q_3}
	\ee^{-\frac{4\pi i}{h}\sigma(Y,Z)}
	\lal\frac{z}{h}\ral^{-2N_\eta}
	\lal\zeta\ral^{-2N_y}
	\\
	&\times
	\left(
	\partial_\eta^{\gamma'}
	\lal\eta\ral^{-2N_z}
	\right)
	\left(
	h^{|\delta'|}
	\partial_y^{\delta'}
	\lal\frac{y}{h}\ral^{-2N_\zeta}
	\right)
	\\
	&\times
	h^{|\delta''|+|\epsilon'|}
	\partial_y^{\delta''}\partial_z^{\epsilon'}
	\left(
	\omega_h^B(x,y,z)
	D_{\alpha,\beta,\wmu,\wnu,j}(x,y,z)
	\right)
	\\
	&\times
	h^{2\ell_y-|\delta'|-|\delta''|}
	h^{2\ell_z-|\epsilon'|}
	(\partial_x^{\rho_x+\delta'''}
	\partial_v^{\rho_v+\gamma''}f)(X-t_{(3)}Y)
	\\
	&\times
	(\partial_x^{\tau_x+\epsilon''}
	\partial_v^{\tau_v+\lambda}g)(X-t_{(3)}Z)
	\dd Y\d Z
	\d t_1\d t_2\d t_3\,,
\end{aligned}
\end{equation*}
where $j=1,2,3$,
\begin{equation*}
	\alpha(1)=\alpha(2)=3,
	\qquad
	\alpha(3)=2,
\end{equation*}
and
\begin{equation*}
	|\gamma'|+|\gamma''|\leq2N_\eta,
	\qquad
	|\delta'|+|\delta''|+|\delta'''|
	\leq2\ell_y,
\end{equation*}
\begin{equation*}
	|\epsilon'|+|\epsilon''|\leq2\ell_z,
	\qquad
	|\lambda|\leq2N_\zeta.
\end{equation*}
All the derivatives appearing here are required to lie within the
finite regularity of $f$, $g$, and $B$. In particular,
\begin{equation*}
	|\delta''|+|\epsilon'|\leq m-3.
\end{equation*}

Set $r:=|\delta''|+|\epsilon'|.$
The estimates for the magnetic factor and the boundedness of the
derivatives of
$D_{\alpha,\beta,\wmu,\wnu,j}$ give, for $0<h\leq1$,
\begin{equation*}
	h^r
	\left|
	\partial_y^{\delta''}\partial_z^{\epsilon'}
	\left(
	\omega_h^B
	D_{\alpha,\beta,\wmu,\wnu,j}
	\right)(x,y,z)
	\right|
	\leq
	C\lal(y,z)\ral^{2r}.
\end{equation*}
Moreover, we have
\begin{equation*}
	\lal(y,z)\ral^{2r}
	\leq
	C
	\lal\frac{y}{h}\ral^{2r}
	\lal\frac{z}{h}\ral^{2r}.
\end{equation*}
Since all the remaining powers of $h$ in the expression for $K_j$
are nonnegative and $\alpha(j)\geq2$, every $K_j$ is bounded
pointwise by a constant times
\begin{equation*}
\begin{aligned}
	I
	={}&
	h^2\lb\frac{2}{h}\rb^{2d}
	\int_0^1\int_0^1\int_0^1
	\iint_{\Xi\times\Xi}
	t_1^{q_1}t_2^{q_2}t_3^{q_3}
	\\
	&\times
	\lal\frac{z}{h}\ral^{-2N_\eta+2r}
	\lal\zeta\ral^{-2N_y}
	\lal\eta\ral^{-2N_z-|\gamma'|}
	\lal\frac{y}{h}\ral^{-2N_\zeta-|\delta'|+2r}
	\\
	&\times
	\lal x-t_{(3)}y\ral^{
	m_1-|\rho_x+\delta'''|}
	\lal v-t_{(3)}\eta\ral^{
	m_2-|\rho_v+\gamma''|}
	\\
	&\times
	\lal x-t_{(3)}z\ral^{
	\widetilde m_1-|\tau_x+\epsilon''|}
	\lal v-t_{(3)}\zeta\ral^{
	\widetilde m_2-|\tau_v+\lambda|}
	\dd Y\d Z
	\d t_1\d t_2\d t_3.
\end{aligned}
\end{equation*}

We use the elementary estimate
\begin{equation}\label{eq:weighted-convolution-estimate}
	\int_{\rr^d}
	\lal y\ral^{-A}\lal x-sy\ral^a
	\,\mathrm{d}y
	\leq
	C_{A,a,d}\lal x\ral^a,
	\qquad 0\leq s\leq1,
\end{equation}
which holds whenever $A>d+|a|$. 
Indeed, this follows directly from the inequality $\langle x+y\rangle^m\leq 2^{|m|/2}\langle x\rangle^m\langle y\rangle^{|m|}$.

After making the changes of variables $y\mapsto hy$ and
$z\mapsto hz$, the Jacobian cancels the factor $h^{-2d}$ in
$(2/h)^{2d}$. We then apply
\eqref{eq:weighted-convolution-estimate} successively in
$y,z,\eta,\zeta$. It is therefore sufficient to choose the finite
integers $N_y,N_\eta,N_z,N_\zeta$ so that, for every term appearing
above,
\begin{equation*}
	2N_\zeta+|\delta'|-2r
	>
	d+
	\left|
	m_1-|\rho_x+\delta'''|
	\right|,
\end{equation*}
\begin{equation*}
	2N_\eta-2r
	>
	d+
	\left|
	\widetilde m_1-|\tau_x+\epsilon''|
	\right|,
\end{equation*}
\begin{equation*}
	2N_z+|\gamma'|
	>
	d+
	\left|
	m_2-|\rho_v+\gamma''|
	\right|,
\end{equation*}
and
\begin{equation*}
	2N_y
	>
	d+
	\left|
	\widetilde m_2-|\tau_v+\lambda|
	\right|.
\end{equation*}
There are only finitely many terms, so such finite integers can be
chosen, provided that $f$, $g$, and $B$ have the corresponding finite
regularity.

The powers $t_1^{q_1}t_2^{q_2}t_3^{q_3}$ are integrable because
$q_1,q_2,q_3$ are nonnegative. We consequently obtain
\begin{equation*}
	\left|h^3R_3(f,g)(X)\right|
	\leq
	Ch^2
	\lal x\ral^{m_1+\widetilde m_1}
	\lal v\ral^{m_2+\widetilde m_2},
	\qquad 0<h\leq1,
\end{equation*}
where $C$ depends on the required finite symbol seminorms of $f$ and
$g$ and on the required $W^{m,\infty}$ norm of $B$.

The same argument applies to $R_3(g,f)$ provided that the
corresponding finite regularity and decay assumptions also hold after
interchanging $f$ and $g$.	    
\end{proof}

	\subsection{Calculations of Hamiltonian and the derivation of the quantized equation}\label{Subsection:ExplicitCalculation}
	The first calculation is about the magnetic Weyl quantization of the Hamiltonian $\frac{1}{2}\snorm{v}^2$.
	\begin{lemma}\label{Lemma:TransformOfHamiltonian}
		Consider the Hamiltonian $\frac{1}{2}\snorm{v}^2$ defined on $\Xi$. Then, in the sense of distributions, we have the following
		\begin{equation*}
			\OPAh\left(\tfrac{1}{2}\snorm{v}^2\right)=\tfrac{1}{2}\left(\opp-\bA\right)^2\,.
		\end{equation*}
	\end{lemma}
	\begin{proof}
		We take a Schwartz function $u\in\mathcal{S}(\rr^d)$. Direct calculation gives 
		\begin{equation*}
			\begin{aligned}
				\OPAh\left(\tfrac{1}{2}\snorm{v}^2\right)u
				=&\frac{h^2}{2}\iint_{\rr^d\times\rr^d}e^{-2\pi i\,v\cdot(y-x)}e^{-2\pi i\frac{1}{h}\Gamma^{A}([x,y])}\snorm{v}^2u(y)\,\mathrm{d} y\mathrm{d}v\\
				=&-\frac{\hbar^2}{2}
				\iint_{\rr^d\times\rr^d}
				e^{-2\pi i\,v\cdot(y-x)}
				\Big(
				\Delta_y\left(
				e^{-2\pi i\frac{1}{h}\Gamma^{A}([x,y])}
				\right)
				u(y)
				\\&+2\nabla_y\left(e^{-2\pi i\frac{1}{h}\Gamma^{A}([x,y])}\right)\cdot\nabla u(y)
				+e^{-2\pi i\frac{1}{h}\Gamma^{A}([x,y])}\Delta u(y)
				\Big)
				\,\mathrm{d} y\mathrm{d}v	\,.
			\end{aligned}
		\end{equation*}
		We first consider $\Delta_y\left(e^{-2\pi i\frac{1}{h}\Gamma^{A}([x,y])}\right)$. Direct computation yields
		\begin{equation*}
			\begin{aligned}
				\Delta_y\left(e^{-2\pi i\frac{1}{h}\Gamma^{A}([x,y])}\right)=\left(\left(\frac{2\pi i}{h}\right)^2\snorm{\nabla_y\Gamma^{A}([x,y])}^2-\frac{2\pi i}{h}\Delta_y(\Gamma^{A}([x,y]))\right)e^{-2\pi i\frac{1}{h}\Gamma^{A}([x,y])}\,.
			\end{aligned}
		\end{equation*}
		By a change of variables, we have
		\begin{equation*}
			\begin{aligned}
				\Gamma^{A}([x,y])=\sum^d_{k=1}\int^1_0(y_k-x_k)\bA_k(x+t(y-x))\,\mathrm{d} t \,.
			\end{aligned}
		\end{equation*}
		Hence, it follows that
		\begin{equation*}
			\begin{aligned}
				\frac{\partial}{\partial{y_k}}\Gamma^{A}([x,y])=&\, \int^1_0{\bA}_k(x+t(y-x))\,\mathrm{d} t+\sum_{j=1}^d\int^1_0t(y_j-x_j)(\partial_k\bA_j)(x+t(y-x))\,\mathrm{d} t
				\\
				=&\, \bA_k(y)
				+\sum_{j=1}^d\int^1_0 t(y_j-x_j)\bB_{kj}(x+t(y-x))\,\mathrm{d} t\,.
			\end{aligned}
		\end{equation*}
		Furthermore, by the anti-symmetry of $B$ (the fact that $\bB_{kk}=0$), we have that
		\begin{equation*}
			\begin{aligned}
				\frac{\partial^2}{\partial{y_k}^2}\Gamma^{A}([x,y])=\partial_k\bA_k(y)+\sum^d_{j=1}\int^1_0t^2(y_j-x_j)\partial_k\bB_{kj}(x+t(y-x))\,\mathrm{d} t\,.
			\end{aligned}
		\end{equation*}
		Consequently, (in the sense of distributions) we have
		\begin{equation*}
			\begin{aligned}
				&\OPAh\left(\tfrac{1}{2}\snorm{v}^2\right)u\\
				&=-\frac{\hbar^2}{2}\int_{\rr^d}\delta(y-x)\Big(\Delta_y\left(
				e^{-2\pi i\frac{1}{h}\Gamma^{A}([x,y])}
				\right)
				u(y)
				\\&\ \ \ \ +2\nabla_y\big(e^{-2\pi i\frac{1}{h}\Gamma^{A}([x,y])}\big)\cdot\nabla u(y)
				+e^{-2\pi i\frac{1}{h}\Gamma^{A}([x,y])}\Delta u(y)\Big)\,\mathrm{d} y
				\\&=-\frac{\hbar^2}{2}\left[-\frac{2\pi i}{h}(\nabla\cdot\bA(x))u(x)+\left(\frac{2\pi i}{h}\right)^2\snorm{\bA(x)}^2u(x)-\frac{4\pi i}{h}\bA(x)\cdot\nabla u(x)+\Delta u(x)\right]
				\\&=\tfrac{1}{2}(-i\hbar\nabla-\bA(x))^2u(x)\,,
			\end{aligned}
		\end{equation*}
		which gives the desired result.
	\end{proof}

	We now derive formula \eqref{equation:ControlOfErrorHartree}. 
	As a corollary of Proposition~\ref{Propositon:QuantizationOfCommutator} and Lemma
	~\ref{Lemma:TransformOfHamiltonian}, we have
	\begin{equation*}
		\left[\tfrac{1}{2}(\bp-A)^2,\brhoAf\right]=-\frac{\hbar}{i}\OPAh\lb\left\{\tfrac{1}{2}\snorm{v}^2,f\right\}_B\rb+\hbar^2\OPAh(\widetilde{R})\,.
	\end{equation*}
	It's easy to compute that for any test function $u\in \mathcal{S}(\rr^d)$ we have
	\begin{equation*}
		\begin{aligned}
			&\OPAh\lb\left\{K*\rho_f,f\right\}_B\rb u\\
			&=\OPAh\lb\left\{K*\rho_f,f\right\}\rb u=\OPAh(\nabla_xK*\rho_f\cdot\nabla_vf)u\\
			&=\int_{\rr^d\times\rr^d}\ee^{-2\pi iv(y-x)-\frac{2\pi i}{h}\Gamma^A([x,y])}\lb\nabla_xK*\rho_f\rb\lb\tfrac{x+y}{2}\rb\cdot\nabla_v\lb f\lb\tfrac{x+y}{2},hv\rb\rb u(y)\,\mathrm{d} y\mathrm{d}v\\
			&=\int_{\rr^d\times\rr^d}\ee^{-2\pi iv(y-x)-\frac{2\pi i}{h}\Gamma^A([x,y])}(2\pi i)(y-x)\lb\nabla_xK*\rho_f\rb\lb\tfrac{x+y}{2}\rb\frac{1}{h}f\lb\tfrac{x+y}{2},hv\rb u(y)\,\mathrm{d} y\mathrm{d}v\,.
		\end{aligned}
	\end{equation*}
	Therefore, $\OPAh\lb\left\{K*\rho_f,f\right\}_B\rb=\frac{2\pi i}{h}(y-x)\lb\nabla_xK*\rho_f\rb\lb\frac{x+y}{2}\rb\brhoAf(x,y)$. Finally,
	\begin{equation*}
		\begin{aligned}
			&i\hbar\partial_t\lb\mU^*(\brho-\brhoAf)\mU\rb\\
			&=-\mU^*\lb\lb\tfrac{1}{2}(\bp-A)^2+V_{\brho}-h^d\bXr\rb(\brho-\brhoAf)\rb\mU\\
			&\ \ \ \ +\mU^*\lb\left[\tfrac{1}{2}(\bp-A)^2+V_{\brho}-h^d\bXr,\brho\right]-\left[\tfrac{1}{2}(\bp-A)^2,\brhoAf\right]\rb\mU\\
			&\ \ \ \ -\mU^*\left[i\hbar\OPAh\lb\left\{K*\rho_f,\brhoAf\right\}\rb,\brhoAf\right]\mU-\hbar^2\mU^*\OPAh(\widetilde{R})\mU\\
			&\ \ \ \ +\mU^*(\brho-\brhoAf)\lb\tfrac{1}{2}(\bp-A)^2+V_{\brho}-h^d\bXr\rb\mU\\
			&=-\hbar^2\mU^*\OPAh(\widetilde{R})\mU-h^d\mU^*\left[\bXr,\brhoAf\right]\mU\\
			&\ \ \ \ +\mU^*\left[K*(\rho-\rho_f),\brhoAf\right]\mU\\
			&\ \ \ \ +\mU^*\left[K*\rho_f-i\hbar\OPAh(\left\{K*\rho_f,\brhoAf\right\}),\brhoAf\right]\mU\\
			&=\mU^*\left[K*(\rho-\rho_f),\brhoAf\right]-h^d\mU^*\left[\bXr,\brhoAf\right]\mU +\mU^*{\sf{B}}_t\mU\\
			&\ \ \ \ -\hbar^2\mU^*\OPAh(\widetilde{R})\mU\,.
		\end{aligned}
	\end{equation*}
	Then we obtain the quantized equation \eqref{equation:ControlOfErrorHartree}.
	
	\section{Wellposedness of the Magnetic Hartree--Fock equation}\label{Section:ExistenceResults}
	The aim of this section is to present some well-posedness results for the magnetic Hartree--Fock equation. We shall first prove the local existence of solutions to the magnetic Hartree--Fock equation in a specific space containing the magnetic Sobolev space. Then we use the conservation of energy to obtain the global existence of solutions. For simplicity, we focus only on the existence results in the three-dimensional case and consider a time-independent but spatially dependent vector potential $\bA$. Moreover, since we are only concerned with existence of solutions for a fixed $\hbar$, we shall simply set $\hbar=1$ in the following discussion.
	
	Before we study the existence of solutions, we recall some facts about magnetic Sobolev spaces and the corresponding embedding results. For more details, see, e.g., \cite[Chapter 7]{Lieb2001-yq} or \cite{frank2008hardy}.
	
	\begin{definition}
		Assume the vector potential $\bA =(A_1, \ldots, A_d)$ satisfies $\bA_j\in L^2_{\mathrm{loc}}(\rr^d)$ for $j=1,\dots,d$. Then the magnetic Sobolev space $\mathcal{H}^1_{A}(\rr^d)$ consists of all functions $f:\rr^d\to \mathbb{C}$ such that $f\in L^2(\rr^d)$ and $(-i\nabla-\bA)f\in L^2(\rr^d)$. We endow $\cH^1_{A}$ with the inner product 
		\begin{equation*}
			\langle f_1,f_2\rangle_{A}:=\langle f_1,f_2\rangle+\sum_{j=1}^{d}\langle (i\partial_j+\bA_j)f_1,(i\partial_j+\bA_j)f_2\rangle\,,
		\end{equation*}
		where $\langle\cdot,\cdot\rangle$ denotes the usual $L^2$ inner product.
	\end{definition}
	
	\begin{theorem}[Diamagnetic inequality \cite{Lieb2001-yq, frank2008hardy}]\label{Theorem:DiamagneticInequality}
		Assume the vector potential $\bA =(A_1, \ldots, A_d)$ satisfies $\bA_j\in L^2_{\mathrm{loc}}(\rr^d)$ for $j=1,\dots,d$ and let $f\in\mathcal{H}^1_{A}$. Then the absolute value $\snorm{f}$ belongs to $H^1(\rr^d)$ and satisfies the following pointwise inequality in $\rr^d$:
		\begin{equation*}
			\snorm{\nabla\snorm{f}}\le \snorm{(i\nabla+\bA)f}.
		\end{equation*}
		Moreover, for $0<s\le 1$ and any function $u\in \mathcal{H}^1_{A}$ we also have
		\begin{equation*}
			\lnorm{2}{(-\Delta)^{\frac{s}{2}}\snorm{u}}\le \lnorm{2}{\snorm{i\nabla+\bA}^s u}.
		\end{equation*}
	\end{theorem}
	
	We shall investigate the existence of solutions to the magnetic Hartree–Fock equation in three-dimensional space. We define the self-adjoint operator $\ma:=(\Id+(i\nabla+\bA)^2)^{\frac{1}{2}}$. Since $(i\nabla+\bA)^2$ is a positive self-adjoint operator, it follows that $\sigma(\ma)\subseteq [1,\infty)$. Consequently, $\ma^{-1}$ is a bounded operator on $L^2(\rr^3)$. Here, we identify the vector potential $\bA$ with a vector in $\rr^3$.

	We define the Banach space $\za$ as
	\begin{equation*}
		\begin{aligned}
			\za:=\left\{\gamma\in \mathfrak{S}^1:\gamma^*=\gamma,\ \znorm{\gamma}:=\cstnorm{1}{\ma\gamma}+\cstnorm{1}{\gamma\ma}<\infty\right\}\,,
		\end{aligned}
	\end{equation*}
	and the energy space $\ya$ as
	\begin{equation*}
		\begin{aligned}
			\ya:=\left\{\gamma\in\mathfrak{S}^1:\gamma^*=\gamma,\ \ynorm{\gamma}:=\cstnorm{1}{\ma\gamma\ma}<\infty\right\}\,.
		\end{aligned}
	\end{equation*}
	These two spaces satisfy the continuous embedding $\ya \hookrightarrow\za$. Indeed, we have
	\begin{equation*}
		\begin{aligned}
			\cstnorm{1}{\ma\gamma}&\le \cstnorm{1}{\ma\gamma\ma}\norm{\ma^{-1}}_{\infty}\le \cstnorm{1}{\ma\gamma\ma}\,,\\
			\cstnorm{1}{\gamma\ma}&\le \cstnorm{1}{\ma\gamma\ma}\norm{\ma^{-1}}_{\infty}\le \cstnorm{1}{\ma\gamma\ma}\,.
		\end{aligned}
	\end{equation*}
	Hence, following the approach in \cite{benedikter2018dirac}, we first investigate the local existence of solutions in the Banach space $\za$ and then extend the solution to global existence in the energy space $\ya$.
	
	From the fractional Hardy--Rellich inequality in Theorem~\ref{Theorem:HardyRellichInequality} and the fractional diamagnetic inequality above, we obtain the following relation: for any $u\in \mathcal{S}(\rr^3)$,
	\begin{equation}\label{est:diamagnetic_Hardy}
		\int_{\rr^3}\frac{\snorm{u}^2}{\snorm{x}^2}\,\mathrm{d} x\le C\int_{\rr^3}\snorm{\nabla\snorm{u}}^2\,\mathrm{d} x \le C\int_{\rr^3}\snorm{(i\nabla+\bA)u}^2\,\mathrm{d} x\,.
	\end{equation}
	This implies $\mG^2\le C(i\grad+A)^2\le C\ma^{2}$, where $\mG$ denotes the Coulomb potential in $\rr^3$. Moreover, we also note the fact that there exists $C>0$ such that for any trace class operator $\brho\in \mathfrak{S}^1$, we have
	\begin{equation*}
		C^{-1}\norm{\brho}_{\mathcal{H}^1_{A}} \le \cstnorm{2}{\ma\brho}+\cstnorm{2}{\brho\ma}\le C \norm{\brho}_{\mathcal{H}^1_{A}}\,.
	\end{equation*}
	With these facts in mind, we can establish the following lemma, which is analogous to \cite[Lemma 5.7]{benedikter2018dirac} in the zero magnetic field case.
	
	\begin{lemma}\label{Lemma:estimateofoperators}
		For any $\brho\in\za$, we have
		\begin{align}
			\norm{V_{\brho}\,\ma^{-1}}_{\infty}\le&\, C\cstnorm{1}{\brho}, &\norm{V_{\brho}}_{\infty}\le&\, C\znorm{\brho}\,,\label{est:effective_potential}\\
			\cstnorm{2}{\bXr\,\ma^{-1}}\le\,& C\cstnorm{2}{\brho}, & \cstnorm{2}{\bXr}\le\,& C\norm{\brho}_{\mathcal{H}^1_{A}}\, \label{est:exchange_term},
		\end{align}
		where $C$ is a positive constant independent of $\bA$.
	\end{lemma}
	
	\begin{proof}
		For $\phi\in L^2(\rr^3)$, set $u=\ma^{-1}\phi$. Then, by the Cauchy–Schwarz inequality and inequality~\eqref{est:diamagnetic_Hardy}, we obtain
		\begin{equation*}
			\begin{aligned}
				\int_{\rr^3}\snorm{V_{\brho}(x)\, u(x)}^2\,\mathrm{d} x&=\int_{\rr^3}\snorm{\int_{\rr^3}\mG(x-y)\rho(y)\,\mathrm{d} y}^2\snorm{u(x)}^2\mathrm{d}x\\
				&\le \int_{\rr^3}\left\{\int_{\rr^3}\snorm{\mG(x-y)}^2\snorm{\rho(y)}\,\mathrm{d} y\int_{\rr^3}\rho(y')\,\mathrm{d} y'\right\}\snorm{u(x)}^2\,\mathrm{d} x\\
				&\le C\lnorm{1}{\rho}^2\lnorm{2}{(i\nabla+\bA)u}^2\le C\cstnorm{1}{\brho}^2\lnorm{2}{\phi}^2\,.
			\end{aligned}
		\end{equation*}
		Hence, we obtain $\norm{V_{\brho}\,\ma^{-1}}_{\infty}\le C\cstnorm{1}{\brho}$. 
		
		Using spectral decomposition, we write $\brho=\sum_j\lambda_j\ket{u_j}\!\!\bra{u_j}$ with orthonormal vectors $u_j\in L^2(\rr^3)$. Applying inequality~\eqref{est:diamagnetic_Hardy} term by term gives
		\begin{equation*}
			\begin{aligned}
				V_{\brho}(x)=&\, \sum_j\frac{\lambda_j}{4\pi}\int_{\rr^3}\frac{\snorm{u_j(y)}^2}{\snorm{x-y}}\,\mathrm{d} y
				\le C\sum_j\lambda_j\int_{\rr^3}\snorm{(i\nabla+\bA)u_j(x)}^2\,\mathrm{d} x\\
				=&\, C\cstnorm{1}{\ma^{\frac{1}{2}}\,\brho\,\ma^{\frac{1}{2}}}\le C\cstnorm{1}{\brho\,\ma}
			\end{aligned}
		\end{equation*}
		Thus, $\norm{V_{\brho}}_{\infty}\le C\znorm{\brho}$. 
		
		The second inequality in \eqref{est:exchange_term} is again a direct consequence of \eqref{est:diamagnetic_Hardy}:
		\begin{equation*}
			\begin{aligned}
				\cstnorm{2}{\bXr}^2=\iint_{\rr^3\times\rr^3}\snorm{\mG(x-y)}^2\snorm{\brho(x,y)}^2\,\mathrm{d} x\mathrm{d}y\le 
				C\norm{\brho}_{\mathcal{H}^1_{A}}\,.
			\end{aligned}
		\end{equation*}
		Finally, we prove the first inequality in \eqref{est:exchange_term}. Since $\ma^{-2}$ is a bounded operator on $L^2(\rr^3)$, it possesses an integral kernel, which we denote by $\ma^{-2}(x,y)$. Then we have
		\begin{equation*}
			\begin{aligned}
				\operatorname{Tr}\left(\bXr\ma^{-2}\bXr\right)=\iiint_{\rr^3\times\rr^3\times\rr^3}\brho(x,y)\mG(x-y)\ma^{-2}(y-z)\brho(x,z)\mG(x-z)\,\mathrm{d} x\mathrm{d}y\mathrm{d}z\,.
			\end{aligned}
		\end{equation*}
		Define $g_x(y):=\mG(x-y)\brho(x,y)$, then we see that
		\begin{equation*}
			\begin{aligned}
				\operatorname{Tr}(\bXr\ma^{-2}\bXr)&=\int_{\rr^3}\inprod{g_x}{\ma^{-2}g_x}\mathrm{d}x\le C\int_{\rr^3}\inprod{g_x}{\mG^{-2}(x-\cdot)g_x}\mathrm{d}x\\
				&\le C\iint_{\rr^3\times\rr^3}\snorm{\brho(x,y)}^2\,\mathrm{d} x \mathrm{d}y\,.
			\end{aligned}
		\end{equation*}
		Hence, $\cstnorm{2}{\bXr\ma^{-1}}\le C\cstnorm{2}{\brho}$. This completes the proof.
	\end{proof}
	
	Now we define the following bilinear form:
	\begin{equation}\label{def:effective_bilinear_form}
		\mathsf{K}(\brho_1,\brho_2):=\left[V_{\brho_1}-\mathsf{X}_{\brho_1},\brho_2\right]\,.
	\end{equation}
	We consider the continuity of the bilinear form \eqref{def:effective_bilinear_form} on $\za$. A direct calculation yields
	\begin{equation*}
		\begin{aligned}
			\cstnorm{1}{\ma \left[V_{\brho_1},\brho_2\right]}&\le \cstnorm{1}{\ma \,V_{\brho_1}\,\brho_2}+\cstnorm{1}{\ma\,\brho_2\,V_{\brho_1}}\\
			&\le \norm{\ma V_{\brho_1}\ma^{-1}}_{\infty}\cstnorm{1}{\ma\brho_2}+\cstnorm{1}{\ma\brho_2}\norm{V_{\brho_1}}_{\infty}\,.
		\end{aligned}
	\end{equation*}
	Notice that 
	\begin{equation*}
		\begin{aligned}
			\ma V_{\brho_1}\ma^{-1} =& \ma^{-1} V_{\brho_1}\ma^{-1}+\ma^{-1} H^{A}_{0} V_{\brho_1}\ma^{-1} \\
			=&\ma^{-1} V_{\brho_1}\ma^{-1} +\sum_{j=1}^3\ma^{-1} (i\partial_j+\bA_j)\left[i\partial_j+\bA_j,V_{\brho_1}\right]\ma^{-1}\\
			&+\sum_{j=1}^3 \ma^{-1}(i\partial_j+\bA_j)V_{\brho_1}(i\partial_j+\bA_j)\ma^{-1}\,,
		\end{aligned}
	\end{equation*}
	then it follows
	\begin{equation*}
		\begin{aligned}
			\norm{\ma V_{\brho_1}\ma^{-1}}_{\infty}\le& \norm{\ma^{-1} V_{\brho_1}\ma^{-1}}_{\infty}+\sum_{j=1}^3\norm{\ma^{-1}(i\partial_j+\bA_j)}^2_{\infty}\norm{V_{\brho_1}}_{\infty}\\
			&+\sum_{j=1}^3\norm{\ma^{-1}(i\partial_j+\bA_j)}_{\infty}\norm{V_{[i\partial_j+\bA_j,\brho_1]}\ma^{-1}}_{\infty}\\
			\le&C\znorm{\brho_1}+C\norm{[i\partial_j+\bA_j,\brho_1]}_{1}\\
			\le&C\znorm{\brho_1}+C\cstnorm{1}{\ma\brho_1}+C\cstnorm{1}{\brho_1\ma}\le C\znorm{\brho_1}\,,
		\end{aligned}
	\end{equation*}
	where we used Lemma~\ref{Lemma:estimateofoperators} and the following relations:
	\begin{equation*}
		[\partial_j,V_{\brho_1}]=V_{[\partial_j,\brho_1]}\quad \text{ and } \quad
		[A_j,V_{\brho_1}]=0\,.
	\end{equation*}
	Thus, we have $\cstnorm{1}{\ma \left[V_{\brho_1},\brho_2\right]}\le C\znorm{\brho_1}\znorm{\brho_2}$. Similarly, we can also obtain $$\cstnorm{1}{ \left[V_{\brho_1},\brho_2\right]\ma}\le C\znorm{\brho_1}\znorm{\brho_2}\,,$$
	and consequently $\znorm{\left[V_{\brho_1},\brho_2\right]}\le C\znorm{\brho_1}\znorm{\brho_2}$. 
	
	For the exchange term, we have
	\begin{equation*}
		\begin{aligned}
			\cstnorm{1}{\ma\left[\mathsf{X}_{\brho_1},\brho_2\right]}&\le \cstnorm{1}{\ma\mathsf{X}_{\brho_1}\brho_2}+\cstnorm{1}{\ma\brho_2\bX_{\brho_1}}\\
			&\le \cstnorm{2}{\ma\bX_{\brho_1}\ma^{-1}}\cstnorm{2}{\ma\brho_2}+\cstnorm{2}{\ma\brho_2}\cstnorm{2}{\bX_{\brho_1}}\,.
		\end{aligned}
	\end{equation*}
	By applying the same argument as above, we have that
	\begin{equation*}
		\begin{aligned}
			\cstnorm{2}{\ma\bX_{\brho_1}\ma^{-1}}
			\le& \cstnorm{2}{\ma^{-1}\bX_{\brho_1}\ma^{-1}}+\sum_{j=1}^3\cstnorm{2}{\ma^{-1}(i\partial_j+\bA_j)\left[i\partial_j+\bA_j,\bX_{\brho_1}\right]\ma^{-1}}\\
			&+\sum_{j=1}^3\cstnorm{2}{
				\ma^{-1}(i\partial_j+\bA_j)\bX_{\brho_1}(i\partial_j+\bA_j)\ma^{-1}}\\
			\le&\cstnorm{2}{\bX_{\brho_1}}\norm{\ma^{-1}}^2_{\infty}+C\sum_{j=1}^3\norm{\ma^{-1}(i\partial_j+\bA_j)}^2_{\infty}\cstnorm{2}{\bX_{\brho_1}}\\
			&+\sum_{j=1}^3\norm{\ma^{-1} (i\partial_j+\bA_j)}_{\infty}\cstnorm{2}{\mathsf{X}_{[i\partial_j+\bA_j,\brho_1]}\ma^{-1}}\\
			\le& C\znorm{\brho_1}+C\cstnorm{2}{[i\partial_j+\bA_j,\brho_1]}\le C\znorm{\brho_1}\,.
		\end{aligned}
	\end{equation*}
	Thus, $\cstnorm{1}{\ma\left[\bX_{\brho_1},\brho_2\right]}\le C\znorm{\brho_1}\znorm{\brho_2}$. Finally, we obtain 
	\begin{equation}\label{est:sfK_continuity}
		\znorm{\mathsf{K}(\brho_1,\brho_2)}\le C\znorm{\brho_1}\znorm{\brho_2}\,,
	\end{equation}
	which asserts the continuity of the bilinear form $\mathsf{K}$. With this in mind, we can prove the following global existence result.
	
	\begin{theorem}
		Suppose $\brho^{\mathrm{in}}\in \ya$ and $\brho^{\mathrm{in}}\ge 0$. Then for any finite $T>0$, the magnetic Hartree--Fock equation with initial data $\brho^{\mathrm{in}}$ has a global solution in $L^{\infty}_{T}\ya$.
	\end{theorem}
	
	\begin{proof}
		The proof is divided into the following steps:
		
		\medskip
		\noindent\textbf{Step 1.} \textit{Duhamel's principle and conservation of positivity.} Fix 
		$T>0$ (whose value will be determined later) and recall $H_0^A = \tfrac12 (i\grad+A)^2$. 
		Given the initial data $\brho^{\mathrm{in}}\in \za$, we define the mapping
		\begin{equation*}
			\Phi(\brho)=e^{-it H^{A}_{0}}\brho^{\mathrm{in}}e^{it H^{A}_{0}}-i\int^t_0e^{-i(t-s) H^{A}_{0}}\mathsf{K}(\brho,\brho)e^{i(t-s) H^{A}_{0}}\,\mathrm{d} s\,
		\end{equation*}
		on $L^\infty_T\za$.
		We shall prove the local existence of a solution in $\za$ by applying the fixed point theorem. Besides if the global solution $\brho$ exists, then it conserves the positivity. Since if we denote the unitary propagator of this solution by $U(t,s)$, i. e.
		\begin{equation*}
			\begin{aligned}
				i\partial_tU(t,s)=H(t)U(t,s),\qquad U(s,s)=\mathrm{Id},\qquad H(t):=H^A_0+V_{\brho}-\mathsf{X}_{\brho}\,.
			\end{aligned}
		\end{equation*}
		Then $\brho(t)=U(t,0)\brho^{\mathrm{in}}U(t,0)^\ast$, which obviously conserves the positivity.
		
		\medskip
		\noindent\textbf{Step 2.} \textit{Contraction and local existence.}
		Consider the closed subset $E_T\subset L^\infty_T\za$ given by 
		$$E_{T}:=\left\{\brho\in L^\infty_T\za:\norm{\brho}_{L^\infty_T\za}\le a\right\}\,.$$
		Here, $a$ is a positive constant that will be determined shortly. By the continuity of $\sfK$ given in \eqref{est:sfK_continuity} and the fact that 
		\begin{equation*}
			\znorm{e^{-it H^{A}_{0}}\brho^{\mathrm{in}}e^{it H^{A}_{0}}} = \znorm{\brho^{\mathrm{in}}}\,,
		\end{equation*}
		we obtain the bound
		\begin{equation*}
			\begin{aligned}
				\norm{\Phi(\brho)}_{L^\infty_T\za}\le \znorm{\brho^{\mathrm{in}}}+CT\norm{\brho}^2_{L^\infty_T\za}\le \znorm{\brho^{\mathrm{in}}}+CTa^2\,.
			\end{aligned}
		\end{equation*}
		Now we take $a=2\znorm{\brho^{\mathrm{in}}}$. Hence, when $T\le \frac{1}{2C\znorm{{\brho^{\mathrm{in}}}}}$ the mapping $\Phi(\brho)$ maps $E_{T}$ into itself. Moreover for $\brho_1$, $\brho_2\in E_{T}$ we have
		\begin{equation*}
			\begin{aligned}
				\norm{\Phi(\brho_1)-\Phi(\brho_2)}_{L^\infty_T\za}&\le CT\left(\norm{\brho_1}_{L^\infty_T\za}+\norm{\brho_2}_{L^\infty_T\za}\right)\norm{\brho_1-\brho_2}_{L^\infty_T\za}\\&\le 2CT\times 2C\znorm{\brho^{\mathrm{in}}}\norm{\brho_1-\brho_2}_{L^\infty_T\za}\\&=4{C}^2T\znorm{\brho^{\mathrm{in}}}\norm{\brho_1-\brho_2}_{L^\infty_T\za}\,.
			\end{aligned}
		\end{equation*}
		Thus, once $T<\min\left\{\frac{1}{8{C}^2\znorm{\brho^{\mathrm{in}}}},\frac{1}{2C\znorm{\brho^{\mathrm{in}}}}\right\}$ the mapping $\Phi(\brho)$ is a contraction in the closed subspace $E_{T}$, which asserts the local existence of a solution for the magnetic Hartree--Fock equation in the space-time Banach space $L^\infty_T\za$. 

		\medskip
		\noindent\textbf{Step 3.} \textit{Control from the energy and global existence.}
		Now we suppose $\brho^{\mathrm{in}}\in\ya$ and $\brho^{\rm in} \ge 0$. To prove global existence we shall gain some a priori estimates from the energy, i.e.
		\begin{equation*}
			\begin{aligned}
				\mathcal{E}(\brho)=&\operatorname{Tr}(H_0^A\brho)+\frac{1}{2}\iint_{\rr^3\times\rr^3}K(x-y)\left(\brho(x, x)\brho(y, y)-\snorm{\brho(x,y)}^2\right)\mathrm{d}x \mathrm{d}y\,.
			\end{aligned}
		\end{equation*}
		If the operator $\brho$ satisfies $\cstnorm{1}{\ma\brho\ma}<\infty$, then, by Lemma~\ref{Lemma:estimateofoperators}, we have that
		\begin{equation*}
			\begin{aligned}
				\iint_{\rr^3\times\rr^3}\mG(x-y)\brho(x,x)\brho(y,y)\,\mathrm{d} x\mathrm{d}y&\le C\cstnorm{1}{\brho}\cstnorm{1}{\ma^{\frac{1}{2}}\brho\ma^{\frac{1}{2}}}
			\end{aligned}
		\end{equation*}
		and
		\begin{equation*}
			\begin{aligned}
				&\iint_{\rr^3\times\rr^3}\snorm{\brho(x,y)}^2\mG(x-y)\,\mathrm{d} x\mathrm{d}y\le \operatorname{Tr}(\brho\mG(x-y)\brho)\le C\operatorname{Tr}(\brho\ma\brho)\\&=C\operatorname{Tr}(\ma^{\frac{1}{2}}\brho^{\frac{1}{2}}\brho\brho^{\frac{1}{2}}\ma^{\frac{1}{2}})\le C\norm{\brho}_{\infty}\operatorname{Tr}(\ma^{\frac{1}{2}}\brho\ma^{\frac{1}{2}})\le C\cstnorm{1}{\brho}\cstnorm{1}{\ma^{\frac{1}{2}}\brho\ma^{\frac{1}{2}}}\,.
			\end{aligned}
		\end{equation*}
		Next, we control $\cstnorm{1}{\ma^{\frac{1}{2}}\brho\ma^{\frac{1}{2}}}$. Notice, we have
		\begin{equation*}
			\begin{aligned}
				\operatorname{Tr}(\ma^{\frac{1}{2}}\brho\ma^{\frac{1}{2}})&=\sum_j\lambda_j\lnorm{2}{\ma^{\frac{1}{2}}u_j}^2=\sum_j\lambda_j\inprod{u_j}{\ma u_j}\\&\le \sum_j\lambda_j\left(\delta\lnorm{2}{\ma u_j}^2+\frac{\lnorm{2}{u_j}^2}{2\delta}\right)\le \delta \cstnorm{1}{\ma\brho\ma}+\frac{\cstnorm{1}{\brho}}{\delta}\,.
			\end{aligned}
		\end{equation*}
		Consequently, it follows that
		\begin{equation*}
			\begin{aligned}
				\mathcal{E}(\brho)&\ge\cstnorm{1}{\ma\brho\ma}-\cstnorm{1}{\brho}-2C\delta\cstnorm{1}{\brho}\cstnorm{1}{\ma\brho\ma}-2C\frac{\cstnorm{1}{\brho}^2}{\delta}\,.
			\end{aligned}
		\end{equation*}
		Hence, when $\delta$ is sufficiently small we have
		\begin{equation*}
			\begin{aligned}
				\cstnorm{1}{\ma\brho\ma}&\le\frac{\mathcal{E}(\brho)+\cstnorm{1}{\brho}\left(1+\frac{2C}{\delta}\cstnorm{1}{\brho}\right)}{1-2C\delta\cstnorm{1}{\brho}}\,.
			\end{aligned}
		\end{equation*}
		Since $\ya \hookrightarrow\za$, together with the local existence result from Step 1, the non-degenerate lifespan provided by Step 2, and the conservation of energy, we obtain the global existence result for the magnetic Hartree--Fock equation.
	\end{proof}

    \section{Global classical solutions for non-compactly supported initial data}\label{Section:classical_ExistenceResults}
As shown in the works \cite{bostan2019asymptotic,knopf2018optimal}, for a given space-time-dependent magnetic field $B$ satisfying suitable assumptions and a compactly supported initial datum $f^{\mathrm{in}}\in C^1_c(\rr^3\times\rr^3)$, one can obtain a global classical solution of \eqref{equation: MagneticVlasovSys} by constructing a contraction for the electric field in a suitable Banach space. We now generalize this result to non-compactly supported initial data with exponential decay.

Suppose $f$ and $g$ are classical solutions of \eqref{equation: MagneticVlasovSys} with initial data $f^{\mathrm{in}}$ and $g^{\mathrm{in}}$, respectively. We also assume that $f^{\mathrm{in}}$ and $g^{\mathrm{in}}$ satisfy the conditions listed in Theorem \ref{Theorem:PropagationOfRegularity} with $m=2$. Let $T>0$ be a given time endpoint and $t\in [0,T]$. A direct computation shows that $f-g$ satisfies the following Cauchy problem:
\begin{equation}\label{equation: DifferenceEquation}
	\begin{dcases}
		\partial_t (f-g) + v\cdot\nabla_x (f-g) +v\wedge B\cdot \nabla_v (f-g) +E_f\cdot\nabla_v f-E_g\cdot\nabla_v g= 0, \\
		(f-g)(0,x,v) = f^{\mathrm{in}}(x,v)-g^{\mathrm{in}}(x,v).
	\end{dcases}
\end{equation}

Let $\bT_0$ denote the operator introduced in Lemma
\ref{lem:exact-commutator}, but with $E=0$. Then, for $\alpha$, $\beta\in\mathbb{N}^3$ with $|\alpha|+|\beta|\leq1$, we have
\begin{equation}
	\begin{aligned}
		\mathsf{T}_0D^{\alpha,\beta}h
		={}&
		-\sum_{i=1}^3\beta_i
		\left[
		D^{\alpha,\beta-e_i}\partial_{x_i}h
		+
		(e_i\wedge B)\cdot
		\nabla_vD^{\alpha,\beta-e_i}h
		\right]
		\notag\\
		&-
		\sum_{0<\lambda\leq\alpha}
		\binom{\alpha}{\lambda}
		(v\wedge\partial_x^\lambda B)\cdot
		\nabla_vD^{\alpha-\lambda,\beta}h
		\notag\\
		&-
		\sum_{\lambda\leq\alpha}
		\binom{\alpha}{\lambda}
		E_{\partial_x^\lambda f}\cdot
		\nabla_vD^{\alpha-\lambda,\beta}h
		\notag\\
		&-
		\sum_{\lambda\leq\alpha}
		\binom{\alpha}{\lambda}
		E_{\partial_x^\lambda h}\cdot
		\nabla_vD^{\alpha-\lambda,\beta}g.
	\end{aligned}
\end{equation}
Arguing as in the proof of propagation of regularity, we obtain the following differential inequality for $1\le p<\infty$:
\begin{equation}\label{ineq:differn-ineq}
	\begin{aligned}
		\frac{\d }{\d t}\norm{f-g}^p_{\mathcal{X}^{1,p}_{\mu(t),\lambda(t)}}
		&\le pC(f^{\mathrm{in}},g^{\mathrm{in}},T)
		\left(
		\norm{f-g}^p_{\mathcal{X}^{1,p}_{\mu(t),\lambda(t)}}
		+\|E_{f-g}\|^p_{L^\infty([0,t],W^{1,\infty})}
		\right)\,.
	\end{aligned}
\end{equation}
Here, the constant is given by
$$C(f^{\mathrm{in}},g^{\mathrm{in}},T):=
\norm{f}_{L^{\infty}([0,T];\mathcal{X}^{1,1}_{\mu(t),\lambda(t)}
\cap\mathcal{X}^{1,\infty}_{\mu(t),\lambda(t)})}
+\norm{g}_{L^{\infty}([0,T];\mathcal{X}^{1,1}_{\mu(t),\lambda(t)}
\cap\mathcal{X}^{1,\infty}_{\mu(t),\lambda(t)})}\,.$$
Since
\begin{equation*}
	\begin{aligned}
		\norm{E_{f-g}}_{L^\infty([0,t];W^{1,\infty})}
		&\le C\norm{\rho_{f-g}}_{L^\infty([0,t];W^{1,1}\cap W^{1,\infty})}\\
		&\le C\norm{f-g}_{L^{\infty}([0,t]; \mathcal{X}^{1,1}_{\mu(s), \lambda(s)}
		\cap\mathcal{X}^{1,\infty}_{\mu(s), \lambda(s)})},
	\end{aligned}
\end{equation*}
applying Gr\"{o}nwall's inequality to \eqref{ineq:differn-ineq}, taking the $p$-th root, and letting $p\rightarrow +\infty$, we obtain
\begin{equation}\label{Formula:error-estimate-of-initial-data}
	\begin{aligned}
		\norm{f-g}_{L^{\infty}([0,T]; \mathcal{X}^{1,1}_{\mu(s), \lambda(s)}
		\cap\mathcal{X}^{1,\infty}_{\mu(s), \lambda(s)})}
		\le
		\norm{f^{\mathrm{in}}-g^{\mathrm{in}}}_{\mathcal{X}^{1,1}_{\mu_0,\lambda_0}
		\cap\mathcal{X}^{1,\infty}_{\mu_0,\lambda_0}}
		\exp\left\{C(f^{\mathrm{in}},g^{\mathrm{in}},T)T\right\}\,.
	\end{aligned}
\end{equation}

With this result in mind, we perform the following approximation. Let $\chi$ be a smooth, radial, nonnegative bump function such that $\chi=1$ when $\snorm{x}\le 1$ and $\chi=0$ when $\snorm{x}\ge 2$. Consider the Cauchy problem \eqref{equation: MagneticVlasovSys} with initial data $f^{\mathrm{in}}$ satisfying the conditions listed in Theorem \ref{Theorem:PropagationOfRegularity}. For $k\in \mathbb{N}$, define
$f^{\mathrm{in}}_k:=f^{\mathrm{in}}(x,v)\chi\left(x/k\right)\chi\left(v/k\right)$.
According to the result in the appendix of \cite{bostan2019asymptotic}, which extends to $3D$, equation \eqref{equation: MagneticVlasovSys} has a classical solution $f_k$ with initial data $f^{\mathrm{in}}_k$. Since $f^{\mathrm{in}}$ satisfies the conditions of Theorem \ref{Theorem:PropagationOfRegularity} with $m=2$, for every $|\gamma|\le 1$ we have
$e^{\mu_0\langle x\rangle+\lambda_0\langle v\rangle}D^\gamma f^{\mathrm{in}}\in W^{1,1}(\Xi)\cap W^{1,\infty}(\Xi)$.
In particular, these functions are uniformly continuous and integrable, and hence vanish at infinity. Together with the fact that the derivatives of $\chi(x/k)\chi(v/k)$ are of order $k^{-1}$, this implies
\[
\norm{f^{\mathrm{in}}_k-f^{\mathrm{in}}}_{\mathcal{X}^{1,1}_{\mu_0,\lambda_0}
\cap\mathcal{X}^{1,\infty}_{\mu_0,\lambda_0}}\longrightarrow 0
\qquad\text{as }k\to\infty.
\]
Moreover, the norms of $f_k$ entering the constant in \eqref{Formula:error-estimate-of-initial-data} are bounded uniformly in $k$ by Theorem \ref{Theorem:PropagationOfRegularity}. The estimate \eqref{Formula:error-estimate-of-initial-data} therefore shows that $\left\{f_k\right\}_{k=1}^\infty$ is a Cauchy sequence in the topology induced by the norm
$\norm{\cdot}_{L^{\infty}([0,T],\mathcal{X}^{1,1}_{\mu(t),\lambda(t)} \cap\mathcal{X}^{1,\infty}_{\mu(t),\lambda(t)})}$. Consequently, there exists a limit function $f$ satisfying equation \eqref{equation: MagneticVlasovSys} on the given time interval $[0,T]$ with non-compactly supported initial data $f^{\mathrm{in}}$.

	\bibliographystyle{amsplain}
	\bibliography{references}

@article {LeopoldSaffirio2026,
    AUTHOR = {Leopold, Nikolai and Saffirio, Chiara},
     TITLE = {Derivation of the {Vlasov--Maxwell} system from the
              {Maxwell--Schr\"odinger} equations with extended charges},
   JOURNAL = {Forum Math. Sigma},
  FJOURNAL = {Forum of Mathematics, Sigma},
    VOLUME = {14},
      YEAR = {2026},
     PAGES = {Paper No. e17, 51},
      ISSN = {2050-5094},
       DOI = {10.1017/fms.2025.10162},
       URL = {https://doi.org/10.1017/fms.2025.10162},
}

@misc {Leopold2024,
    AUTHOR = {Leopold, Nikolai},
     TITLE = {Derivation of the {Maxwell--Schr\"odinger} and
              {Vlasov--Maxwell} equations from non-relativistic {QED}},
      YEAR = {2024},
      NOTE = {Preprint, arXiv:2411.07085},
       URL = {https://arxiv.org/abs/2411.07085},
}

@article{athanassoulis2011strong,
	author = {Athanassoulis, Agissilaos and Paul, Thierry and Pezzotti, Federica and Pulvirenti, Mario},
	doi = {10.4171/RLM/613},
	journal = {Rend. Lincei Mat. Appl.},
	number = {4},
	pages = {525--552},
	title = {Strong semiclassical approximation of {Wigner} functions for the {Hartree} dynamics},
	volume = {22},
	year = {2011}}

@article{benedikter2016mixed,
	author = {Benedikter, Niels and Jaksic, Vojkan and Porta, Marcello and Saffirio, Chiara and Schlein, Benjamin},
	title = {Mean-field evolution of fermionic mixed states},
	journal = {Comm. Pure Appl. Math.},
	volume = {69},
	number = {12},
	year = {2016},
	pages = {2250--2303},
	doi = {10.1002/cpa.21598},
}

@article{grillakis2011second,
	author = {Grillakis, M. G. and Machedon, M. and Margetis, D.},
	title = {Second-order corrections to mean field evolution of weakly interacting bosons. {II}},
	journal = {Adv. Math.},
	volume = {228},
	number = {3},
	year = {2011},
	pages = {1788--1815},
	doi = {10.1016/j.aim.2011.06.028},
}

@article{grillakis2013pair,
	author = {Grillakis, M. and Machedon, M.},
	title = {Pair excitations and the mean field approximation of interacting bosons. {I}},
	journal = {Comm. Math. Phys.},
	volume = {324},
	number = {2},
	year = {2013},
	pages = {601--636},
	doi = {10.1007/s00220-013-1818-7},
}

@article{chen2021combined,
	author = {Chen, Li and Lee, Jinyeop and Liew, Matthew},
	title = {Combined mean-field and semiclassical limits of large fermionic systems},
	journal = {J. Stat. Phys.},
	volume = {182},
	number = {2},
	year = {2021},
	pages = {24},
	doi = {10.1007/s10955-021-02700-w},
}

@article{chen2022convergence,
	author = {Chen, Li and Lee, Jinyeop and Liew, Matthew},
	title = {Convergence towards the {Vlasov--Poisson} equation from the {$N$}-fermionic {Schr\"odinger} equation},
	journal = {Ann. Henri Poincar{\'e}},
	volume = {23},
	year = {2022},
	pages = {555--593},
	doi = {10.1007/s00023-021-01103-7},
}

@article{saffirio2020inverse,
	author = {Saffirio, Chiara},
	doi = {10.1007/s00220-019-03397-5},
	journal = {Comm. Math. Phys.},
	pages = {571--619},
	title = {Semiclassical limit to the {Vlasov} equation with inverse power law potentials},
	volume = {373},
	year = {2020}}

@article{saffirio2020hartree,
	author = {Saffirio, Chiara},
	doi = {10.1137/19M1299529},
	journal = {SIAM J. Math. Anal.},
	number = {6},
	pages = {5533--5553},
	title = {From the {Hartree} equation to the {Vlasov--Poisson} system: strong convergence for a class of mixed states},
	volume = {52},
	year = {2020}}

@article{chong2022global,
	author = {Chong, Jacky J. and Lafleche, Laurent and Saffirio, Chiara},
	doi = {10.1063/5.0089741},
	journal = {J. Math. Phys.},
	number = {8},
	pages = {081904},
	title = {Global-in-time semiclassical regularity for the {Hartree--Fock} equation},
	volume = {63},
	year = {2022}}

@article{chong2023l2,
	author = {Chong, Jacky J. and Lafleche, Laurent and Saffirio, Chiara},
	doi = {10.5802/jep.230},
	journal = {J. \'{E}c. polytech. Math.},
	pages = {703--726},
	title = {On the {$L^2$} rate of convergence in the limit from the {Hartree} to the {Vlasov--Poisson} equation},
	volume = {10},
	year = {2023}}

@article{frenod1998homogenization,
	author = {Fr{\'e}nod, Emmanuel and Sonnendr{\"u}cker, Eric},
	doi = {10.3233/ASY-1998-298},
	journal = {Asymptot. Anal.},
	number = {3-4},
	pages = {193--213},
	title = {Homogenization of the {Vlasov} equation and of the {Vlasov--Poisson} system with a strong external magnetic field},
	volume = {18},
	year = {1998}}

@article{diperna1989vlasovmaxwell,
	author = {DiPerna, Ronald J. and Lions, Pierre-Louis},
	doi = {10.1002/cpa.3160420603},
	journal = {Comm. Pure Appl. Math.},
	number = {6},
	pages = {729--757},
	title = {Global weak solutions of {Vlasov--Maxwell} systems},
	volume = {42},
	year = {1989}}

@misc{benporat2026magnetized,
	author = {Ben Porat, Immanuel},
	note = {Preprint, arXiv:2608.18476},
	title = {From magnetized {Coulombic} quantum dynamics to magnetized fluids},
	year = {2026}}

@article{rege2025stability,
	author = {Rege, Alexandre},
	doi = {10.1016/j.jde.2025.01.051},
	journal = {J. Differential Equations},
	pages = {763--788},
	title = {Stability estimates for magnetized {Vlasov} equations},
	volume = {425},
	year = {2025}}

@misc{benedikter2025magnetic,
	author = {Benedikter, Niels and Boccato, Chiara and Monaco, Domenico and Nguyen, Ngoc Nhi},
	note = {Preprint, arXiv:2503.16001},
	title = {Derivation of {Hartree--Fock} dynamics and semiclassical commutator estimates for fermions in a magnetic field},
	year = {2025}}

@article{ferrero2024effective,
	author = {Ferrero, Margherita and Monaco, Domenico},
	doi = {10.46298/ocnmp.13820},
	journal = {Open Commun. Nonlinear Math. Phys.},
	pages = {157--187},
	title = {Effective quantum dynamics for magnetic fermions},
	volume = {4},
	year = {2024}}

@article{luhrmann2012mean,
	author = {L{\"u}hrmann, Jonas},
	doi = {10.1063/1.3687024},
	journal = {J. Math. Phys.},
	number = {2},
	pages = {022105},
	title = {Mean-field quantum dynamics with magnetic fields},
	volume = {53},
	year = {2012}}

@article{lein2010two,
	author = {Lein, Max},
	doi = {10.1063/1.3499660},
	journal = {J. Math. Phys.},
	number = {12},
	pages = {123519},
	title = {Two-parameter asymptotics in magnetic {Weyl} calculus},
	volume = {51},
	year = {2010}}

@article{mantoiu2004magnetic,
	author = {M{\u{a}}ntoiu, Marius and Purice, Radu},
	doi = {10.1063/1.1668334},
	journal = {J. Math. Phys.},
	number = {4},
	pages = {1394--1417},
	title = {The magnetic {Weyl} calculus},
	volume = {45},
	year = {2004}}

@article{moller2025pauli,
	author = {M{\"o}ller, Jakob},
	doi = {10.1080/03605302.2024.2439358},
	journal = {Comm. Partial Differential Equations},
	number = {1-2},
	pages = {130--161},
	title = {The {Pauli--Poisson} equation and its semiclassical limit},
	volume = {50},
	year = {2025}}

@article{benporat2024magnetic,
	author = {Ben Porat, Immanuel},
	doi = {10.1137/22M1528562},
	journal = {SIAM J. Math. Anal.},
	number = {1},
	pages = {955--992},
	title = {The magnetic {Liouville} equation as a semiclassical limit},
	volume = {56},
	year = {2024}}

@article{colliander2004global,
	author = {Colliander, James and Keel, Markus and Staffilani, Gigliola and Takaoka, Hideo and Tao, Terence},
	doi = {10.1002/cpa.20029},
	journal = {Comm. Pure Appl. Math.},
	number = {8},
	pages = {987--1014},
	title = {Global existence and scattering for rough solutions of a nonlinear {Schr\"odinger} equation on $\mathbb{R}^3$},
	volume = {57},
	year = {2004}}

@article{colliander2009tensor,
	author = {Colliander, James and Grillakis, Manoussos and Tzirakis, Nikolaos},
	doi = {10.1002/cpa.20278},
	journal = {Comm. Pure Appl. Math.},
	number = {7},
	pages = {920--968},
	title = {Tensor products and correlation estimates with applications to nonlinear {Schr\"odinger} equations},
	volume = {62},
	year = {2009}}

@misc{porat2025propagation,
	author = {Ben Porat, Immanuel and Gagnebin, Antoine and Iacobelli, Mikaela and Junn{\'e}, Jonathan},
	note = {Preprint, arXiv:2510.22753},
	title = {Propagation of velocity moments for the magnetized {Vlasov--Poisson} system with space-time dependent magnetic fields},
	year = {2025}}

@article{rege2023propagation,
	author = {Rege, Alexandre},
	journal = {Comm. Partial Differential Equations},
	number = {3},
	pages = {386--414},
	title = {Propagation of velocity moments and uniqueness for the magnetized {Vlasov--Poisson} system},
	volume = {48},
	year = {2023}}

@article{iftimie2007magnetic,
	author = {Iftimie, Viorel and M{\u{a}}ntoiu, Marius and Purice, Radu},
	journal = {Publ. Res. Inst. Math. Sci.},
	number = {3},
	pages = {585--623},
	title = {Magnetic pseudodifferential operators},
	volume = {43},
	year = {2007}}

@article{amour2013semiclassical,
	author = {Amour, Laurent and Khodja, Mohamed and Nourrigat, Jean},
	journal = {Anal. PDE},
	number = {7},
	pages = {1649--1674},
	title = {The semiclassical limit of the time dependent {Hartree--Fock} equation: the {Weyl} symbol of the solution},
	volume = {6},
	year = {2013}}

@article{athmouni2018schatten,
	author = {Athmouni, Nassim and Purice, Radu},
	journal = {Comm. Partial Differential Equations},
	number = {5},
	pages = {733--749},
	title = {A {Schatten--von Neumann} class criterion for the magnetic {Weyl} calculus},
	volume = {43},
	year = {2018}}

@unpublished{laflechesemiclassical,
	author = {Lafleche, Laurent},
	note = {Master 2 course, {\'E}cole Normale Sup{\'e}rieure de Lyon},
	title = {Semiclassical dynamics},
	year = {2026}}

@unpublished{golse2013mean,
	author = {Golse, Fran{\c{c}}ois},
	note = {{\'E}cole Polytechnique},
	title = {Mean field kinetic equations},
	year = {2013}}

@article{lafleche2023strong,
	author = {Lafleche, Laurent and Saffirio, Chiara},
	journal = {Anal. PDE},
	number = {4},
	pages = {891--926},
	title = {Strong semiclassical limits from {Hartree} and {Hartree--Fock} to {Vlasov--Poisson} equations},
	volume = {16},
	year = {2023}}

@book{simon2005trace,
	address = {Providence, RI},
	author = {Simon, Barry},
	edition = {Second},
	publisher = {American Mathematical Society},
	series = {Mathematical Surveys and Monographs},
	title = {Trace ideals and their applications},
	volume = {120},
	year = {2005}}

@article{lions1991propagation,
	author = {Lions, Pierre-Louis and Perthame, Beno{\^\i}t},
	journal = {Invent. Math.},
	number = {1},
	pages = {415--430},
	title = {Propagation of moments and regularity for the 3-dimensional {Vlasov--Poisson} system},
	volume = {105},
	year = {1991}}

@article{boulkhemair1999l2,
	author = {Boulkhemair, Abdesslam},
	journal = {J. Funct. Anal.},
	number = {1},
	pages = {173--204},
	title = {{$L^2$ estimates for {Weyl} quantization}},
	volume = {165},
	year = {1999}}

@article{calderon1972class,
	author = {Calder{\'o}n, Alberto P. and Vaillancourt, R{\'e}mi},
	journal = {Proc. Natl. Acad. Sci. USA},
	number = {5},
	pages = {1185--1187},
	title = {A class of bounded pseudo-differential operators},
	volume = {69},
	year = {1972}}

@book{bahouri2011fourier,
	address = {Heidelberg},
	author = {Bahouri, Hajer and Chemin, Jean-Yves and Danchin, Rapha\"{e}l},
	publisher = {Springer},
	series = {Grundlehren der mathematischen Wissenschaften},
	title = {Fourier analysis and nonlinear partial differential equations},
	volume = {343},
	year = {2011}}

@article{pfaffelmoser1992global,
	author = {Pfaffelmoser, Klaus},
	journal = {J. Differential Equations},
	number = {2},
	pages = {281--303},
	title = {Global classical solutions of the {Vlasov--Poisson} system in three dimensions for general initial data},
	volume = {95},
	year = {1992}}

@article{lafleche2024quantum,
	author = {Lafleche, Laurent},
	journal = {J. Funct. Anal.},
	number = {10},
	pages = {110400},
	title = {On quantum {Sobolev} inequalities},
	volume = {286},
	year = {2024}}

@article{knopf2018optimal,
	author = {Knopf, Patrik},
	journal = {Calc. Var. Partial Differential Equations},
	number = {5},
	pages = {134},
	title = {Optimal control of a {Vlasov--Poisson} plasma by an external magnetic field},
	volume = {57},
	year = {2018}}

@article{narnhofer1981vlasov,
	author = {Narnhofer, Heide and Sewell, Geoffrey L.},
	journal = {Comm. Math. Phys.},
	number = {1},
	pages = {9--24},
	title = {{Vlasov} hydrodynamics of a quantum mechanical model},
	volume = {79},
	year = {1981}}

@article{spohn1981vlasov,
	author = {Spohn, Herbert},
	journal = {Math. Methods Appl. Sci.},
	number = {1},
	pages = {445--455},
	title = {On the {Vlasov} hierarchy},
	volume = {3},
	year = {1981}}

@article{lions1993mesures,
	author = {Lions, Pierre-Louis and Paul, Thierry},
	journal = {Rev. Mat. Iberoam.},
	number = {3},
	pages = {553--618},
	title = {Sur les mesures de {Wigner}},
	volume = {9},
	year = {1993}}

@article{markowich1993classical,
	author = {Markowich, Peter A. and Mauser, Norbert J.},
	journal = {Math. Models Methods Appl. Sci.},
	number = {1},
	pages = {109--124},
	title = {The classical limit of a self-consistent quantum-{Vlasov} equation in {3D}},
	volume = {3},
	year = {1993}}

@article{benedikter2016hartree,
	author = {Benedikter, Niels and Porta, Marcello and Saffirio, Chiara and Schlein, Benjamin},
	journal = {Arch. Ration. Mech. Anal.},
	number = {1},
	pages = {273--334},
	title = {From the {Hartree} dynamics to the {Vlasov} equation},
	volume = {221},
	year = {2016}}

@article{lafleche2019propagation,
	author = {Lafleche, Laurent},
	journal = {J. Stat. Phys.},
	number = {1},
	pages = {20--60},
	title = {Propagation of moments and semiclassical limit from {Hartree} to {Vlasov} equation},
	volume = {177},
	year = {2019}}

@article{lafleche2021global,
	author = {Lafleche, Laurent},
	journal = {Ann. Inst. H. Poincar{\'e} C Anal. Non Lin{\'e}aire},
	number = {6},
	pages = {1739--1762},
	title = {Global semiclassical limit from {Hartree} to {Vlasov} equation for concentrated initial data},
	volume = {38},
	year = {2021}}

@article{chong2024many,
	author = {Chong, Jacky J. and Lafleche, Laurent and Saffirio, Chiara},
	journal = {J. Eur. Math. Soc. (JEMS)},
	number = {12},
	pages = {4923--5007},
	title = {From many-body quantum dynamics to the {Hartree--Fock} and {Vlasov} equations with singular potentials},
	volume = {26},
	year = {2024}}

@article{golse1999vlasov,
	author = {Golse, Fran{\c{c}}ois and Saint-Raymond, Laure},
	journal = {J. Math. Pures Appl. (9)},
	number = {8},
	pages = {791--817},
	title = {The {Vlasov--Poisson} system with strong magnetic field},
	volume = {78},
	year = {1999}}

@article{golse2003vlasov,
	author = {Golse, Fran{\c{c}}ois and Saint-Raymond, Laure},
	journal = {Math. Models Methods Appl. Sci.},
	number = {5},
	pages = {661--714},
	title = {The {Vlasov--Poisson} system with strong magnetic field in quasineutral regime},
	volume = {13},
	year = {2003}}

@article{charles2021magnetized,
	author = {Charles, Fr{\'e}d{\'e}rique and Despr{\'e}s, Bruno and Rege, Alexandre and Weder, Ricardo},
	journal = {J. Stat. Phys.},
	number = {2},
	pages = {23},
	title = {The magnetized {Vlasov--Amp\`ere} system and the {Bernstein--Landau} paradox},
	volume = {183},
	year = {2021}}

@article{rege2021vlasov,
	author = {Rege, Alexandre},
	journal = {SIAM J. Math. Anal.},
	number = {2},
	pages = {2452--2475},
	title = {The {Vlasov--Poisson} system with a uniform magnetic field: propagation of moments and regularity},
	volume = {53},
	year = {2021}}

@book{benedikter2016effective,
	address = {Cham},
	author = {Benedikter, Niels and Porta, Marcello and Schlein, Benjamin},
	publisher = {Springer},
	series = {SpringerBriefs in Mathematical Physics},
	title = {Effective evolution equations from quantum dynamics},
	volume = {7},
	year = {2016}}

@article{bardos2002derivation,
	author = {Bardos, Claude and Erd{\"o}s, L{\'a}szl{\'o} and Golse, Fran{\c{c}}ois and Mauser, Norbert and Yau, Horng-Tzer},
	journal = {C. R. Math. Acad. Sci. Paris},
	number = {6},
	pages = {515--520},
	title = {Derivation of the {Schr\"odinger}--{Poisson} equation from the quantum $\mathbf{N}$-body problem},
	volume = {334},
	year = {2002}}

@article{erdos2001derivation,
	author = {Erd{\"o}s, L{\'a}szl{\'o} and Yau, Horng-Tzer},
	journal = {Adv. Theor. Math. Phys.},
	number = {6},
	pages = {1169--1205},
	title = {Derivation of the nonlinear {Schr{\"o}dinger} equation from a many body {Coulomb} system},
	volume = {5},
	year = {2001}}

@article{grillakis2010second,
	author = {Grillakis, Manoussos G. and Machedon, Matei and Margetis, Dionisios},
	journal = {Comm. Math. Phys.},
	number = {1},
	pages = {273--301},
	title = {Second-order corrections to mean field evolution of weakly interacting bosons. {I}},
	volume = {294},
	year = {2010}}

@article{rodnianski2009quantum,
	author = {Rodnianski, Igor and Schlein, Benjamin},
	journal = {Comm. Math. Phys.},
	number = {1},
	pages = {31--61},
	title = {Quantum fluctuations and rate of convergence towards mean field dynamics},
	volume = {291},
	year = {2009}}

@article{bach2016kinetic,
	author = {Bach, Volker and Breteaux, S{\'e}bastien and Petrat, S{\"o}ren and Pickl, Peter and Tzaneteas, Tim},
	journal = {J. Math. Pures Appl. (9)},
	number = {1},
	pages = {1--30},
	title = {Kinetic energy estimates for the accuracy of the time-dependent {Hartree--Fock} approximation with {Coulomb} interaction},
	volume = {105},
	year = {2016}}

@article{benedikter2014mean,
	author = {Benedikter, Niels and Porta, Marcello and Schlein, Benjamin},
	journal = {Comm. Math. Phys.},
	number = {3},
	pages = {1087--1131},
	title = {Mean--field evolution of fermionic systems},
	volume = {331},
	year = {2014}}

@article{petrat2016new,
	author = {Petrat, S{\"o}ren and Pickl, Peter},
	journal = {Math. Phys. Anal. Geom.},
	number = {1},
	pages = {3},
	title = {A new method and a new scaling for deriving fermionic mean-field dynamics},
	volume = {19},
	year = {2016}}

@inproceedings{saffirio2017mean,
	author = {Saffirio, Chiara},
	booktitle = {Workshop on Macroscopic Limits of Quantum Systems},
	pages = {81--99},
	publisher = {Springer},
	title = {Mean-field evolution of fermions with singular interaction},
	year = {2017}}

@article{michelangeli2015global,
	author = {Michelangeli, Alessandro},
	journal = {Nonlinearity},
	number = {8},
	pages = {2743--2765},
	title = {Global well-posedness of the magnetic {Hartree} equation with non-{Strichartz} external fields},
	volume = {28},
	year = {2015}}

@article{dong2021hartree,
	author = {Dong, Xin},
	journal = {Lett. Math. Phys.},
	number = {4},
	pages = {101},
	title = {The {Hartree} equation with a constant magnetic field: well-posedness theory},
	volume = {111},
	year = {2021}}

@book{Lieb2001-yq,
	address = {Providence, RI},
	author = {Lieb, Elliott H. and Loss, Michael},
	edition = {Second},
	publisher = {American Mathematical Society},
	series = {Graduate Studies in Mathematics},
	title = {Analysis},
	volume = {14},
	year = {2001}}

@article{benedikter2018dirac,
	author = {Benedikter, Niels and Sok, J{\'e}r{\'e}my and Solovej, Jan Philip},
	journal = {Ann. Henri Poincar{\'e}},
	number = {4},
	pages = {1167--1214},
	title = {The {Dirac--Frenkel} principle for reduced density matrices, and the {Bogoliubov--de Gennes} equations},
	volume = {19},
	year = {2018}}

@article{frank2008hardy,
	author = {Frank, Rupert and Lieb, Elliott H. and Seiringer, Robert},
	journal = {J. Amer. Math. Soc.},
	number = {4},
	pages = {925--950},
	title = {{Hardy--Lieb--Thirring} inequalities for fractional {Schr{\"o}dinger} operators},
	volume = {21},
	year = {2008}}

@article{serre2018divergence,
	author = {Serre, Denis},
	journal = {Ann. Inst. H. Poincar{\'e} C Anal. Non Lin{\'e}aire},
	number = {5},
	pages = {1209--1234},
	title = {Divergence-free positive symmetric tensors and fluid dynamics},
	volume = {35},
	year = {2018}}

@misc{wang2026critical,
	author = {Wang, Zhaopeng},
	note = {Preprint, arXiv:2607.26840},
	title = {A critical density estimate for the {Vlasov--Poisson} system: energy conservation and moment propagation},
	year = {2026}}

@article{bostan2019asymptotic,
	author = {Bostan, Mihai},
	journal = {SIAM J. Math. Anal.},
	number = {3},
	pages = {2713--2747},
	title = {Asymptotic behavior for the {Vlasov--Poisson} equations with strong external magnetic field: straight magnetic field lines},
	volume = {51},
	year = {2019}}
	
\end{document}